\documentclass[11pt, reqno]{amsart}
\usepackage{fullpage}
\usepackage{amssymb,amsmath,mathrsfs,amsthm}
\usepackage{amsfonts}
\usepackage{enumitem}
\usepackage{mathtools}
\usepackage{caption}
\usepackage{esint,comment}
\usepackage[dvips]{graphicx}
\usepackage{xcolor,scalerel} 


\newtheorem{thm}{Theorem}[section]
\newtheorem{cor}[thm]{Corollary}
\newtheorem{corollary}[thm]{Corollary} 
\newtheorem{lem}[thm]{Lemma}
 \newtheorem{lemma}[thm]{Lemma}
\newtheorem{prop}[thm]{Proposition}

\newtheorem{que}{Question}

\theoremstyle{definition}
\newtheorem{defn}[thm]{Definition}
\newtheorem{rem}[thm]{Remark}
\numberwithin{equation}{section}

\newcommand{\norm}[1]{\Vert#1\Vert}
\newcommand{\Norm}[1]{\left\Vert#1\right\Vert}
\newcommand{\abs}[1]{\left|#1\right|}
\newcommand{\na}{\nabla}
\newcommand{\pa}{\partial}
\newcommand{\lec}{\lesssim}
\newcommand{\bs}{\begin{split}}
\newcommand{\essss}{\end{split}}
\newcommand{\td}{\tilde}

\renewcommand{\div}{\operatorname{div}}
\newcommand{\bna}{\overline\na} 
\newcommand{\eqnb}{\begin{equation}}
\newcommand{\eqne}{\end{equation}}

\newcommand\al{\alpha}
\newcommand\be{\beta}
\newcommand\de{\delta}
\newcommand\De{\Delta}
\newcommand{\ga}{\gamma}
\newcommand\Ga{\Gamma}
\newcommand\ep{\epsilon}

\newcommand\ka{\kappa}
\newcommand{\la}{\lambda}
\newcommand{\La}{\Lambda}
\newcommand\ph{\varphi}

\newcommand{\p}{\partial}
\renewcommand{\th}{\theta}

\newcommand{\om}{\omega}
\newcommand{\Om}{\Omega}

\newcommand{\R}{\mathbb{R}}
\newcommand{\RR}{\mathbb{R}}
\newcommand{\ZZ}{\mathbb{Z}}
\newcommand{\Z}{\mathbb{Z}}

\newcommand{\cM}{\mathcal{M}}

\renewcommand{\d}{\mathrm{d}}

\newcommand{\supp}{\operatorname{supp}}

\newcommand{\I}{\text{Id}}

\newcommand{\loc}{\text{loc}}
\newcommand{\tail}{\text{tail}}

\usepackage[usestackEOL]{stackengine}
\def\dint{\,\ThisStyle{\ensurestackMath{%
  \stackinset{c}{.2\LMpt}{c}{.5\LMpt}{\SavedStyle-}{\SavedStyle\phantom{\int}}}%
  \setbox0=\hbox{$\SavedStyle\int\,$}\kern-\wd0}\int}

\stackMath
\newcommand\reallywidehat[1]{%
\savestack{\tmpbox}{\stretchto{%
  \scaleto{%
    \scalerel*[\widthof{\ensuremath{#1}}]{\kern.1pt\mathchar"0362\kern.1pt}%
    {\rule{0ex}{\textheight}}
  }{\textheight}%
}{2.4ex}}%
\stackon[-6.9pt]{#1}{\tmpbox}%
}

\begin{document}

\title{Local regularity of weak solutions of the hypodissipative Navier-Stokes equations}

\author{Hyunju Kwon, Wojciech S. O\.za\'{n}ski}
\address{
\begin{minipage}{\linewidth}
Hyunju Kwon, \\School of Mathematics, Institute for Advanced Study, Princeton, NJ 08540 \\\texttt{hkwon@ias.edu}\\  \\
Wojciech S. O\.za\'{n}ski\\
Department of Mathematics, University of Southern California, Los Angeles, CA 90089\\
\texttt{ozanski@usc.edu}
\end{minipage}
} 

\begin{abstract} 
We consider the 3D incompressible hypodissipative Navier-Stokes equations, when the dissipation is given as a fractional Laplacian $(-\De)^s$ for $s\in (\frac34,1)$, and we provide a new bootstrapping scheme that makes it possible to analyse weak solutions locally in space-time. This includes several homogeneous Kato-Ponce type commutator estimates which we localize in space, and which seems applicable to other parabolic systems with fractional dissipation. We also provide a new estimate on the pressure, $\|(-\Delta)^s p \|_{\mathcal{H}^1}\lec \| (-\Delta )^{\frac s2} u \|^2_{L^2}$. We apply our main result to prove that any suitable weak solution $u$ satisfies $\na^n u \in L^{p,\infty }_\loc(\R^3\times(0,\infty))$ for $p=\frac{2(3s-1)}{n+2s-1}$, $n=1,2$. As a corollary of our local regularity theorem, we improve the partial regularity result of Tang-Yu [\emph{Comm. Math. Phys., 334(30), 2015, pp. 1455--1482}], and obtain an estimate on the box-counting dimension of the singular set $S$, $d_B(S\cap \{t\geq t_0 \} )\leq \frac13 (15-2s-8s^2) $ for every $t_0>0$. 
\end{abstract}
\maketitle

\section{Introduction}\label{sec_intro}
We consider three-dimensional incompressible Navier-Stokes equations with fractional Laplacian dissipation,
\begin{align}\label{fNS}
\begin{cases}
\pa_t u + (-\De)^{s} u+ (u\cdot \na) u+\na p  = 0, \\
\div u = 0,
\end{cases}
\end{align}
for $s>0$, in the whole space $\R^3\times (0,T)$, $T\in (0,\infty )$. The system is supplemented with initial data $u|_{t=0}= u_0\in L^2(\R^3)$ that is divergence free. The fractional Laplacian is defined as the Fourier multiplier with symbol $|\xi |^2$, \[
\reallywidehat{(-\De)^s f} (\xi)= |\xi|^{2s}\widehat{f}(\xi) 
\quad \forall f\in \mathcal{D}'(\R^3). 
\]

From the physical point of view, this model, for $0<s<1$, describes fluids with internal friction in \cite{mercado} and has also been obtained from a stochastic Lagrangian particle approach by \cite{zhang}. From the analytical point of view, \eqref{fNS} has special importance as a generalization of the classical Navier-Stokes equations (i.e. when $s=1$). Lions \cite{Lion69} first studied \eqref{fNS} and has established the existence and uniqueness of global-in-time classical solution when $s\geq \frac 54$, satisfying the energy inequality,
\eqnb\label{EI}
\int_{\RR^3} |u(t)|^2  \d x + 2\int_0^t \int_{\RR^3} | \La^s u |^2 \d x \, \d \tau \leq \int_{\RR^3} |u_0|^2  \d x , \quad \forall t\geq 0,
\eqne
where $\Lambda \equiv (-\Delta )^{1/2}$. In the case of $s<\frac 54$, the existence of global-in-time classical solution remains open. In particular, this question for the classical Navier-Stokes equations ($s=1$) remains one of the Millennium Problems. One of the important developments of the regularity theory is the celebrated $\epsilon$-regularity theory of Caffarelli-Kohn-Nirenberg \cite{CKN}, who showed that the $1$-dimensional parabolic Hausdorff measure of the singular set (i.e. the set of point $(x,t)$ such that $u$ is unbounded in any neighbourhood of $(x,t)$) vanishes for every suitable weak solution (see \cite{CKN} for the definition). Recently, Tang-Yu \cite{TY15} have extended this result to the hypodissipative case $\frac 34<s<1$, by showing that the $(5-4s)$-dimensional Hausdorff measure of the singular set vanishes for every suitable weak solution\footnote{see Definition~\ref{def_suitable}}. They also showed existence of a suitable weak solution for given divergence-free initial data $u_0\in L^2$. In the case of hyperdissipation $1<s\leq\frac54$, a similar result has recently been obtained by Colombo-De Lellis-Massaccesi \cite{colombo_dl_m} (see also \cite{katz_pavlovic} and \cite{pr_not_suitable}). We note that these results cover, at most, H\"older continuity of solutions outside of the singular set, and any regularity aspects of derivatives of suitable weak solutions have remained an open question for $s\ne 1$, $s<5/4$. 

The regularity of derivatives of the Navier-Stokes equations with fractional dissipation is an interesting open problem. In the case of classical Navier-Stokes equations one can deduce boundedness of higher derivatives using the classical procedure, which we now briefly sketch. (It is described in detail in Section 13 and Section D.3 in \cite{NSE_book}, for example.)\\ Consider the vorticity $\om = \mathrm{curl}\,  u$ and let us focus on the vorticity formulation $\p_t \omega - \Delta \omega = (\omega \cdot \nabla ) u - (u\cdot \nabla )\omega $. Set $W\coloneqq \omega \phi$, where $\phi\in {C_c^\infty} (\RR^3 \times (0,\infty))$ is a cutoff function. Then $W$ satisfies an equation which is, roughly speaking, of the form
\[
\p_t W -\Delta W = \pa (W u) + Wu + \p W + W
\]
(where $\p$ denotes any spatial partial derivative). Considering only the leading order term ``$\p (Wu )$'' on the right-hand side, and applying standard parabolic regularity estimates gives
\[
\| W \|_{L^r} \lec \| W u \|_{L^a} \lec \| W \|_{L^m} \| u \|_{L^q}
\]
for $5/a\leq 5/r+1$, where $L^r\equiv L^r (\RR^3 \times (0, \infty ))$, and we also applied H\"older's inequality with $1/a=1/m+1/q$. This gives the condition
\[
5\left( \frac{1}{m} -\frac{1}{r} \right) < 1-\frac{5}{q},
\]
from which it is clear that if $u\in L^q$ for some $q>5$ then the right hand side of the above inequality is strictly positive, and so one can choose $r>m$, which improves local regularity of $\om$. Therefore, using  the ``initial regularity'' $\omega \in L^2$ obtained from the energy inequality, one can use a bootstrapping argument (with decreasing cutoff functions $\phi$) together with the Biot-Savart estimates to obtain local boundedness of all spatial derivatives of $u$. 


Considering the case $s\ne 1$, it is clear that the hypodissipative case ($s<1$) is drastically more complicated as we do not necessarily have vorticity $\omega \in L^2$. Indeed the energy inequality \eqref{EI} gives only $\La^{s} u \in L^2$, and so it is not even clear that $\omega = \mathrm{curl}\,u$ is a well-defined quantity. Therefore one cannot use the vorticity equation to bootstrap regularity. On the other hand using the equations \eqref{fNS} directly becomes much more difficult as one needs to take into account both the nonlocality of the fractional Laplacian $(-\Delta )^s$ and the nonlocality of the pressure function $p$. This gives rise to two important open questions:

\begin{que} If a Leray-Hopf weak solution $u$ to \eqref{fNS} is bounded on some cylinder $Q_2$ then are the derivatives of $u$ bounded, on some smaller cylinder $Q_1$?\end{que}
\begin{que} Do derivatives of solutions $u$ to \eqref{fNS} admit any a priori estimates? \end{que}

In this work we provide positive answers to both questions in the hypodissipative case
 \[
\frac 34<s<1.
\]
Namely, we answer the first question in our first result (Theorem~\ref{thm_bootstrapping_intro} below), and we show (in Theorem~\ref{thm_main} below) that derivatives of $u$ admit local estimates in weak Lebesgue spaces with an optimal exponent for any suitable weak solution $u$ (see Definition~\ref{def_suitable}).

\begin{thm}\label{thm_bootstrapping_intro}
Suppose that a Leray-Hopf weak solution $u$ to \eqref{fNS} for $\frac  34<s<1$ satisfies  
\eqnb\label{local_terms_intro}\begin{split}
\| u \|_{L^\infty_{t,x}(Q_1)}&+ \| u\|_{ L^2_tW^{s,2}_x(Q_2)}+\|p\|_{ L^1_{t,x}(Q_1)}+ \| \na p \|_{L^1_{t,x}(Q_1)} \\
&+ \| \mathcal{M} (\La^s u )\|_{L^2 (Q_2)}+\|\mathcal{M} |\La^s u |^{\frac 2{1+\de}} \|_{ L^{1+\de} (Q_2)}+ \| \mathscr{M}_4 |\La^{2s-1}\na p | \|_{ L^1 (Q_2)}\leq c<\infty
\end{split}
\eqne
for $\de= \frac {2s}{6-s}$. Then the velocity $u$ satisfies
\[
\sup_{Q_1} |u(x,t)| + |\na u(x,t) | + |\nabla^2 u(x,t) | \leq C_0
\]
for some constant $C_0=C_0(c,s)>0$. 
\end{thm}
Here $p$ stands for the pressure function of a weak solution $u$ (see Definition~\ref{def_weak} below), $\mathcal{M}$ denotes the Hardy-Littlewood maximal function and $\mathscr{M}_4$ denotes the grand maximal function of order $4$. We give the precise definitions in Section~\ref{sec_hardy_and_grand_max} below. We note that the assumptions of Theorem~\ref{thm_bootstrapping_intro} imply that $\sup_{Q_1} |\na^k u |\leq C_k (c,s)$ for all $k\geq 0$, where $C_k=C_k(c,s)>0$, which can be shown using the same method (see Section~\ref{sec_bootstrapping}). For simplicity, we restrict ourselves to $k\leq 2$.

In order to prove Theorem~\ref{thm_bootstrapping_intro} we develop a new bootstrapping scheme which provides a robust method of dealing with all nonlocalities. In fact, introducing an arbitrary space-time cut-off $\phi$ one needs to estimate a number of commutators of the form $[(-\De )^s , \phi]v\coloneqq (-\De)^s(v\phi) - \phi (-\De)^s v$. Here, $v$ can be $u$, $\na u$, $\La^\ga u$ (for $\ga \in (s,1)$), $u\cdot \nabla u$ or $\na p$. In contrast to the usual Kato-Ponce type estimates (see \cite{kato_ponce, grafakos_oh, li}), such commutators need to be localized in the sense that the right hand sides can involve only local information of some controlled quantities, and they all appear to be new. Each instance of $v=u,\na u,\La^\ga u,u\cdot \nabla u, \na p$ brings new challenges to our analysis, which we discuss in more detail in Section~\ref{sec_commutator_est}. 

A remarkable property of our commutator estimates, presented in Lemmas~\ref{lem_1.10}--\ref{lem_comm.p}, is that merely local information of $u$, $p$,  $\mathcal{M}(\La^s u)$ and $\mathscr{M}_4 (\La^{2s-1}\na p)$ suffices to control all the tail terms related to the fractional Laplacian $(-\Delta)^s$. In this sense the commutators are well-suited to the local regularity result of Theorem~\ref{thm_bootstrapping_intro} above.
We discuss the main new ideas (i.e. ``tricks'') of this control of the tail terms in Lemma~\ref{lem_tricks}, for the reader's convenience. We explain the reason for the use of the grand maximal function (as opposed to some simpler notion of a maximal function) below (above Corollary~\ref{cor_single_scale_local_reg_into}). 

We note that the theorem above is still valid with the grand maximal function $\mathscr{M}_4$ replaced by the Hardy-Littlewood maximal function $\mathcal{M}$ and with $\delta =0 $, but in any of these simplifications the boundedness of the last two terms appearing in \eqref{local_terms_intro} cannot be guaranteed in general, as $\mathcal{M}f\not \in L^1$ for any $f\in L^1$, $f\ne 0$. In this sense Theorem~\ref{thm_bootstrapping_intro} gives an optimal answer to Question 1 above.\\

The boundedness of derivatives of $u$ requires some minimal local control of $p$, $\nabla p$, as well as local $L^1$ control of $\mathscr{M}_4(\Lambda^{2s-1} \na p)$, as stated in Theorem~\ref{thm_bootstrapping_intro}. We show that these quantities are finite for any weak solution since the pressure is given via the singular integral \eqref{pressure_formula}. While for the classical Navier-Stokes equations one has a classical global estimate for $\na^2 p$, based on the fact that $-\De p =\pa_i u_j \pa_j u_i $ as well as the Coifman-Lions-Meyer-Semmes \cite{CLMS} estimate and the Fefferman-Stein \cite{fefferman_stein} estimate, the analogous result has been unknown for the hypodissipative Navier-Stokes equations \eqref{fNS}. In contrast with the classical Navier-Stokes equations, we now have 
\[(-\De)^s p = (-\De)^{2s-2}(\pa_i u_j \pa_j u_i ),\]
which involves the non-local operator $(-\De)^{2s-2}$, and so distributing it equally to $\pa_i u_j$ and $\pa_j u_i$ is not possible in general. We get around this issue by generalising the technique of Li \cite{li} for the Kenig-Ponce-Vega-type commutator estimates  \cite{kenig_ponce_vega} and using the divergence-free condition of $u$ to obtain that
\eqnb\label{est_on_p_intro}
\norm{\mathcal{R}^n(-\De)^s p}_{L^1(\R^3)} \lec_{s,n} \norm{\Lambda^s u}_{L^2(\R^3)}^2
\eqne
for every $n\geq 0$ (where $\mathcal{R}=\La^{-1}\na$ denotes the Riesz transform), which is another main result of this paper, see Proposition~\ref{prop_integrability_of_pressure}. Thanks to this global estimate, the global integrability of $\mathscr{M}_4(\Lambda^{2s-1} \na p)$ follows (see \eqref{pressure_global_grand_max}), and so using a Poincar\'e-type Lemma~\ref{lem_poincare_for_p} and 
the Calder\'on-Zygmund inequality gives boundedness of all pressure terms appearing in \eqref{local_terms_intro} above.

As a corollary of Theorem~\ref{thm_bootstrapping_intro} we obtain an improved statement of the partial regularity result of Tang-Yu \cite{TY15}: if a suitable weak solution $(u,p)$ is such that 
\[
\limsup_{r\to 0^+} r^{-5+4s} \int_{Q_r^*} y^a |\overline{\na} u^*|^2 < \epsilon_0,
\]
then, for some $\rho>0$, $\na^k u \in L^\infty (Q_{\rho})$ for every $k\geq 0$ (rather than merely for $k=0$). As mentioned above, such boundedness result of derivatives is well-known in the case of classical Navier-Stokes equations (see Theorem 1.4 in \cite{NSE_book}, for example), but has been an open problem in the case of fractional dissipation.\\

Our second result is concerned with an application of Theorem~\ref{thm_bootstrapping_intro} that provides an answer to Question 2.
\begin{thm}[Derivatives of suitable weak solutions]\label{thm_main}
Let $u$ be a suitable weak solution of \eqref{fNS} for $\frac 34<s<1$ (see Definition~\ref{def_suitable}). Then
\[
\norm{\na^n u}_{L^{p,\infty}(t_0,T; L^{p,\infty}(K))}^p 
{ \lec_{n,s}}  \norm{u_0}_{L^2(\R^3)}^2 + \frac {|K|}{t_0^{2-\frac1s}}
\]
for every $t_0\in (0,T)$ and every open and bounded subset $K\subset \R^3$, where $p \coloneqq \frac{2(3s-1)}{n+2s-1} $ and $n\in \{1,2\}$. 
\end{thm}
We note that the restriction $n\in \{1,2\}$ comes from the fact that only these values of $n>0$ give that $p\geq 1$ for $s\in (3/4,1)$

Theorem~\ref{thm_main} is related to the study of second derivatives in the case of the classical Navier-Stokes equations, which was initiated by Constantin \cite{constantin}. He showed the existence of global-in-time Leray-Hopf weak solution (i.e. weak solution that satisfies the strong energy inequality\footnote{We refer the reader to \cite{NSE_book} for the definition of Leray-Hopf weak solutions as well as other notions of solutions.}) satisfying a priori estimate for $\na^2 u$ in $L^p $ for every $p<\frac43$ in a periodic setting. Then Lions \cite{lions} improved this result to $\na^2 u \in L^{\frac43,\infty }(\RR^3\times (0,T))$ for any Leray-Hopf weak solution $u$. On the other hand, Vasseur \cite{vasseur} suggested a new approach for the analysis of higher derivatives based on the $\epsilon$-regularity theory, and obtained bounds for $\norm{\na^n u}_{L^{p,\infty }_{\loc}(\RR^3\times (0,T))}$ for $p<\frac{4}{n+1}$, $n\in \mathbb{N}$, that are uniform up to the putative blow-up time of a smooth solution $u$. The latter result was improved to $\norm{\na^n u}_{L^{\frac{4}{n+1}, \infty}_{\loc}(\RR^3\times (0,T))}$ by Choi-Vasseur \cite{choi_vasseur}, who also obtained the estimates on fractional derivatives. Very recently, Vasseur and Yang \cite{VaYa20} improved this result to $\na^2 u\in L^{\frac43, q}_\loc (\R^3\times(0,T))$ for $q>\frac43$, and for any suitable weak solution. In this context, Theorem~\ref{thm_main} is the first result concerned with the higher derivatives of solutions to Navier-Stokes equations with fractional dissipation. \\

The value of the exponent $p=2(3s-1)/(n+2s-1)$ in Theorem~\ref{thm_main} is determined by the {\it energy scaling}, which can be made precise by noting that the hypodissipative Navier-Stokes equation \eqref{fNS} is invariant under the scaling
\begin{align}\label{scale.invarience}
u_\lambda (x,t) \coloneqq \lambda^{2s-1} u(\lambda x, \lambda^{2s}t), \quad
p_\lambda (x,t) \coloneqq \lambda^{4s-2} p(\lambda x, \lambda^{2s}t)
\end{align}
for any $\la>0$. The energy functional, defined by 
\[
E(u) \coloneqq \sup_{t\geq 0} \int_{\RR^3} |u(t)|^2 \d x + 2\int_0^\infty \int_{\RR^3} | \La^s u |^2 \d x\, \d t,
\]
of the rescaled velocity $u_\la$ for $\la>0$ satisfies
\[
E(u_\lambda ) = \lambda^{4s-5} E(u).
\]
We say that a pivot quantity that is integrated over a cylinder has the \emph{energy scaling} if it scales with the same exponent as the energy. For example $|\La^s u |^2$ has the energy scaling because
\[
\int_{Q_1} |\La^s u_\lambda |^2 \d x\,\d t = \lambda^{4s-5} \int_{Q_\lambda } |\La^s u |^2 \d x\, \d t,
\]
where $Q_\la = B_\la \times  (-\la^{2s},0)$. The exponent $p$ in Theorem~\ref{thm_main} is chosen for $|\na^n u|^p$ to have the energy scaling.\\

Our proof of Theorem~\ref{thm_main} is inspired by the approach of Vasseur \cite{vasseur}, which is based on an $\ep$-regularity theorem and {\it Galilean invariance}, that is the invariance of \eqref{fNS} under a transformation 
\eqnb\label{gali_inv}
u_{\bold c} (x,t) \coloneqq  {\bold c}'(t) + u(x-{\bold c}(t), t), \quad
p_{\bold c}(x,t) \coloneqq  p(x-{\bold c}(t),t)
\eqne
for $ \bold{c}(t)\in \RR^3$. To be more precise, suppose that we can obtain local boundedness of $|\na^n u(x,t)|$ under the smallness assumption only on the pivot quantities over $Q_\la(x,t)$ that obey the energy scaling. For example, suppose that $\dint_{Q_\la(x,t)} |\La^s u|^2 \leq \ep^2 \la^p$ for $p\coloneqq \frac{2(3s-1)}{2s}$ implies boundedness of  $|\na^n u(x,t)|$ on $Q_{\la/2}(x,t)$ if $\ep>0$ is sufficiently small. Then the Lebesgue measure of the super level set $\{{ (x,t)\in \R^3\times(\la^{2s}, \infty): }|\na u(x,t)|\geq \la\}$ can be estimated using Chebyshev's inequality, provided that $|\La^s u|^2$ is integrable in the whole domain $\R^3\times(0,T)$, which results in $\na u \in L^{p,\infty}_{\loc}$. The point here is that the desired value of $p$ comes from using the pivot quantities that are globally integrable and have the energy scaling. Such quantities will be called \emph{scale optimal}. For instance the quantities $|\La^s u|^2$ and $|(-\De)^s p|$ are scale optimal\footnote{A suitable weak solution $u$ satisfies the energy inequality \eqref{EI} (see Section~\ref{sec_suitable_ws}), which gives the global integrability of $|\La^s u|^2$. For the global integrability of $(-\De)^sp$, recall \eqref{est_on_p_intro}.}. Thus one would wish for an $\epsilon$-regularity result that implies local boundedness of spatial derivatives of $u$ from a smallness assumption that involves only scale optimal quantities.

Although such a result is currently unknown, using the Galilean invariance \eqref{gali_inv} it turns out sufficient to prove such $\ep$-regularity result under the assumption that the velocity has zero $\psi$-mean, $\int_{\RR^3} u(t)\psi \d x=0$, for every $t$. Here, $\psi$ is a function in $C_c^\infty (B_1)$ with $\int_{\RR^3} \psi \,\d x=1$, which we now fix (and it will remain fixed throughout the paper). Under such assumption, we obtain the following  $\ep$-regularity result.
\begin{thm}[Local regularity]\label{thm_local_reg_intro} Let $s\in (\frac34,1)$. There exists $\ep=\ep (s, \psi )>0$ such that if $(u,p )$ is a suitable weak solution of \eqref{fNS} such that $\int u(x,t)\psi(x) \d x =0$ for all $t\in (-5^{2s},0) $ and
\begin{equation}\begin{split}
\label{smallness_intro}
&\int_{Q_5^*} y^a |\bna u^*|^2 \d x\,\d y\, \d t  + \int_{-5^{2s}}^0 \int_{B_5}\int_{B_5} \frac{|u(x,t)-u(y,t)|^2}{|x-y|^{3+2s}}\d x \,\d y\, \d t\\
&\hspace{2cm}+ \int_{Q_5}   
\left( (\mathcal{M} |\La^s u |^{\frac 2{1+\de}})^{1+\de} +  |\La^{2s-1}\na p|  +|\mathscr{M}_4 (\La^{2s-1}\na p)|\right) \,\d x\, \d t \leq \ep,
\end{split}\end{equation}
where $\delta \coloneqq \frac{2s}{6-s}$, then 
\[
\sup_{Q_\frac12} (|u|+ |\na u|+|\nabla^2  u|)  \leq C_0
\]
for some positive constant $C_0=C_0(s)$.
\end{thm}
Here, $a\coloneqq 1-2s$ and $u^*=u^*(x,y)$ denotes the Caffarelli-Silvestre extension of $u$, which gives rise to the extended cylinder $Q_5^*\subset \R^5$ of $Q_5$ and the gradient $\bna$ with respect to $(x,y)$. We give the precise definitions in Section~\ref{sec_frac_lap} below. We note that, in order to prove the above theorem, only local boundedness of $u$ needs to be shown, as local boundedness of $|\na u|$ and $|\na^2 u |$ follows from Theorem~\ref{thm_bootstrapping_intro}.

An important ingredient of the proof of Theorem~\ref{thm_local_reg_intro} is a new Poincar\'e-type inequality for the pressure function,
\begin{equation}\label{poin.p_intro}
\norm{\na p - (\na p)_\psi}_{L^{\frac65}(B_\frac54)} 
\lec \norm{\La^{2s-1}\na p}_{L^1(B_5)} 
+ \norm{ \mathscr{M}_4(\La^{2s-1}\na p) }_{L^1(B_3)} ,
\end{equation}
where $(g)_\psi \coloneqq \int_{\RR^3} g\psi \,\d x$, which we develop in Lemma~\ref{lem_poincare_for_p}. Such inequality is necessary in the local regularity argument (see \eqref{p_in_L2}) to control the pressure function using only scale-optimal quantities. As above, such inequality is also valid with the grand maximal function $\mathscr{M}_4$ replaced by the Hardy-Littlewood maximal function $\mathcal{M}$, but in that case the global integrability would be lost, and so would be the scale-optimality of the local regularity result of Theorem~\ref{thm_local_reg_intro}. In fact, one could suspect that perhaps employing the smooth maximal function or the non-tangential maximal function would be sufficient to get around this difficulty. This has been demonstrated, for example, by Choi-Vasseur \cite{choi_vasseur} whose application of the smooth maximal function allowed them to obtain the endpoint integrability exponent $4/3$. In fact, it is also true of \eqref{poin.p_intro}, for which we show (in Lemma~\ref{lem_poincare_for_p}) the stronger estimate with $\mathscr{M}_4$ replaced by the smooth maximal function. 

The reason for the necessity to use of the grand maximal function $\mathscr{M}_4$ comes from our first result, Theorem~\ref{thm_bootstrapping_intro}, where it is needed to estimate the commutator involving the pressure function, $[\La^\ga, \phi]\na p$ (where $\ga\in (s,1)$). It is our most challenging estimate, and we present it in Lemma~\ref{lem_comm.p}. Its difficulty comes from the fact that this commutator involves both the nonlocality caused by the pressure function $p$ and the nonlocality of $\La^\ga$, and, as above, its estimate needs to be strong enough to involve only local information of a scale-optimal quantity. This is where the flexibility allowed by the grand maximal function becomes essential as, in some sense, it allows to control, in $L^1$, a family of double convolutions with uniform estimates (see \eqref{def_of_L} and Lemma~\ref{lem_bound_G_by_grand_max}). We discuss it in more detail below \eqref{It11_optimal}, but we point out that it is the main reason why we are able to prove a scale-optimal result of the form of Theorem~\ref{thm_local_reg_intro}. 
In other words, we show that one can obtain the endpoint integrability exponent $p=2(3s-1)/(n+2s-1)$ in Theorem~\ref{thm_main} thanks to the grand maximal function $\mathscr{M}_4$. \\

 Furthermore, as a corollary of Theorem~\ref{thm_local_reg_intro}, we prove that the local regularity of suitable weak solutions is still valid if the zero $\psi$-mean condition is replaced by a smallness assumption on $|u|^3+|p|^{\frac32}$. 
\begin{cor}\label{cor_single_scale_local_reg_into}
There exists $\ep>0$ such that if a suitable weak solution $(u,p )$ of \eqref{fNS} satisfies
\begin{equation*}\begin{split}
&\int_{Q_5^*} y^a |\bna u^*|^2 \d X \,\d t + \int_{-5^{2s}}^0 \int_{B_5}\int_{B_5} \frac{|u(x,t)-u(y,t)|^2}{|x-y|^{3+2s}}\d x \,\d y\, \d t\\
&\hspace{1cm}+ \int_{Q_5}  \left( 
(\mathcal{M} |\La^s u |^{\frac 2{1+\de}})^{1+\de} +  |\La^{2s-1}\na p|  +|\mathscr{M}_4 (\La^{2s-1}\na p)|+|u|^3 + |p|^{\frac32} \right) \leq \ep,
\end{split}\end{equation*}
where $\delta \coloneqq \frac{2s}{(6-s)}$. Then
\[
\sup_{Q_\frac12}\left(|u|+
|\na u|+|\nabla^2  u| \right) \leq C_1
\] 
for some constant $C_1>0$. 
\end{cor}
The corollary makes it possible to estimate the box-counting dimension of singular set, whose upper bound is consistent with a similar result in the hyperdissipative case \cite[Corollary 1.4]{colombo_dl_m} that is concerned with H\"older continuity of solutions (regularity of higher derivatives in the hyperdissipative case remains an open problem).
\begin{corollary}[The box-counting dimension]\label{cor_box_counting} Let $(u,p)$ be a suitable weak solution  in $\R^3\times(0,T)$, and let 
\[\begin{split}
S \coloneqq &\{ (x,t) \in \RR^3 \times (0,T )\colon \text{some spatial derivative of } u \\
&\hspace{4cm}\text{ is unbounded in any neighbourhood  of } (x,t) \} 
\end{split}
\]
denote the singular set of $u$. Then, for any $t_0>0$, the box-counting dimension of the singular set satisfies
\[
d_B(S\cap \{ t>t_0 \}) \leq \frac13 (15-2s-8s^2)
\]
for every $t_0\in (0,T)$.
\end{corollary}
In Figure~\ref{fig_dimensions} below, we sketch the currently known estimates on the dimension of the singular set for the Navier-Stokes equations with different powers of dissipation.\nocite{wang_yang,rob_sad_2009}
\begin{figure}[h]
\centering
 \includegraphics[width=0.6\textwidth]{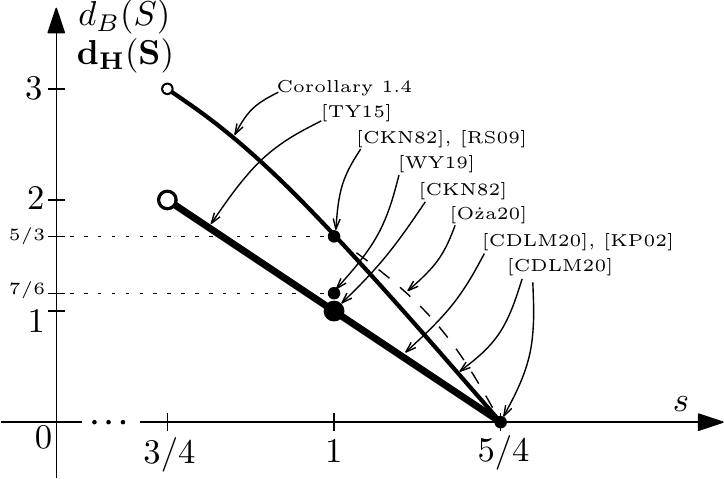}
 \nopagebreak
  \captionof{figure}{Sketch of the currently known estimates on the dimension of the singular set. The curve of the box-counting bound is described by the polynomial $(-8s^2-2s+15)/3$. The dashed line represents the bound on the box-counting dimension of the singular set in space of any Leray-Hopf weak solutions (i.e. not necessarily suitable weak solutions) and is described by the polynomial $(-16s^2+16s+5)/3$.}\label{fig_dimensions} 
\end{figure}

Finally, let us briefly comment why we only consider $s>3/4$. If $s<3/4$ then it is not clear how one should interpret the local energy inequality \eqref{LEI}, since the cubic term on the right-hand side can no longer be well-defined using the a priori estimates. (Note that the a priori estimates $u\in L^\infty_tL^2_x$, $\La^s u \in L^2_tL^2_x$ give only that $u\in L^{2(2s+3)/3}(\RR^3 \times (0,T))$; and $2(2s+3)/3>3$ iff $s>3/4$.) In fact, the existence of suitable weak solutions is not clear if $s\leq 3/4$, as one can no longer use compactness in $L^3$. In the range $s\in (3/4 ,1)$ we control $\| u \|_{L^4 (B_{5/4})}$ using the $L^2$ norm of the extension $u^*$ (see \eqref{computation_local_reg}) as well as the $L^1_tL^2_x$ norm of the pressure function (see \eqref{p_in_L2} and \eqref{computation_local_reg}), which are crucial elements of the proof of Theorem~\ref{thm_local_reg_intro}. 

The structure of the article is as follows. In Section~\ref{sec_prelims}, we first introduce notations and some preliminary concepts. 
Then, in Section~\ref{sec_pf_of_main_thm}, we prove Theorem~\ref{thm_main} using the global estimate \eqref{est_on_p_intro} on the pressure and the local regularity, Theorem~\ref{thm_local_reg_intro}, which are consequently proved in Section~\ref{sec_global_p} and Sections~\ref{sec_local_study}, respectively. Section~\ref{sec_local_study} is the central part of the paper, where we first prove a Poincar\'e-type inequality for the pressure function in Section~\ref{sec_poincare_for_p}, and then obtain the local boundedness of $u$ in Section~\ref{sec_boundedness_of_u}. Section~\ref{sec_consequences} discusses the proof of Corollary~\ref{cor_single_scale_local_reg_into} and Corollary~\ref{cor_box_counting}. Section~\ref{sec_bootstrapping} is dedicated to the proof of Theorem~\ref{thm_bootstrapping_intro} via the new bootstrapping scheme and Section~\ref{sec_commutator_est} contains the required commutator estimates.

\section{Preliminaries and Notations}\label{sec_prelims}
\subsection{Notations}
For any quantities $A$ and $B$, we will write $A\lec  B$ if $A\leq CB$ for some positive constant $C>0$. Similarly, we will write $A\gtrsim B$ if $A\geq CB$ and $A\sim B$ if $A\lec B$ and $A\gtrsim B$. 

We let $B_r(x) \coloneqq \{ x'\in \R^3 : |x'-x|<r\}$ be the Euclidean ball in $\R^3$ centred at $x$ with radius $r>0$ and  $B_r^*(x)$ be its extension $B_r^*(x) \coloneqq  B_r(x) \times [0,r)$ in $\R^4$. We denote the parabolic cylinder centred at $(x,t)$ with radius $r$ by $Q_r(x,t) \coloneqq B_r(x) \times (t-r^2, t]$ and its extension by $Q_r^*(x,t)\coloneqq B_r^*(x) \times (t-r^2, t]$. For brevity, when the centres are origin, i.e., either  $x=0$ or $(x,t)=(0,0)$, we use the abbreviations $B_r$, $B_r^*$, $Q_r$, and $Q_r^*$, respectively.   

Given a sequence $\{a_j\}$, we use the sequence space $l^p$ with its norm $(a_j)_{l^p_j} =(\sum_{j} |a_j|^p)^\frac 1p$ for $p\in [1,\infty)$ and $(a_j)_{l^\infty_j} = \sup_j |a_j|$.  
We denote the usual Lebesgue spaces by $L^p(\Om)$, where $\Om$ is a subset of either $\R^3$ or $\R^3\times\R$. When $\Om$ is whole space $\RR^3$ we simply write $L^p \equiv L^p (\RR^3)$ (and similarly for other function spaces) and $\norm{\cdot}_{L^p}\equiv \norm{\cdot}_p $. Given a cylinder $Q= B\times I\subset \R^3\times \R$ we use the abbreviation $L^{p'}_tL^{p}_x (Q)\coloneqq L^{p'}(I;L^p(B))$. Moreover, given a domain $K$, we denote the weak $L^p$-space (or the {Lorentz space}) by $L^{p,\infty}(K)$, with the norm  
\[
\| f \|_{L^{p,\infty} (K)} \coloneqq \inf \left\lbrace C>0 \colon \left| K\cap \{ f >\lambda \} \right| \leq \frac{C^p}{\lambda^{p}} \text{ for every } \lambda>0 \right\rbrace .
\]
We follow the usual definition of the Sobolev space $W^{k,p}(\Om)$ for integer $k$ and $p\in [1,\infty]$. (As mentioned above, if $\Om=\R^3$ then we simply write $W^{k,p} \coloneqq W^{k,p}(\R^3)$.) We denote the collection of all smooth functions with the compact support in $\Om$ by $C_c^\infty(\Om)$. For brevity we omit the integral region if it is $\R^3$, i.e. we write $\int f \d x \coloneqq  \int_{\RR^3} f \d x$. We denote the Lebesgue measure of a set $E$ by $|E|$ and we let $(u)_E$ denote the average of $u$ over $E$, $(u)_E\equiv \dint_E u\, \d x \coloneqq \frac 1{|E|}\int_E u\,\d x$.

We use the following the Fourier transform convention,
\begin{align*}
\widehat{f}(\xi) =& \int f(x) e^{-ix\cdot\xi} \d x, \qquad
f(x) = \frac 1{(2\pi)^3} \int \widehat{f}(\xi) e^{ix\cdot \xi} d \xi;
\end{align*}
then $\widehat{fg}(\xi)=\frac 1{(2\pi)^3} \int \widehat{f}(\xi-\eta) \widehat{g}(\eta)  \d \eta$. We follow the standard convention regarding the Littlewood-Paley operators: we let $\rho(\xi)$ be a radial smooth function supported in ${B_2}$ which is identically $1$ on $\overline{B_1}$. For any integer $j$ and distribution $f$ in $\R^3$, we set
\begin{align}
\widehat{P_{\leq j} f}(\xi) &\coloneqq \rho\left(2^{-j} \xi \right) \hat{f}(\xi),\quad
\widehat{P_{> j} f}(\xi) \coloneqq  \left(1-\rho\left(2^{-j} \xi\right)\right) \hat{f}(\xi), \nonumber\\
\label{varrho}
\widehat{P_{j}f}(\xi) &\coloneqq \left( \rho\left(2^{-j} \xi\right)
-\rho\left(2^{-(j-1)} \xi\right)\right) \hat{f}(\xi)
=: \varrho\left(2^{-j} \xi\right) \widehat{f}(\xi).
\end{align}
\subsection{Fractional Laplacian and its extension}\label{sec_frac_lap}
We first introduce several characterizations of the fractional Laplacian. The fractional Laplacian $(-\De)^s$ for $s\in (0,1)$ can be represented as 
\begin{align}\label{def.frac}
(-\De)^s u(x) = C_s \text{ p.v.}\int \frac {u(x)-u(y)}{|x-y|^{3+2s}} \d y
\end{align}
for some normalization constant $C_s$, and for $s\in (-\frac32,0)$ as
\begin{align*}
(-\De)^{s} u(x) = C_{s} \int \frac {u(y)}{|x-y|^{3-2s}} \d y.
\end{align*} 
See \cite{Stinga19} for the details. Moreover the  fractional Laplacian for $s\in (0,1)$ can be characterised using the Caffarelli-Silvestre \cite{caffarelli_silvestre} extension 
\eqnb\label{def_of_caff_silv_ext}
u^*(x , y)\coloneqq  \int P(x-z,y)u(z) \d z ,
\eqne
where $P(x,y) \coloneqq { c_s}\, {y^{2s}}{(|x|^2+|y|^2)^{\frac{-(3+2s)}2}}$ for some normalization constant $c_s>0$. It is a solution of the extension problem 
\begin{align}\label{ext.prob}
\begin{cases}
\bna \cdot (y^a \bna u^*) =0, \quad (x,y)\in \R^4_+\\
u^*(x,0) = u(x)
\end{cases}
\end{align}
where $a\coloneqq 1-2s$ and $\overline{\na}$ is the gradient with respect to $(x,y)$. The fractional Laplacian can be recovered using the extension by the formula  
\begin{align}\label{frac.limit}
(-\De)^{s} u(x) = -\overline{C}_s\lim_{y\to 0^+} y^a \pa_y u^*(x,y)
\end{align}
in the sense of distributions where $\overline{C}_s$ is a  constant depending only on $s$. We note that we will sometimes use the same notation $\phi^*$ to denote any extension to $\RR^4_+$ of a function $\phi$ defined on $\RR^3$, but in such case, we will specify it. What is more, considering the energy functionals, we also have 
\begin{align}\label{ext.norm}
\int_{\R^4_+} y^a |\bna u^*|^2 \d X 
= \int |\La^s u|^2 \d x,
\end{align}
where we set $X\coloneqq (x,y)$. Moreover, the Caffarelli-Silvestre extension of rescaled solution $u_\la$, defined as in \eqref{scale.invarience}, can be written as a rescaled extended solution,
\[
(u_\la)^*(x,y, t) = \la^{2s-1} u^*(\la x,\la y, \la^{2s}t).
\]

Furthermore, we introduce fractional Leibniz rules \cite{kenig_ponce_vega, grafakos_oh, li}: for $\al>0$, $\be_1, \be_2\geq 0$, $\be = \be_1 + \be_2\in (0,1)$,  
\begin{align}
\| \La^\al (fg) \|_{L^r} &\lec_{\al, r, p_1, p_2, q_1, q_2} \| \La^\al f \|_{L^{p_1}} \| g\|_{L^{q_1}}  +\| f \|_{L^{p_2}} \|  \La^\al g\|_{L^{q_2}}, 
\label{fractional_leibniz}\\
\| \La^\be (fg)- g\La^\be f - f \La^\be g  \ \|_{L^r} &\lec_{\be, \be_1, \be_2, r, p_1,q_1} \| \La^{\be_1} f \|_{L^{p_1}} \|\La^{\be_2} g\|_{L^{q_1}}
\label{fractional_leibniz1}
\end{align}
provided that $1\le r <\infty $, and $1<p_1,p_2, q_1, q_2\le \infty$ satisfying $\frac 1r=\frac 1{p_1}+\frac 1{q_1} = \frac 1{p_2} + \frac 1{q_2}$, 

\subsection{Sobolev-Slobodeckij space}
In this subsection, we introduce fractional order Sobolev spaces. For a given Lipschitz domain $\Om\subset \R^3$, $k\in \mathbb{N}\cup\{0\}$, $\ga \in (0,1)$, $p\in [1, \infty ) $ we define the Sobolev–Slobodeckij space 
\begin{align*}
W^{k+\ga,p}(\Om) \coloneqq 
\left \{f \in W^{k,p}(\Om) \colon \norm{f}_{\dot{W}^{k+\ga ,p}(\Om)}^p \coloneqq \int_\Om\int_\Om \frac{|\na^k f(x)-\na^k f(y)|^p}{|x-y|^{3+\ga p}} \d x \, \d y <+\infty 
\right \}.
\end{align*}
We define the Sobolev-Slobodeckij space norm by
\begin{align*}
\norm{f}_{W^{\ga ,p}(\Om)}
= \norm{f}_{\dot{W}^{\ga ,p}(\Om)} + \norm{f}_{L^p(\Om)}. 
\end{align*}
In case of $p=2$, we also use the notation $W^{\ga ,2} = H^{\ga }$ and $\dot{W}^{\ga ,2} = \dot{H}^{\ga }$.

When $\Om = \R^3$ and $\ga >0$, the Sobolev-Slobodeckij space is related to the $L^p$ norm of the fractional derivatives, 
\eqnb\label{norm_comparisons}
\begin{split}
\norm{f}_{\dot{W}^{\ga,p}} &\lec \norm{\La^\ga f}_{L^p} + \norm{f}_{L^p} \quad \text{ for } p\in [2,\infty),\\
\norm{\La^{\ga} f}_{L^p} &\lec \norm{f}_{{W}^{\ga,p}} \ \  \qquad\qquad\qquad \text{ for } p\in (1,2],
\end{split}
\eqne
see Theorem~5.5 in \cite{stein_singular} for a proof. Furthermore, if $p=2$ and $0<\ga<1$, then
\begin{align}\label{Sob.Slo.L2norm}
\norm{f}_{\dot{W}^{\ga, 2}} \sim \norm{\La^\ga f}_{L^2},
\end{align}
see, for example, in \cite[Proposition 3.4]{hitchhiker} for a proof.

\subsection{Suitable weak solutions}\label{sec_suitable_ws}
In this section we introduce the notion of a suitable weak solution to \eqref{fNS}. We first define Leray-Hopf weak solutions. 
\begin{defn}[Leray-Hopf weak solution]\label{def_weak}
Let $u_0\in L^2(\R^3)$ be divergence-free. We say that a function $u$ is a \emph{weak solution} of \eqref{fNS} with initial data $u_0$ on $\R^3\times (0,T)$ if
\begin{enumerate}
\item $u$ belongs to the space
\[
u\in L^\infty(0,T;L^2) \cap L^2(0,T; \dot{H}^s)
\]
and is divergence free in the sense of distributions. 
\item $u$ satisfies \eqref{fNS} in a weak sense,
\begin{align*}
\int_0^T\int u \cdot (\pa_t \xi -(-\De)^s \xi) + u\otimes u:\na \xi \,\d x \, \d t = -\int u_0 \cdot \xi|_{t=0} \d x
\end{align*} 
for any divergence-free $\xi \in C_c^\infty(\R^3\times [0,T))$.

\item $u$ satisfies the strong energy inequality,
\begin{align*}
\int |u(t)|^2  \d x + 2\int_{t_0}^t \int | \La^s u |^2 \d x \, \d \tau \leq \int |u(t_0)|^2  \d x 
\end{align*}
for almost all $t_0\geq 0$ (including $0$) and all $t>t_0$.
\end{enumerate}
\end{defn}
Using the energy inequality it is clear that the initial data is achieved as the strong limit $u(\cdot, t) \to u_0$ in $L^2$ as $t\to 0^+$. Given $u$ the corresponding pressure is given by the singular integral
\eqnb\label{pressure_formula}
p(x) \coloneqq \int_{\RR^3} \sum_{i,j=1}^3 \frac{ \p_i u_j (y) \p_j u_i (y)}{4\pi |x-y |}\,  \d y,
\eqne
as in the case of classical Navier-Stokes equations.

We now define suitable weak solutions.  

\begin{defn}[Suitable weak solution]\label{def_suitable}
We say that a Leray-Hopf weak solution $(u,p)$ of \eqref{fNS} on $\RR^3 \times (0,T)$ is \emph{suitable} if 
\begin{enumerate}
\item $(u,p)$ satisfy \eqref{fNS} in the sense of distributions,
\item for every $\xi = \xi(x,y,t)\in C^\infty_c (\RR^4 \times (0,T); [0,\infty ))$, {\it the local energy inequality}
\eqnb\label{LEI}
\begin{split}
\int |u(t)&|^2 \xi(t)|_{y=0} \d x
+ 2 \overline{C}_s \int_{t_0}^t 
 \int_{\RR^4_+} y^a |\overline{\na } u^* |^2 \xi 
 \d X \,\d \tau\\
&\leq  \int |u(t_0)|^2 \xi(t_0)|_{y=0} \d x+ \overline{C}_s \int_{t_0}^t \int_{\RR^4_+} |u^*|^2 \overline{\mathrm{div}} (y^a \overline{\na } \xi ) \d X \, \d\tau \\
&+ \int_{t_0}^t \int (u\cdot \na \xi|_{y=0} )\left( 2p+|u|^2\right)  \d x \, \d \tau 
+ \int_{t_0}^t \int |u|^2 \left( \pa_t\xi|_{y=0} +\overline{C}_s \lim_{y\to 0^+}y^a\pa_y \xi \right)   \d x \,\d \tau
\end{split}
\eqne
holds for all $0<t_0<t<T$ (recall $a\coloneqq 1-2s$), where $\overline{\mathrm{div}}$ denotes the divergence operator with respect to $X=(x,y)$ and $\overline{C}_s$ is defined as in \eqref{frac.limit}. 
\end{enumerate} 
\end{defn}
(Analogously one can define a suitable weak solution on any open time interval.) As mentioned in the introduction, the existence of a suitable weak solution to \eqref{fNS} on $\RR^3 \times (0,T)$ has been obtained by \cite[Theorem 4.1]{TY15} for any divergence-free initial data $u_0\in L^2$. The proof is based on dissipative regularization. 

We note that any suitable weak solution $(u,p)$ defined on $\RR^3 \times (0,T)$ satisfies $u\in L^r(0,{T};L^q)$ and $p\in L^{r/2}(0,{T};L^{q/2})$ for any $r,q\in [2,\infty ]\times [2,\frac 6{3-2s}] $ with
\begin{align*}
\frac {2s}{r} + \frac 3q = \frac 32.
\end{align*}
This follows from interpolation and the Calder\'on-Zygmund estimate. In particular, $u\in L^3(\R^3\times(0,T))$ and $p\in L^{3/2}(\R^3\times(0,T))$, so that the second last integral in the local energy inequality is well-defined.

\subsection{Poincar\'e inequality}
Here we introduce two Poincar\'e inequalities that involve the Caffarelli-Silvestre extension.
\begin{lem}[Poincar\'e inequality using the extension]\label{lem_Poin} If $u$ satisfies $\int u(x) \phi(x) \d x=0$ for some non-zero smooth cut-off $\phi$ supported in $B_1$, then it satisfies 
\begin{align*}
\norm{u}_{L^{\frac{6}{3-2s}}(B_1)} 
\lec \norm{y^\frac{a}2 \bna u^*}_{L^2(B_2^*)}
\end{align*}
\end{lem}

\begin{rem} The domains $B_1$ and $B_2^*$ can be easily replaced by $B_{r}$ and $B_R^*$ any $r<R$; a suitable factor involving $r$ and $R$ then appears to respect the scaling of the inequality.  
\end{rem}
\begin{proof} We note that we have the Poincar\'e-type inequality 
\eqnb\label{prePoin}
\norm{u-(u)_{B_1}}_{L^{\frac{6}{3-2s}}(B_1)} 
\lec \norm{y^\frac{a}2  \bna u^*}_{L^2(B_2^*)},
\eqne
which was introduced in \cite[Proposition 2.2]{TY15}. Then the desired estimate follows from
\begin{equation}\begin{split}
\label{est.ave}
\norm{(u)_{B_1} }_{L^{\frac{6}{3-2s}}(B_1)} 
&\sim |(u)_{B_1}|
=\abs{(u)_{B_1}-\frac 1{A_\phi}\int u\phi \,\d x' }
= \left |\frac 1{A_\phi}\int (u-(u)_{B_1})\phi \,\d x'\right |\\
&\lec \norm{u-(u)_{B_1}}_{L^{\frac{6}{3-2s}}(B_1)} 
\lec \norm{y^\frac{a}2  \bna u^*}_{L^2(B_2^*)},
\end{split}\end{equation}
where $A_\phi \coloneqq \int \phi(x) \d x$.
\end{proof}
We will also use the weighted Poincar\'e inequality
\eqnb\label{weighted_poin}
\int_{B_1^*} y^a | u^* - (u^*)_{B_1^*} |^2 \lec \int_{B_1^*} y^a | \bna u^* |^2 ,
\eqne
where $(u^*)_{B_1^*} \coloneqq |B_1^*|^{-1} \int_{B_1^*} y^a u^*$, proved by \cite{fabes_k_s} (see also (2.10) in \cite{TY15}).

We will later show (in Section~\ref{sec_poincare_for_p}) another Poincar\'e-type inequality, which will use the smooth maximal function (see \eqref{def_of_smooth_max_fcn} below) instead of the extension.

\subsection{The Hardy space and the grand maximal function}\label{sec_hardy_and_grand_max}

The Hardy space $\mathcal{H}^1$ is defined by
\[
\mathcal{H}^1 (\RR^3) \coloneqq \{ f\in L^1(\R^3) \colon \mathcal{R}f\in L^1(\R^3) \}
\]
with the Hardy norm
\[
\| f\|_{\mathcal{H}^1} \coloneqq \| f \|_{1} + \| \mathcal{R} f \|_{1} .
\]
where $\mathcal{R}\coloneqq \La^{-1}\na$ is the Riesz transform. 

The Hardy norm can be characterized in various ways; first of all, using the Littlewood-Paley projection operator $P_j$, we have that
\begin{align}\label{hardy.norm.1}
\norm{f}_{\mathcal{H}^1} 
\sim \|(P_j f)_{l^2_j}\|_1, 
\end{align}
see (2.1.1) in \cite{grafakos_modern}. Moreover, it can be characterized using the grand maximal function. To this end we first consider several types of maximal functions. To be more precise, the {\it Hardy-Littlewood maximal function} is defined by
\[
\mathcal{M}f(x) \coloneqq \sup_{r>0} \dint_{B_r(x)} |f(y)| \d y.
\]
Given $\Psi \in \mathcal{S}(\RR^3)$ we denote the {\it smooth maximal function} of $f$ with respect of $\Psi$
by 
\eqnb\label{def_of_smooth_max_fcn}
\mathcal{M}(f;\Psi )(x) \coloneqq \sup_{t>0} \left| \Psi_t \ast f (x)\right|,
\eqne
where  $\Psi_t (x) \coloneqq t^{-3} \Psi(t^{-1}x) $. Furthermore, we denote the {\it non-tangential maximal function} with aperture $1$ with respect to $\Psi$ by
\[
\mathcal{M}_1^*(f;\Psi)(x) \coloneqq \sup_{t>0, \, |y-x|\leq t} \left|  \Psi_t \ast f (y)\right|.
\]
The {\it grand maximal function } is defined as
\begin{align*}\mathscr{M}_N (f)(x) 
\coloneqq \sup\left\{ \mathcal{M}_1^*(f;\Psi)(x) : \Psi\in \mathcal{S}, \int (1+|x|)^N \sum_{|\al|\leq N+1} |\pa^\al \Psi(x)| \d x \leq 1\right\}, 
\end{align*}
where $N\geq 4$. By definition, we clearly have
\[
\mathcal{M}(f,\Psi )  \leq \mathcal{M}_1^* (f;\Psi ) \hspace{2cm} \text{ for all }\Psi \in \mathcal{S}.
\]
Using the grand maximal function, the Hardy norm can be characterized as
\eqnb\label{hardy_characterisation}
\norm{f}_{\mathcal{H}^1}\sim \norm{\mathscr{M}_4 (f)}_1.
\eqne
For a proof we refer the reader to \cite[Theorem 2.1.4]{grafakos_modern} .

In particular, we have
\eqnb\label{smooth_max_fcn_vs_hardy}
\|\mathcal{M}(f;\Psi ) \|_1 \lec_\Psi \| f \|_{\mathcal{H}^1}
\eqne
for any given $\Psi \in \mathcal{S}$. The benefit of the grand maximal function is that it is bounded as an operator $\mathcal{H}^1(\R^3)\to L^1(\R^3)$, while the Hardy-Littlewood maximal function
is not.
Furthermore, \eqref{hardy_characterisation} and \eqref{est_on_p_intro} imply that
\eqnb\label{pressure_global_grand_max}
\| \mathscr{M}_4 ( \La^{2s-1} \na p ) \|_1 \lec \| \mathcal{R} (-\Delta )^s p \|_{\mathcal{H}^1} \lec \| \mathcal{R}^2 (-\Delta )^s p \|_{1} \lec \| \La^s u \|_2^2 \lec C(u_0)
\eqne
for any weak solution $u$ to \eqref{fNS}.

We conclude this section by introducing several properties \cite{li}
of the Hardy-Littlewood maximal function.
\begin{lemma} 
\label{lem.max} Suppose $f\in \mathcal{S}'(\R^3)$ and $\supp(\widehat{f})\subset B_r$ for some $r>0$. Then
\[
\sup_{z\in  \R^3} \frac{|f(x-z)|}{(1+r|z|)^3} 
\lec \mathcal{M}|f|(x) \qquad \text{ for } x\in \R^3.  
\] 
\end{lemma}

\begin{lemma} For any $\Psi \in \mathcal{S}$ and $f\in L^1_\loc(\R^3)$, we have
\eqnb\label{smooth_max_bdd_by_max}
\mathcal{M} (f ; \Psi )(x) \lec_\Psi \mathcal{M}f(x), \qquad \text{ for } x\in \R^3. 
\eqne
\end{lemma}
\begin{proof} For every $t>0$, we let $\Psi_t (x) \coloneqq t^{-3} \Psi (t^{-1} x)$ and note that
\[\begin{split}
\left| \Psi_t \ast f (x) \right| 
&\le t^{-3}
 \left(\int_{|y|\le t} +  \sum_{l\geq 0} \int_{2^lt<|y|\leq 2^{l+1}t}\right) |\Psi (t^{-1} y)| |f (x-y)| \d y  \\
&\leq \| \Psi \|_\infty  t^{-3}\int_{|y|\leq t}  |f(x-y)|\d y 
+\sup_{x'} |x'|^4 |\Psi(x')|  t^{-3} \sum_{l\geq 0} \int_{2^lt<|y|\leq 2^{l+1}t} \frac{1}{t^{-4}|y|^4}| f(x-y)|\d y\\
&\lec_\Psi   \dint_{|y|\leq t}  |f(x-y)|\d y 
+  \sum_{l\geq 0} 2^{-l} \dint_{|y|\leq 2^{l+1}t} | f(x-y)|\d y\\
&\lec \mathcal{M}f (x).
\end{split}
\]
\end{proof}

\subsection{Parabolic regularity}\label{sec_para_reg}
Here we mention several facts regarding regularity of solutions to the initial value problem
\begin{align}\label{fractional.dis.eq}
\begin{cases}
\pa_t v + (-\De)^s v = g  &\qquad\R^3\times(0,T) \\
v|_{t=0} =0			,
\end{cases}
\end{align}
where $T\in (0,\infty)$ and $g$ is given. We let $e^{-t(-\De)^s}$ denote that fractional heat semigroup define in the Fourier space by $\reallywidehat{e^{-t(-\De)^s}g}(\xi) = e^{-t|\xi|^{2s}} \widehat{g}(\xi)$. The fractional heat semigroup satisfies the following $L^{\overline{p}}$-$L^p$ estimate. 
\begin{lem}\label{lemma:para.est} For any $s>0$, $ \al \le \overline{\al}$, and $1\le p \le \overline{p} \le \infty$, we have
\begin{align*}
\norm{\La^{\overline\al} e^{-t(-\De)^s}f}_{{\overline p}}
\lec t^{-\frac{\overline{\al}-\al}{2s} - \sigma} \norm{\La^{\al }f}_{{p}} 
\end{align*}
where $\sigma = \frac 3{2s} \left(\frac 1p - \frac 1{\overline p}\right)$. 
\end{lem}
We refer the reader to Lemma~2.2 in \cite{Zhai10} for a proof. The lemma shows that if $v$ is defined by the Duhamel formula
\begin{align}\label{sol.fr.dis.eq}
v \coloneqq \int_0^t e^{-(t-\tau)(-\De)^s } g(x, \tau) \d \tau,
\end{align}
then 
\begin{align}\label{para.est}
\norm{\La^{\overline\al} v}_{L^{\overline r}(0,T;L^{\overline p})}
\lec \norm{\La^{\al }g}_{L^r (0,T;L^{p})}
\end{align}
for any $T>0$, provided that $1\le r \leq \bar{r}\leq  \infty$ satisfies
\[
\frac{\overline\al - \al}{2s} + \sigma  + \left(\frac 1r - \frac 1{\overline r} \right) < 1 .
\]
Moreover, in the case of $\sigma = 0$ (i.e. $ p= \overline{p}$) and $\al = \overline{\al}$, we have 
\begin{align}\label{para.est.0}
\norm{v}_{L^{\infty}(0,T;L^{p})}
\lec \norm{g}_{L^1(0,T;L^p)}
\end{align}
for every $T>0$. Furthermore, we note that $v$ defined by \eqref{sol.fr.dis.eq} for $g\in L^p(\R^3\times (0,T))$, where $p\in [1,\infty )$, is the unique distributional solution to \eqref{fractional.dis.eq} in the space $L^r (0,T; L^p )$ for any $r,p\in [1,\infty ]$, that is if $w\in L^r (0,T; L^p )$ and
\eqnb\label{uniqueness_of_fr_heat}
\int_0^T w(\phi_t - (-\Delta )^s \phi ) =0
\eqne
for every $\phi \in C_c^\infty (\RR^3 \times [0,T))$ then $w=0$, which can be proved in the same way as Theorem~4.4.2 in \cite{giga_giga_saal}.

\section{Proof of Theorem~\ref{thm_main}}\label{sec_pf_of_main_thm}

In this section, we prove Theorem~\ref{thm_main}.
First, we denote by  
\begin{align*}
\overline{u}_\la(x,t) := \int u(x+ \la y,t)\psi(y) \d y
\end{align*} 
the mollified velocity, where $\psi$ is defined as in Theorem~\ref{thm_local_reg_intro}. We now fix $(x,t)\in \RR^3 \times ((5\la )^{2s},T)$. We define the flow map $\Phi_t(x,\tau)$ corresponding to the mollified velocity $u_\la$, starting from a point $x$ at time $t$;
\begin{equation*}
\begin{cases}
\pa_\tau \Phi_t (x,\tau) 
= \overline{u}_\la(\Phi_t(x,\tau), \tau),\\
\Phi_t(x, t) = x, 
\end{cases}\qquad \tau \in (t-5^{2s},t].
\end{equation*}
The flow map $\Phi_t (x,\tau )$ is well defined (since $u_\la$ is smooth in space, uniformly in time) and $|\det D\Phi_t(\cdot,\tau)|=1$ at each time $\tau \in(t-5^{2s},t]$ (as $\div u_\la =0$). We define $ v_\la, q_\la$ by applying the Galilean transformation 
\begin{align*}
v_\la(z, \tau) &\coloneqq  \la^{2s-1} u(\Phi_t(x,t + \la^{2s}\tau) + \la z, t + \la^{2s}\tau) - \la^{2s-1} \overline{u}_\la(\Phi_t(x,t + \la^{2s}\tau), t+\la^{2s}\tau),\\
q_\la(z,w, \tau)
&\coloneqq \la^{4s-2} p(\Phi_t(x,t + \la^{2s}\tau) + \la z, t + \la^{2s}\tau) + \la^{2s-1} z \pa_{\tau}(\overline{u}_\la(\Phi_t(x,t + \la^{2s}\tau), t+\la^{2s}\tau))
\end{align*}
and we define the extension 
\begin{align*}
v_\la^*(z,w, \tau)
&\coloneqq \la^{2s-1} u^*(\Phi_t(x,t + \la^{2s}\tau) + \la z, \la w, t + \la^{2s}\tau)\\
&\quad-\la^{2s-1}\int u^*(\Phi_t(x,t + \la^{2s}\tau) + \la z', \la  w, t + \la^{2s}\tau)\psi(z') \d z',
\end{align*}
where $u^*$ is the Caffarelli-Silvestre's extension of $u$ (recall \eqref{def_of_caff_silv_ext}). 
By the Galilean invariance of \eqref{fNS} we can easily see that $(v_\la, q_\la)$ also solves \eqref{fNS} on $\R^3\times (-5^{2s},0)$. Indeed, the purpose of the construction $(v_\la, q_\la)$ comes from the mean-zero property of $v_\la$;
\begin{align*}
\int v_\la(z,\tau) \psi(z) \d z =0, \quad \forall \tau\in (-5^{2s},0).
\end{align*}
Now, we set 
\eqnb\label{def_of_F_G}
\begin{split}
F(x,t) &\coloneqq ( \mathcal{M} |\La^s u(x,t) |^{\frac 2{1+\de}})^{1+\de}
+|\La^{2s-1}\na p(x,t)|  +|\mathscr{M}_4 (\La^{2s-1}\na p )(x,t)|\\
G(x,y,t) &\coloneqq y^a |\bna u^* (x,y,t)|^2 
\end{split}
\eqne
(recall $\delta \coloneqq 2s/(6-s)$), and 
\[\begin{split}
H^\la(x,t) \coloneqq \int_{Q_{5\la}(0,t)} 
 F(\Phi_t(x,\tau)+z, \tau) \d z \,\d \tau 
+
 \int_{Q^*_{5\la}(0,t)}  
G(\Phi_t(x,\tau)+z, y, \tau) \d z \,\d y\, \d \tau& \\
+\int_{t-(5\la)^{2s}}^t \int_{B_{5\la}} \int_{B_{5\la}} 
\frac{|u(\Phi_t(x,\tau) + z, \tau) -u(\Phi_t(x,\tau) + z', \tau)|^2}{|z-z'|^{3+2s}} \d z \,\d z'\, d\tau.
\end{split}\]
For each $t\in ((5\la)^{2s},T)$ we define 
\eqnb\label{condition.Om}
\Om_{\ep}^\la(t) \coloneqq \{ x\in \R^3 \colon H^\la(x,t) \leq \ep\la^{5-4s} \},
\eqne
where $\ep$ is a sufficiently small constant given by Theorem~\ref{thm_local_reg_intro}. By a simple change of the variables, we obtain the following lemma.

\begin{lemma}\label{lem:connection} Given $\la>0$ and $t\in ((5\la)^{2s},T)$, let $x\in \Om_{\ep}^\la(t)$. Then $(v_{\la}, q_{\la}) = (v_{\la,x,t}, p_{\la,x,t})$ and the extension $v^*_\la$ satisfy \eqref{smallness_intro};
\begin{equation}\label{smallness.new}\begin{split}
&\int_{Q_5^*} w^a |\bna v^*_\la|^2 \d Z\d\tau  + \int_{-5^{2s}}^0 \int_{B_5}\int_{B_5} 
\frac{|v_\la(z,\tau)-v_\la(z',\tau)|^2}{|z-z'|^{3+2s}}\d z \d z' \d \tau\\
&\hspace{1cm}+ \int_{Q_5}   
(\mathcal{M} |\La^s v_\la |^{\frac 2{1+\de}})^{1+\de} +  |\La^{2s-1}\na q_\la|  +|\mathscr{M}_4 (\La^{2s-1}\na q_\la)| \d Z \d\tau \leq \ep.
\end{split}\end{equation}
\end{lemma}
\begin{proof}
The claim follows by applying the change of variables $(\la z, \la w, t+\la^{2s} \tau) \mapsto (z, w, \tau)$ to the left hand side of \eqref{smallness.new} and then using the definition \eqref{condition.Om}.
\end{proof}
As a consequence of Lemma~\ref{lem:connection} and Theorem~\ref{thm_local_reg_intro}, we obtain
\begin{align*}
\sup_{Q_\frac12} \left( |\na v_\la | + |\na^2 v_\la| \right) \leq C_0
\end{align*}
for some $C_0>0$. In particular, using $(\Phi_t(x,t+\la^{2s}\tau) + \la z, t+ \la^{2s}\tau)|_{(z,\tau)=(0,0)} = (\Phi_t(x,t),t) = (x,t)$, we have
\begin{align*}
\la^{2s}|\na u(x,t)|  + \la^{2s+1} |\na^2 u(x,t)| \leq C_0,
\end{align*}
so that 
\begin{equation}\label{inc.Omc}
\left\{x : |\na^2 u(x,t)| \geq \frac{C_0}{\la^{2s+1}} \right\}, \left\{x: |\na u(x,t)| \geq \frac{C_0}{\la^{2s}} \right\}
\subset (\Om_{\ep}^\la(t))^c.
\end{equation}
We now estimate $|(\Om_{\ep}^\la(t))^c|$.
\begin{lemma}\label{lem_measure.Omc} 
For given $\la>0$ and $t\in((5\la)^{2s},T)$, the set $\Om_\ep^\la(t)$ (defined by \eqref{condition.Om}) satisfies
\begin{align*}
|(\Om_{\ep}^\la(t))^c| \lec \frac{\la^{6s-2}}{\ep}  \dint_{t-(5\la)^{2s}}^t\int |\La^s u(x,\tau)|^2 \d x \,\d \tau.
\end{align*}
\end{lemma}

\begin{proof}
The claim follows from Chebyshev's inequality and the fact that
\begin{equation}\label{global}\begin{split}
&\int_{\R^4_+} y^a |\bna u^*|^2 \d X
+  \int\int \frac{|u(x,t)-u(x',t)|^2}{|x-x'|^{3+2s}}\d x \d x' \\
&\hspace{1cm}+ \int  
(\mathcal{M} |\La^s u |^{\frac 2{1+\de}})^{1+\de} +  |\La^{2s-1}\na p|  +|\mathscr{M}_4 (\La^{2s-1}\na p)| \d x  \lec \int |\La^s u(x,t)|^2 \d x,
\end{split}\end{equation}
which we verify below. Indeed, assuming \eqref{global}, 
\begin{align*}
&\int H^\la(x,t) \d x \\
&=\int  \left[
\int_{Q_{5\la}(0,t)} F(\Phi_t(x,\tau)+z, \tau) \d z \,\d \tau
+
\int_{Q^*_{5\la}(0,t)}  
G(\Phi_t(x,\tau)+z, y, \tau) \d z \,\d y\, \d \tau\right] \d x \\
&\quad+\int\int_{t-(5\la)^{2s}}^t \int_{B_{5\la}} \int_{B_{5\la}} 
\frac{|u(\Phi_t(x,\tau) + z, \tau) -u(\Phi_t(x,\tau) + z', \tau)|^2}{|z-z'|^{3+2s}} \d z \,\d z'\, d\tau \,\d x \\
&\leq 
\int_{Q_{5\la}(0,t)} \int
 F(x+z, \tau) \d x  \,\d z\, \d \tau 
+\int_{Q^*_{5\la}(0,t)} \int
G(x+z, y, \tau)\d x \,\d z\, \d y\, \d \tau \\
&\quad+\int_{t-(5\la)^{2s}}^t \int_{B_{5\la}} \int \int
\frac{|u(x + z,\tau) -u(x + z', \tau)|^2}{|z-z'|^{3+2s}} \d x\, \d z' \,\d z  \,d\tau \\
&=
|B_{5\la}|\left[\int_{t-(5\la)^{2s}}^t\left( \int
 F(x,\tau) \d x   
+\int_{\R^4_+}
G(x, y, \tau)\d X  
+
\iint
\frac{|u(x, \tau) -u(x',\tau)|^2}{|x-x'|^{3+2s}} \d x \d x'  \right)\d \tau\right] \\
&\lec  (5\la)^{3} \int_{t-(5\la)^{2s}}^t\int |\La^s u(x,\tau)|^2 \d x\, \d \tau.
\end{align*}
where the first inequality follows from Tonelli's theorem and the change of variables $\Phi_t(x,\tau) \mapsto x$, the second equality follows from the change of variables $x+z\mapsto x$ in the first two terms and $z'+x \mapsto x'$ and $x+z\mapsto x$ in the last term, and the last inequality follows from \eqref{global}. This together with Chebyshev's inequality give 
\begin{align*}
|(\Om_{\ep}^\la(t))^c|
&= |\{x  :  H^\la(x,t) \geq \ep \la^{5-4s} \}|\\
&\leq \frac {\la^{4s-5}}{\ep} 
\int H^\la(x,t) \d x 
\lec \frac{\la^{6s-2}}{\ep}\dint_{t-(5\la)^{2s}}^t  \int |\La^s u(x,\tau)|^2 \d x\, \d \tau ,
\end{align*}
as required. 

It remains to verify \eqref{global}. Using the equivalence between norms, \eqref{ext.norm} and \eqref{Sob.Slo.L2norm}, we have
\begin{align*}
\int_{\R^4_+} y^a |\bna u^*|^2 \d X
+ \int\int \frac{|u(x,t)-u(x',t)|^2}{|x-x'|^{3+2s}}\d x \,\d x'
\lec
 \int |\La^s u|^2 \d x 
\end{align*}
Since the maximal operator $\mathcal{M}$ is bounded on $L^q(\R^3)$-space for any $1<q<\infty$, we take $q\coloneqq 1+\delta $ to obtain
\begin{align*}
\int  
(\mathcal{M} |\La^s u |^{\frac 2{1+\de}})^{1+\de}\d x  \lec
\int   
(|\La^s u |^{\frac 2{1+\de}})^{1+\de}\d x 
=
\int   
|\La^s u |^2\d x.
\end{align*}
On the other hand, using the characterization \eqref{hardy_characterisation} and the definition of the Hardy norm, we have
\begin{align*} 
\int|\mathscr{M}_4(\La^{2s-1}\na p)| \d x 
\lec \norm{\La^{2s-1}\na p}_{1}  + \norm{\mathcal{R} (\La^{2s-1}\na p)}_{1}
\end{align*}
Furthermore, writing $\La^{2s-1}\na p = \mathcal{R} (-\De)^s p$ and applying \eqref{est_on_p_intro} (see Proposition~\ref{prop_integrability_of_pressure} below), we get
\begin{align}\label{pressure.bound}
\int   
|\La^{2s-1}\na p|
+|\mathscr{M}_4(\La^{2s-1}\na p)| \d x
\lec 
\sum_{n=1}^2 \norm{\mathcal{R}^n (-\De)^s p}_{1}
\lec \norm{\La^s u}_{2}^2,
\end{align}
which concludes the proof of \eqref{global}
\end{proof}

We now verify that Theorem~\ref{thm_main} follows from \eqref{inc.Omc} and Lemma~\ref{lem_measure.Omc}. For any $\la>0$ and $t\in( (5\la)^{2s},T)$, we have
\begin{align*}
\left|\left\{x\in K : |\na^2 u(x,t)| \geq \frac{C_0}{\la^{2s+1}} \right\}\right| 
&\leq |(\Om_{\ep}^\la(t))^c| 
\lec_s \la^{6s-2} \dint_{t-(5\la)^{2s}}^t\int |\La^s u(x,\tau)|^2 \d x\d \tau\\
&\lec \la^{6s-2} \mathcal{M}(\norm{\La^s u}_{2}^21_{(0,T)})(t)  
\end{align*}
and analogously
\begin{align*}
\left|\left\{x\in K : |\na u(x,t)| \geq \frac{C_0}{\la^{2s}} \right\}\right| 
 \lec \la^{6s-2}\mathcal{M}(\norm{\La^s u}_{2}^21_{(0,T)})(t),
\end{align*} 
where $1_E$ denotes the characteristic function of a set $E\subset \R$.  
In other words, by letting $R\coloneqq C_0 \la^{-2s-(n-1)}$ ($n=1,2$) we obtain
\[
\left|\left\{x\in K : |\na^n u(x,t)| \geq R \right\}\right| 
 \lec R^{-p} \mathcal{M}(\norm{\La^s u}_{L^2(\R^3)}^21_{(0,T)})(t) \qquad \text{ for } R^p > c_0  t^{-\frac{2(3s-1)}{2s }},
\]
where $n=1,2$ and $c_0\coloneqq  C_0^p5^{2(3s-1)}$ (recall that $p = \frac{2(3s-1)}{n+2s-1}$). Thus
\begin{align*}
R^p \left|\left\{x\in K : |\na^n u(x,t)| \geq R \right\}\right| \leq 
\begin{cases}
R^p |K| &R^p\leq c_0 t^{-\frac{2(3s-1)}{2s }},\\
C \mathcal{M}(\norm{\La^s u}_{2}^21_{(0,T)})(t)  , & R^p>c_0 t^{-\frac{2(3s-1)}{2s }}
\end{cases}
\end{align*}
for some constant $C>0$, and hence
\begin{align*}
\norm{\na^n u(\cdot, t)}_{L^{p,\infty}(K)}^p
&\lec \max\left(\frac{|K|}{t^{\frac{2(3s-1)}{2s }}}, \mathcal{M}(\norm{\La^s u}_{2}^21_{(0,T)})(t) 
\right)\\
&\leq  \mathcal{M}(\norm{\La^s u}_{2}^21_{(0,T)})(t) 
+\frac{|K|}{t^{3-\frac1s}}.
\end{align*}
Including time dependence this gives
\begin{align*}
\norm{\na^n u}_{L^{p,\infty}(((t_0,T)\times K)}^p
&\lec \norm{  \mathcal{M}(\norm{\La^s u}_{2}^21_{(0,T)})}_{L_t^{1,\infty}(\R)} + \frac{|K|}{t_0^{2-\frac 1s}} 
\lec \norm{\La^s u}_{L^2((0,T)\times \R^3)}^2 
+ \frac{|K|}{t_0^{2-\frac 1s}}\\
&\lec \norm{u_0}_{2}+ \frac{|K|}{t_0^{2-\frac 1s}},
\end{align*}
as required, where we used the fact that the maximal operator $\mathcal{M}:  L^1(\R) \to L^{1,\infty}(\R)$ is bounded in the second inequality, as well as the energy inequality \eqref{EI} in the last step.

\section{Global integrability of the pressure}\label{sec_global_p}
In this section we show (in Proposition~\ref{prop_integrability_of_pressure}) that for each integer $n\geq 0$
\[
\norm{\mathcal{R}^n(-\De)^s p}_{1}
\leq c_{s,n} \norm{\Lambda^s u}_{2}^2,
\]
which we have used above in \eqref{pressure.bound}.

We first introduce some decay estimate, required in the proof of Proposition~\ref{prop_integrability_of_pressure}.

\begin{lemma}\label{lem_pointwise_Dlk}
Let $g\in \mathcal{S}(\R^3)$. Then for any $n\in \mathbb{N}\cup\{0\}$ and $s>0$,
\[
|\mathcal{R}^n \Lambda^{2s} g (x)| \lec_{g,s,\eta,n} \frac 1{(1+|x|)^{3+\eta}}
\]
for every $\eta \in (0,2s)$, $x\in \RR^3$, where $\mathcal{R}$ is the Riesz transform defined by $\widehat{\mathcal{R} g}(\xi) = \frac{-i\xi}{|\xi|} \hat{g}(\xi)$. 
\end{lemma}

\begin{proof}
The case $n=0$ follows from the pointwise decay estimate proved in \cite[Lemma 1]{grafakos_oh},
\begin{align}\label{decay.f}
|\Lambda^{2s} g(x)|\lec_{g,s} \frac 1{(1+|x|)^{3+2s}}.
\end{align} 
For $n\geq 1$ the proof follows by induction: we fix $n\in \mathbb{N}$ and $\eta\in (0,2s)$, and define $\eta_i := 2s - \frac{(2s-\eta)i}n$ for $i=0,1, \cdots, n$. In particular, $\eta_0 =2s$, $\eta_n = \eta$, and $\eta_{i+1} <\eta_i$. We claim that for each $i$, we have
\begin{align}\label{ind.g}
|\mathcal{R}^i \La^{2s}g(x)| \lec_{g,s,i} \frac 1{(1+|x|)^{3+\eta_i}}
\end{align}
The base step (when $i=0$) holds true by \eqref{decay.f}. 
Assume \eqref{ind.g} holds for $i$. We will write 
\[
f\coloneqq \mathcal{R}^i \La^{2s} g
\]
for brevity. We first recall the integral expression of the Riesz transform of $f$,
\begin{align}\label{Riesz}
\mathcal{R} f (x) \sim \text{p.v.} \int f(y) \frac{x-y}{|x-y|^4} \d y.
\end{align}
In order to estimate $|\mathcal{R}f(x)|$, we split the integral region in \eqref{Riesz} into three parts by writing 
\[\int = \int_{|y|\leq |x|/2} + \int_{|y-x|\leq |x|/2} + \int_{\substack{|y|\geq |x|/2 \\ |y-x| \geq |x|/2}}.
\]
As for the first part,
\[\begin{split}
\left| \text{p.v.}\int_{|y|\leq |x|/2}f(y)\frac{x-y}{|x-y|^4} \d y  \right|
&\leq \left| \text{p.v.}\int_{|y|\leq |x|/2} f(y)\left(\frac{x-y}{|x-y|^4} - \frac{x}{|x|^4}\right) \d y \right| +\frac{1}{|x|^3} \int_{|y|\leq |x|/2} |f(y)| \d y\\ 
&\lec \frac{1}{|x|^3} \int_{|y|\leq |x|/2} |f(y)| \d y\\
  &\lec_{g,s,i} \frac{1}{|x|^{3}} \int_{|y|\leq |x|/2} \frac1{(1+|y|)^{3+\eta_i } }  \d y
   \lec_i \frac1{|x|^{3+\eta_{i+1}}},
\end{split}\]
where we used the fact that 
\begin{align*}
\left|\frac{x-y}{|x-y|^4} - \frac{x}{|x|^4}\right|
\lec \int_0^1 \frac{|y|}{|x-\th y|^4} \d \th
\lec \frac {|y|}{|x|^4} \lec \frac 1{|x|^3} \quad \text{ for } |y|\leq \frac {|x|}2
\end{align*} 
(since $|x-\th y|\geq |x|-|y|\geq \frac {|x|}2$) in the second inequality, and \eqref{ind.g} in the last line.

As for the second part,
\[
\begin{split}
\left| \int_{\substack{|y|\geq |x|/2 \\ |y-x| \geq |x|/2}} f(y)\frac{x_i-y_i}{|x-y|^4} \d y \right| &\lec \frac{1}{|x|^3}  \int_{\substack{|y|\geq |x|/2 \\ |y-x| \geq |x|/2}} |f(y)| \d y\\
&\lec \frac{1}{|x|^{3+\eta_{i+1} }}  \int |f(y)| |y|^{\eta_{i+1}} \d y
\lec_i \frac{1}{|x|^{3+\eta_{i+1} }}.
\end{split}
\]
As for the last part,
\[\begin{split}
&\left|\text{p.v} \int_{|y-x|\leq |x|/2}f(y)\frac{x_i-y_i}{|x-y|^4} \d y \right| = \left| \text{p.v}\int_{ |y-x|\leq |x|/2} (f(x)-f(y))\frac{x_i-y_i}{|x-y|^4} \d y \right| \\
&\qquad \leq \int_{|y-x|\leq |x|/2} \frac{|f(x)-f(y)|^\th }{|x-y|^\th}  \frac {(|f(x)|+|f(y)|)^{1-\th }}{|x-y|^{(3-\th )}}  \d y \\
&\qquad\lec\frac{\| \nabla f \|_\infty^\th}{ |x|^{(3+\eta_i)(1-\th )}} \int_{ |y-x|\leq |x|/2} |x-y|^{-(3-\th ) }  \d y \\
&\qquad \lec_\th \frac 1{|x|^{(3+\eta_i)(1-\th ) -\th}} = \frac1{|x|^{3+\eta_{i+1}}} ,
\end{split}\]
where $\th \coloneqq \frac{2s-\eta}{n(2+3\eta_i)}\in (0,1)$. Indeed, the third line follows from the inductive assumption and the fact that $|y|\geq |x|-|x-y|\geq |x|/2$, and the last line is obtained by noting that
\begin{align*}
\norm{\na f}_\infty
\leq \int |\xi|^{1+2s} |\hat{g}(\xi)| d \xi
\lec \int_{|\xi|\leq 1}  |\xi|^{1+2s}  |\hat{g}(\xi)| d \xi 
+ \int_{|\xi|\geq 1} \frac{|\xi|^{1+2s}}{(1+|\xi|)^{7}} (1+|\xi|)^{7}|\hat{g}(\xi)| d\xi \lec_g 1. 
\end{align*}
Therefore, since a similar calculation gives $\mathcal{R}f \in L^\infty$, the decay estimate $\sim |x|^{3+\eta_{i+1}} $ for each of the parts above give \eqref{ind.g} for $i+1$. 
\end{proof}

\begin{prop}[Global estimate on $(-\Delta)^sp$]\label{prop_integrability_of_pressure} Let $s\in (0,1)$. Suppose that $f$ and $g$ are divergence-free. Then, for any $n\in \mathbb{N}\cup\{0\}$, it satisfies
\[
\norm{\mathcal{R}^n {\Lambda^{2s-2}}\div \div (f \otimes g) }_{1} \lec_{s,n} \norm{\Lambda^s f}_{2}\norm{\Lambda^s g}_{2}
\]
where $\mathcal{R}$ is the Riesz transform defined by $\widehat{\mathcal{R}f}(\xi) = \frac{-i\xi}{|\xi|}\hat{f}(\xi)$. 
In particular, if $p$ is a solution  to $(-\De) p = \div \div (u\otimes u)$, we have a global bound of $(-\Delta)^sp$, 
\[
\norm{\mathcal{R}^n(-\De)^s p}_{1}
\lec_{s,n} \norm{\Lambda^s u}_{2}^2.
\]
\end{prop}

\begin{proof}
Note that when $F$ and $G$ are divergence-free, we have
\begin{align*}
\sum_{l,k=1}^3 \La^{2s-2}\pa_{lk} (F^l G^k) 
= \sum_{l,k=1}^3 [\La^{2s-2}\pa_{lk}, F^l] G^k 
=\sum_{l,k=1}^3 [\La^{2s-2}\pa_{lk}, G^k] F^l  
\end{align*}
where $[A,B]=AB - BA$ and $F^l$ and $G^k$ are $l^{\text{th}}$ and $k^{\text{th}}$ components of vector functions $F$ and $G$, respectively. Using this together with Bony's paraproduct decomposition, we get
\[
{\Lambda^{2s-2}\div\div} (f\otimes g) = \sum_{l,k=1}^3\sum_{j\in  \Z} \Lambda^{2s-2}\p_{lk}  f_j^l \tilde{g}_j^k + [\Lambda^{2s-2}\p_{lk}, f_{\leq j-3}^l] g_j^k + [\Lambda^{2s-2}\p_{lk}, g_{\leq j-3}^k] f_{j}^l  ,
\]
where $h_j \coloneqq P_{j} h$, $ h_{\leq j }\coloneqq P_{\leq 2^j} h $, and
$\tilde{h}_j \coloneqq \td P_j h=\sum_{k=j-2}^{j+2} P_{2^j} h$. For convenience, we drop the indices $l$ and $k$ in $f^l$ and $g^k$. \vspace{0.3cm} \\
\texttt{Step 1.} We estimate the diagonal piece, that is we show that
\begin{align}\label{diag}
 \| \mathcal{R}^n \sum_{l,k}\sum_j \Lambda^{2s-2}{\p_{lk}} (f_j\tilde{g}_j ) \|_{1}  \lec \| \Lambda^s f \|_2\| \Lambda^s g \|_2.
\end{align}

We consider the case $n=0$ first.  
Let $\chi \in C_0^\infty (B_{2^{5}})$ be such that $\chi = 1 $ on $B_{2^4}$. We have from Fourier series expansion
\eqnb\label{def_of_chi_ms}
\frac{z_k z_l}{|z|^{2-2s}} \chi (z) = \sum_{m\in \ZZ^3} \chi_{m,s} e^{i \pi m\cdot \frac{z}{2^5}},
\eqne
where $\chi_{m,s} = \chi_{m,s}^{l,k}$ satisfies
\begin{align}\label{est.chims}
|\chi_{m,s}| \sim \left|\int \frac{z_l z_k}{|z|^{2-2s}} \chi(z) e^{-i \pi m \cdot z /2^5} \d z\right| \sim \left| \,(\Lambda^{2s-2}{\p_{lk}}   \widehat{\chi})\left(\frac{\pi m}{2^5}\right)\right| \lec_s(1+|m|)^{-3-s}
\end{align}
where the last inequality follows from Lemma~\ref{lem_pointwise_Dlk} (applied with $\eta \coloneqq s$).
Then for each $j\in \Z$
\begin{align}
-\Lambda^{2s-2}{\p_{lk}}  ( f_j \tilde{g}_j ) (x) &= \frac{1}{(2\pi )^{6}} \iint  \left( \frac{(\xi_l + \eta_l )(\xi_k + \eta_k)}{|\xi + \eta |^{2-2s}}  \right)  \widehat{f}_j (\xi ) \widehat{\tilde{g}}_j (\eta ) e^{i(\xi + \eta ) \cdot x} \d \xi \d \eta \label{dig.eq}\\
&= \frac{2^{2js}}{(2\pi )^{6}} \iint  \left( 2^{-2js} \frac{(\xi_l + \eta_l )(\xi_k + \eta_k)}{|\xi + \eta |^{2-2s}}  \right) \chi (2^{-j} (\xi +\eta )) \widehat{f}_j (\xi ) \widehat{\tilde{g}}_j (\eta ) e^{i(\xi + \eta ) \cdot x} \d \xi \d \eta  \nonumber\\
&= \frac{2^{2js}}{(2\pi )^{6}} \sum_{m\in \ZZ^3} \chi_{m,s} \int   \widehat{f}_j (\xi )  e^{i\xi  \cdot x} e^{i\pi m \cdot \frac{2^{-j} \xi}{2^5} } \d \xi \int\widehat{\tilde{g}}_j (\eta )  e^{i \eta \cdot x} e^{i\pi m \cdot \frac{2^{-j}  \eta }{2^5} } \d \eta  \nonumber\\
&= 2^{2js} \sum_{m\in \ZZ^3} \chi_{m,s} P_j^m f \tilde{P}_j^m g, \nonumber
\end{align}
where we set $\widehat{P_j^m f }(\xi) \coloneqq \varrho( 2^{-j} \xi ) \hat{f} (\xi )e^{i\pi m \cdot \frac{2^{-j} \xi }{2^5} }$ ($\varrho$ is defined as in \eqref{varrho}). Indeed, the second line follows from the facts that $|\xi | \leq 2^{j+1}$ and $|\eta | \leq 2^{j+3}$ (so that $|\xi + \eta | \leq 2^{j+1} + 2^{j+3} \leq 2^{j+4}$, which gives $\chi (2^{-j} (\xi + \eta ))=1$). Therefore we have
\begin{align*}
\| \sum_j \Lambda^{2s-2}{\p_{lk}} (f_j\tilde{g}_j ) \|_1 &\leq \sum_{m\in \ZZ^3} |\chi_{m,s} |\,  \| (2^{js} P_j^m f )_{l^2_j} \|_2\| (2^{js} \tilde{P}_j^m g )_{l^2_j} \|_2\\
&\lec_s   \| \Lambda^s f \|_2  \| \Lambda^s g \|_2  \sum_{m\in \ZZ} \frac{\log^2 (10+|m|)}{(1+|m|)^{3+s} }
\lec_s  \| \Lambda^s f \|_2\| \Lambda^s g \|_2 ,
\end{align*}
where the second line follows from \eqref{est.chims} and the fact that $P_j^m$ is bounded on $L^2 (\RR^3 , l^2)$ with constant $C \log (10+|m|)$ (see \cite{li}
).

The estimate \eqref{diag} for $n\in \mathbb{N}$ follows from the same argument as above. Indeed, with additional Riesz transform $\mathcal{R}^n$ the definition \eqref{def_of_chi_ms} of $\chi_{m,s}$ will include additional factor of $(-iz |z|^{-1})^{\otimes n}$, which will then also appear in the integrands in \eqref{dig.eq} and \eqref{est.chims}. However, since Lemma~\ref{lem_pointwise_Dlk} gives the same bound $(1+|m|)^{-3-s}$ for $|\chi_{m,s}|$ up to constant multiple,  \eqref{diag} follows in the same way for all $n\in \mathbb{N}$. \\

\noindent\texttt{Step 2.} We estimate the low-high and high-low pieces, by showing that 
\begin{align}\label{low.high}
\|\sum_j  \mathcal{R}^n [\Lambda^{2s-2}{\p_{lk}}, f_{\leq j-3}] g_j \| _{1} \lec_s \| \Lambda^s f \|_{2}\| \Lambda^s g \|_{2}.
\end{align}
(Then one can obtain the same bound for the high-low piece.) This together with \texttt{Step 1} proves the lemma.\\

Consider the case $n=0,1$ first. For each $j\in  \Z$, we have 
\[\begin{split}
|\left[   \Lambda^{2s-2}{\p_{lk}} , f_{\leq j-3} \right] g_j (x)| 
& \leq \ 2^{j(3+2s)}  \int |\psi ( 2^j y) (f_{\leq j-3}(x-y) - f_{\leq j-3} (x) ) g_j (x-y) |\d y\\
&\leq 2^{j(3+2s)}  \int | \psi ( 2^j y)|   \max_{\theta \in [0,1]} |\nabla f_{\leq j-3} (x+\theta y )|\, |y| \,    |g_j (x-y)| \d y\\
&\lec 2^{j(3+2s)}  \mathcal{M} (\nabla f_{\leq j-3}  ) (x)\mathcal{M} (g_j  ) (x)  \int | \psi ( 2^j y)| |y| (1+2^j |y| )^6   \d y\\
&\lec 2^{j(-1+2s)}  \mathcal{M} (\nabla f_{\leq j-3}  ) (x)\mathcal{M} (g_j  ) (x)
\end{split}\]
where $\widehat{\psi } (\xi ) \coloneqq \xi_l \xi_k |\xi |^{2s-2} \td\varrho (\xi )$ (here $\td\varrho (\xi )\coloneqq \sum_{|j|\leq3} \varrho(2^{-j} \xi)$ and $\varrho$ is defined as in \eqref{varrho}) and we used Lemma~\ref{lem.max} in the third line.

Therefore, using the characterization \eqref{hardy.norm.1} of the Hardy norm, we obtain 
\[\begin{split}
\| \sum_j[\Lambda^{2s-2}{\p_{lk}},  f_{\leq j-3} ]g_j \|_{\mathcal{H}^1} 
&\sim  \| (P_m \sum_j[\Lambda^{2s-2}{\p_{lk}},  f_{\leq j-3} ]g_j  )_{l^2_m} \|_{1} \\
& =  \|{ ( P_m \sum_{j=m-2}^{m+2}\td P_j [\Lambda^{2s-2}{\p_{lk}},  f_{\leq j-3} ]g_j }\|_1
\lec\|{ \left(  [\Lambda^{2s-2}{\p_{lk}},  f_{\leq j-3} ]g_j   \right)_{l^2_j}}\|_1\\
&\lec \| \ \mathcal{M} ( { \sup_j}(2^{j(-1+s)} |\nabla f_{\leq j-3 } | ) )\|_2 \| (\mathcal{M} (2^{js} g_j ))_{l^2_j} \|_2 \\
&\lec \| \sup_j (2^{j(-1+s)} |\nabla f_{\leq j-3 } | ) \|_2 \| \Lambda^s g \|_2 \lec \| \Lambda^s f \|_2 \| \Lambda^s g \|_2 \\
&\lec \| \Lambda^s u \|_2^2,
\end{split}\]
where the fifth line follows from $\norm{(\mathcal{M}(2^{js} g_j))_{l^2_j}}_2 
=(\norm{\mathcal{M}(2^{js} g_j)}_2)_{l^2_j}
\lec   \norm{(2^{js} g_j)_{l^2_j}}_2 \lec \norm{\La^s g}_2$, and the second last inequality follows from
\begin{align*}
|2^{j(s-1)} \na f_{\leq j-3}(x)| &=   2^{j(s-1)}|\na \La^{-s}  P_{\leq j-3} \La^s f(x)| 
 \leq 2^{j(s-1)}\sum_{k\leq j-3}  |\na \La^{-s}  P_{ k} \La^s f(x)|\\
&=  2^{j(s-1)}\sum_{k\leq j-3} 2^{k(1-s)}|\varrho_k\ast ( \La^s f)(x)| \lec  2^{j(s-1)}\sum_{k\leq j-3} 2^{k(1-s)} \mathcal{M}(\La^s f;\varrho ) (x) \\
&\lec \mathcal{M}(\La^s f)(x)
\end{align*}
where $\varrho_k (z) \coloneqq 2^{-3k} \varrho (2^{-k}x)$, $\widehat{\varrho  }(\xi ) \coloneqq  \xi |\xi |^{-s} (\rho (\xi)-\rho (2\xi))$, and we have used \eqref{smooth_max_bdd_by_max} in the last line.
It then follows that $\norm{(2^{j(s-1)} \na f_{\leq j-3})_{l^\infty}}_2 \lec \norm{\mathcal{M}(\La^s f)}_2 \lec\norm{\La^s f}_2$, as required. \\
In the case of the integer $n\geq 2$, we can obtain the same estimate \eqref{low.high} simply by modifying the definition of $\widehat{\psi}$ to include the additional factor of $(-i\xi |\xi|^{-1})^{\otimes n}$. Then we still have $ \psi \in \mathcal{S}$, and so the rest of the calculation remains the same.  \end{proof}

\section{Local study}\label{sec_local_study}
In this section, we prove Theorem~\ref{thm_local_reg_intro}, which gives a local regularity condition in terms of scale-optimal quantities for a suitable weak solution of \eqref{fNS} with zero $\psi$-mean. For the reader's convenience, we restate Theorem~\ref{thm_local_reg_intro}. 
\begin{thm}[Local regularity]\label{thm_local_reg}Let $s\in (\frac34,1)$. There exists $\ep=\ep (s, \psi)>0$ such that if $(u,p )$ is a suitable weak solution of \eqref{fNS} such that $\int u(x,t)\psi(x) \d x =0$ for all $t\in (-5^{2s},0) $ and
\begin{equation}\begin{split}
\label{smallness}
&\int_{Q_5^*} y^a |\bna u^*|^2 \d x\,\d y\, \d t  + \int_{-5^{2s}}^0 \int_{B_5}\int_{B_5} \frac{|u(x,t)-u(y,t)|^2}{|x-y|^{3+2s}}\d x \,\d y\, \d t\\
&\hspace{2cm}+ \int_{Q_5}   
\left( (\mathcal{M} |\La^s u |^{\frac 2{1+\de}})^{1+\de} +  |\La^{2s-1}\na p|  +|\mathscr{M}_4 (\La^{2s-1}\na p)|\right) \,\d x\, \d t \leq \ep,
\end{split}\end{equation}
where $\delta \coloneqq \frac{2s}{6-s}$, then 
\[
\sup_{Q_\frac12} (|u|+ |\na u|+|\nabla^2  u|)  \leq C_0
\]
for some positive constant $C_0=C_0(s)$.
\end{thm}
In the statement, $\ep>0$ is determined by Proposition~\ref{prop:red}.

\begin{rem} The conclusion of the theorem could be easily extended to the boundedness of $|\na^k u|$ on $Q_{\frac12}$ for any $k\geq 0$, see the comment below Theorem~\ref{thm_bootstrapping_intro}. Indeed, the proof of Theorem~\ref{thm_local_reg} is based on a bootstrapping argument that could be continued for higher derivatives. However, we only cover the case $k=2$ for the purposes of our main a priori bound, Theorem~\ref{thm_main}. 
\end{rem}
\subsection{A Poincar\'e-type inequality for the pressure function}\label{sec_poincare_for_p}

In this subsection, we discuss a Poincar\'e-type inequality, which will be used to estimate the pressure (applied with $g:=\na p$) in the proof of Theorem~\ref{thm_local_reg}. 
\begin{lemma}[Poincar\'e-type inequality]\label{lem_poincare_for_p} For $s\in (\frac34,1)$, $\psi\in C_c^\infty(B_1)$ satisfying $\int \psi (x)\d x =1$, we have  
\[
\norm{g - (g)_\psi}_{L^{\frac65}(B_\frac54)} 
\lec_{s,\psi}  \norm{\La^{2s-1}g}_{L^1(B_5)} 
+ \norm{ \mathcal{M}(\La^{2s-1} g ; \eta ) }_{L^1(B_3)} 
\]
for some $\eta\in C_c^\infty(\R^3)$, where $(g)_\psi = \int g\psi \d y$. In particular
\begin{equation}\label{poin.p}
\norm{g - (g)_\psi}_{L^{\frac65}(B_\frac54)} 
\lec_{s,\psi} \norm{\La^{2s-1}g}_{L^1(B_5)} 
+ \norm{ \mathscr{M}_4(\La^{2s-1}g) }_{L^1(B_3)} 
\end{equation}
\end{lemma}

\begin{rem}
The oscillation $g - (g)_\psi$ can be also controlled by the Hardy-Littlewood maximal function of $\La^{2s-2}\div g$, 
\[
\| g - (g)_\psi \|_{L^{6/5} (B_1)} \lec_{s,\psi} \| \mathcal{M} (\Lambda^{2s-1} g) \|_{L^1 (B_1)},
\]
which can be proved directly using the approach from Lemma~3 in \cite{vasseur} (and also follows directly from the lemma above and \eqref{smooth_max_bdd_by_max}). However, the maximal operator $\mathcal{M}:\mathcal{H}^1(\R^3)\to L^1(\R^3)$ is not bounded and hence we have no global bound for $\norm{\mathcal{M} (-\Delta)^s p}_1$. To get around this issue, we introduce the grand maximal function in the above lemma. 
\end{rem}
\begin{proof}
Using the fact that $\int\psi=1$ and the representation 
\begin{align*}
g(x)= \La^{1-2s} \La^{2s-1}g(x) \sim \int \frac { \La^{2s-1}g(z)}{|x-z|^{4-2s}}\, \d z,
\end{align*}
we decompose the oscillation into two parts, 
\begin{align*}
g(x)-(g)_\psi 
&= \int (g(x)-g(y)) \psi(y)\, \d y 
\sim \iint \left( \frac 1{|x-z|^{4-2s}} - \frac 1{|y-z|^{4-2s}} \right)\La^{2s-1} g (z) \psi (y) \,\d z\,\d y \\
&= \iint_{|z|\leq \frac32 }  \left( \frac 1{|x-z|^{4-2s}} - \frac 1{|y-z|^{4-2s}} \right)\La^{2s-1} g (z) \psi (y)\, \d z\,\d y \\
&\qquad +\iint_{|z|>\frac32}\int_0^1 \frac{(\th x +(1-\th) y - z)\cdot (x-y)}{|\th x +(1-\th) y - z|^{6-2s}} \,\d \th \,\La^{2s-1} g (z) \psi (y)\,\d z\, \d y \\
&=: I_1(x) + I_2(x).
\end{align*}
To estimate $I_1(x)$, we have
\begin{align*}
|I_1(x)|
&\leq \int_{|z|\leq \frac32} \frac {|\La^{2s-1}g(z)|}{|x-z|^{4-2s}} \d z + \iint_{|z|\leq \frac32} \frac {|\La^{2s-1}g(z)|}{|y-z|^{4-2s}} \d z |\psi(y)|\d y
\end{align*}
and we apply Young's inequality for convolutions to get
\begin{align*}
\norm{I_1}_{L^\frac65(B_\frac54)} 
\lec \norm{\La^{2s-1} g}_{L^1(B_2)} \left( \norm{|\cdot|^{-(4-2s)}1_{B_3}}_{\frac65} + \norm{|\cdot|^{-(4-2s)}1_{B_3}}_{1} \right) \lec \norm{\La^{2s-1} g}_{L^1(B_2)},
\end{align*}
where $1_E$ denotes the characteristic function of a set $E$, and we used the fact that $|x-z|\leq |x|+|z|<3$ (and similarly $|y-z|\leq 3$) in the first inequality, as well as the fact that $s>\frac34$ in the last inequality. 

As for $I_2$, we let $\chi \in C_c^\infty (B_\frac78)$ be such that $\chi=1$ on $B_{\frac34}$, we set $\overline{\eta}(z)\coloneqq \chi(z) - \chi (2z )$ and 
\[
\eta_j (z)\coloneqq 2^{j(2-2s) }\frac{z}{|z|^{6-2s}} \overline{\eta }(2^{-j }z)\qquad \text{ for }j\in \ZZ.\]
By construction $\norm{ \eta_j}_1 = \norm{|\cdot|^{-6+2s}\bar{\eta}}_1= c$, $\supp \, \eta_j \subset B_{\frac78 \cdot 2^{j} }\setminus B_{\frac 38 2^{j} }$ for every $j\in \ZZ$. Moreover
\eqnb\label{defn.eta}
\frac{z}{|z|^{6-2s}}
= \sum_{j\geq -1 } 2^{-j(2-2s)}  \eta_j(z)\hspace{1cm}\text{ for }|z|\geq \frac14,
\eqne
since $\eta_j(z)=0$ for such $z$ and $j< -1$. Since $|\theta x+(1-\theta )y -z|\geq |z| - \theta |x| - (1-\theta )|y| \geq \frac32 -\frac54 =\frac14$ in the definition of $I_2$ for every $x\in B_\frac54$, we have
\begin{align*}
I_2 (x)
&= \sum_{j\geq -1} 2^{-j(2-2s)} \int_0^1\iint_{|z|>\frac32} 
\eta_j(\th x +(1-\th) y - z) \La^{2s-1} g(z) \d z \cdot (x-y) \psi(y) \,\d y\, \d\th\\
&= \sum_{j\geq 3} 2^{-j(2-2s)} \int_0^1\iint 
\eta_j(\th x +(1-\th) y - z) \La^{2s-1} g(z) \d z \cdot (x-y) \psi(y) \, \d y \, \d\th\\
&\quad + \sum_{j=-1}^2 
2^{-j(2-2s)} \int_0^1\iint_{|z|>\frac32} 
\eta_j(\th x +(1-\th) y - z) \La^{2s-1} g(z) \d z \cdot (x-y) \psi(y) \, \d y\, \d\th\\
&= I_{21}(x)  + I_{22}(x).
\end{align*}
This decomposition allows us to drop ``$|z|>3/2$'' in $I_{21}(x)$ because we can assume $ |\th x + (1-\th) y -z|\geq 3$ (as $j\geq 3$) and hence
$|z|\geq |\th x + (1-\th) y -z| - \th |x| -(1-\th)|y| >3-\frac54> \frac32$.

We consider $I_{22}$ first. Since $|\th x + (1-\th) y -z|\leq \frac72$ (as $j\leq 2$) we have $|z|\leq |\th x + (1-\th) y -z| + \theta |x|+(1-\theta )|y|\leq 5$ whenever $|x|\leq \frac54$, $|y|\leq 1$, $\th\in [0,1]$, and so
\begin{align*}
|I_{22}(x)|
&\lec \sum_{j=-1}^2 \int_0^1 \int_{|y|\leq1}\int 
|\eta_j(\th x +(1-\th) y -z)| |\La^{2s-1}g(z)| 1_{B_5}(z)\,\d z\, \d y\,\d \th
\lec \| \La^{2s-1}g \|_{L^1(B_5)}
\end{align*} 
for every $x\in B_\frac54$, where we used the fact that $\sum_{j=-1}^2|\eta_j |\lec 1$ in the last inequality. Therefore 
\[
\| I_{22} \|_{L^1(B_\frac54)} \lec \| \La^{2s-1}g \|_{L^1(B_5)}.
\]

As for $I_{21}$, we write $G\coloneqq \sup_{j\geq 3} |\eta_j\ast \La^{2s-1}g| $ (for brevity) to get
\begin{align*}
|I_{21}(x)|
&= \left| \sum_{j\geq 3} 2^{-j(2-2s)}
\int_0^1
\int (\eta_j\ast \La^{2s-1} g )(\th x+ (1-\th)y) \cdot (x-y) \psi(y) \, \d y\, \d \th \right| \\
&\lec \int_0^1
\int_{B_{1}} |G(\th x+ (1-\th)y) |\, \d y\, \d \th  
= \left( \int_0^\frac12 +\int_\frac12^1 \right) \int_{B_{1}}|G(\th x+ (1-\th)y) |\, \d y\, \d \th  \\
&=:  I_{211}(x) + I_{212}(x),
\end{align*}
for $x\in B_\frac54$, where we used the facts that $s<1$, $|\psi|\lec 1$ and $|x-y|\leq |x|+|y|\lec 1$ in the second line. 

The estimate for $I_{211}$ follows from the change of variable $ y\mapsto\theta x + (1-\th)y=: y' $,
\begin{align}\label{I211.est}
\norm{I_{211}}_{L^1(B_\frac54)}&= \int_{B_1}\int_0^\frac12 
\int_{\theta x+(1-\th )B_1} |G(y')|\d y' \frac{\d \th}{(1-\th)^3} \d x
\lec  \| G \|_{L^1 (B_{2})} 
\end{align}

On the other hand, $I_{212}$ can be estimated by
\begin{align*}
\norm{I_{212}}_{L^\frac65(B_\frac54)}^{\frac65}
&= \int_{B_\frac54} \left[ \int_\frac12^1 \int_{B_1} |G(\theta x + (1-\theta )y)| \,\d y\,\d \theta \right]^{\frac65} \d x \\
&\lec \int_\frac12^1 \int_{B_\frac54} \left[ \int_{B_1} |G(\theta x - (1-\theta )y)| \,\d y \right]^{\frac65} \d x \,\d \theta \\
&= \int_\frac12^1 \int_{B_\frac54} \left[ \int  1_{B_{1-\th}}(y')\, 1_{B_3}(x'-y')|G(x' - y')| \frac{\d y'}{(1-\theta)^3} \right]^{\frac65} \d x' \frac{\d \theta }{\theta^3} \\
&\leq \int_\frac12^1 \left\| |1_{B_3}\,G|\ast \frac{1_{B_{1-\th}}}{(1-\theta)^3}\right\|_{6/5}^{6/5}\frac{\d \theta }{\theta^3}\\
&\lec \| G \|_{L^1(B_3)}^{\frac65}  \int_\frac12^1 
\left\| \frac{1_{B_{1-\th}}}{(1-\theta)^3}\right\|_{L^\frac65_{x'}}^\frac65 \d \theta \\
&\sim  \| G \|_{L^1(B_3)}^{\frac65}  \int_\frac12^1   (1-\theta )^{-\frac35}    \d \theta  \\
&\lec  \| G \|_{L^1(B_3)}^{\frac65}.
\end{align*}
where we applied the changes of variables $x\mapsto\theta x=: x'$, $y\mapsto(1-\theta )y=: y'$ and the fact that $|x'-y'|\leq |x-y| \leq |x|+|y|<3 $ in the third line, and we used Young's inequality for convolutions in the fourth line. Combining with \eqref{I211.est}, we obtain
\[
\| I_{21} \|_{L^{\frac65}(B_1)} \lec \| G \|_{L^1(B_3)} .
\]
Finally letting $\eta(z) \coloneqq \overline{\eta}(z)z/|z|^{6-2s}$ we see that $\eta_j(z) = 2^{-3j} \eta (2^{-j}z)$, so that
\[
G=\sup_{j\geq 3} |\eta_j\ast \La^{2s-1}g| = \sup_{j\geq 3} |2^{-3j}\eta(2^{-j}\,\cdot ) \ast \La^{2s-1}g| \leq \mathcal{M} (\La^{2s-1} g; \eta ),
\]
which (together with the above estimates on $I_1$ and $I_{22}$) completes the proof. Note that \eqref{poin.p} easily follows from the pointwise estimate $ \mathcal{M} (\La^{2s-1} g; \eta )(x) \lec \mathscr{M}_4 (\La^{2s-1} g)(x)$ at any point $x\in \R^3$.
\end{proof}

\subsection{$L^\infty$-boundness}\label{sec_boundedness_of_u}

In this subsection, we obtain the local boundedness of $u$ under the assumptions of Theorem~\ref{thm_local_reg}. We first recall an $\ep$-regularity result of \cite[Proposition 2.9]{TY15}.

\begin{prop}\label{prop_tang_yu} There exists $\ep_0=\ep_0(s)>0$ such that if a suitable weak solution $(u,p)$ to \eqref{fNS} satisfies
\eqnb\label{small}
\begin{split}
\sup_{t\in \left(-\left(\frac{10}9\right)^{2s},0\right)} \int_{B_{\frac{10}9}}& |u(x,t)|^2 \d x
+\int_{Q_{\frac{10}9}^*} y^a |\bna u^*|^2 \d X \d t\\ 
&+ \left(\int_{Q_{\frac{10}9}} |u|^3\right)^\frac23 
+ \left(\int_{-\left(\frac{10}9\right)^{2s}}^0\left(\int_{B_{\frac{10}9}} |p(x,t)| \d x\right)^{q} \d t\right)^\frac2q \leq \ep_0,
\end{split}\eqne
for some $q\in [1,2)$, then the local boundedness of $u$ follows,
\[
 \| u \|_{L^\infty (Q_1)} \leq 1.
\]
\end{prop}

\begin{rem}
In \cite{TY15}, the authors restrict $q$ to the range of $q\in (\frac{4s}{(6s-3)},2)$ for the purpose of their main result, but
the proof of Proposition~\ref{prop_tang_yu} remains valid for any $q\in[1,2)$.
\end{rem} 
We now find $\ep>0$ in Theorem~\ref{thm_local_reg} such that the assumptions of the theorem imply \eqref{small}. This shows that \eqref{small} can be guaranteed (and so local boundedness follows) using only scale-optimal quantities given $u$ has the $\psi$-mean zero, $\int u \psi \d x=0$. (Recall that $\psi\in C_c^\infty (B_1)$ has been fixed in Theorem~\ref{thm_local_reg}.) 
\begin{prop}\label{prop:red} There exists $\ep>0$ such that if a suitable weak solution $(u,p)$ satisfies 
the assumptions in Theorem~\ref{thm_local_reg}, then
\[
 \| u \|_{L^\infty (Q_1)} \leq1. 
\]
\end{prop}
This proposition, together with our first result (Theorem~\ref{thm_bootstrapping_intro}), which guarantees boundedness of derivatives, concludes the proof of Theorem~\ref{thm_local_reg}.
\begin{proof} By Proposition~\ref{prop_tang_yu}, it suffices to prove that
\begin{align}\label{smallness.pre}
\sup_{t\in (-\left(\frac{10}9\right)^{2s},0)} \int_{B_\frac{10}9} |u(x,t)|^2 \d x
&+\int_{Q_\frac{10}9^*} y^a |\bna u^*|^2 \d X\d t
+\norm{u}_{L^3(Q_\frac{10}9)}^2
+\norm{p}_{L^1(Q_\frac{10}9)}^2
\leq \ep_0,
\end{align}
where $\ep_0$ is given by Proposition~\ref{prop_tang_yu}. We will show that the left hand side of \eqref{smallness.pre} is bounded by $c_*(\ep+ \ep^2)$ (and we will refer to such bound as ``smallness'') for some constant $c_*=c_*(s,\psi)$. Then, choosing $\ep$ sufficiently small such that $c_*(\ep+ \ep^2)\leq \ep_0$, we obtain \eqref{smallness.pre}. Without loss of generality we assume that $\int p \psi \d x=0$, since the pressure enters \eqref{fNS} only via $\nabla p$. \\

\noindent\texttt{Step 1.} We reduce the claim \eqref{smallness.pre} to showing only the smallness of $\| u \|_{L^\infty_t L^2_x (Q_{10/9})}$.\\

We note that $\int_{Q_1^*} y^a |\bna u^*|^2 \d X\d t\le \ep$ holds by assumption.
Since $\frac 6{3-2s}\geq 3$ for $s\geq \frac 12$, we use the interpolation inequality $\| f \|_{L^3} \leq \| f \|_{L^{\frac{6}{3-2s}}}^{\frac1{2s}} \|f \|_{L^2}^{\frac{2s-1}{2s}} $ and Lemma~\ref{lem_Poin} to get
\begin{align}\label{est.u}
\norm{u}_{L^3(Q_\frac{10}9)} 
\lec \norm{u}_{L^{4s}_tL^3_x(Q_\frac{10}9)}
\lec \norm{u}_{L^\infty_tL^2_x(Q_\frac{10}9)}
+ \norm{y^\frac{a}2  \bna u^*}_{L^2(Q^*_2)}
\lec \norm{u}_{L^\infty_tL^2_x(Q_\frac{10}9)}
+\sqrt{\ep}.
\end{align}

To  estimate the pressure, multiplying \eqref{fNS} by $\psi$ and integrating in space, we obtain
 \[
(\na  p )_\psi \coloneqq 
\int \psi \na p \,\d x  = \int \left( u (u\cdot\na)\psi -\psi \La^{2s}u  \right) \d x .
\]
By Lemma~\ref{lem_Poin}, the first term on the right hand side can be bounded by 
\[
\int |u|^2|\na \psi| 
\lec \| u \|_{L^2(B_1)}^2 
\lec \norm{y^\frac{a}2 \bna u^*}_{L^2(B_2^*)}^2.
\]
As for the second term, for each $i$ we have
\[
\int \psi \La^{2s} u_i  \d x
= -\overline{C}_s \lim_{y\to 0} \int y^a \p_y u^*_i \psi^* \d X
= \overline{C}_s \int_{\RR^4_+} \overline{\mathrm{div}} (y^a \bna u^*_i \psi^* ) \d X 
= \overline{C}_s \int_{\RR^4_+ } y^a \bna u^*_i \cdot \bna \psi^* \d X ,
\]
where $\psi^* \in C_c^\infty ( \RR^3 \times [0,1] )$ is any extension of $\psi$ such that $\psi^* (x,0)=\psi (x)$ (not the Caffarelli-Silvestre extension). This implies that $| \int  \psi \La^{2s}u  | \lec \|y^{\frac a2} \bna u^*\|_{L^2(B_2^*)} $ and hence at almost every time $t$,
\begin{align}
\| p \|_{L^2(B_\frac54)} &\leq \| \nabla p \|_{L^{\frac65} (B_\frac54)}
\lec \| \nabla p -(  \nabla p)_\psi  \|_{L^{\frac65} (B_\frac54)} +\norm{y^{\frac a2} \bna u^*}_{L^2(B_2^*)}^2+ \norm{y^{\frac a2} \bna u^*}_{L^2(B_2^*)} \nonumber \\
&\lec \norm{\La^{2s-1}\na p}_{L^1(B_5)} 
+ \norm{\mathscr{M}_4 ( \La^{2s-1} \na p)}_{L^1(B_3)}
+\norm{y^{\frac a2} \bna u^*}_{L^2(B_2^*)}^2
+ \norm{y^{\frac a2} \bna u^*}_{L^2(B_2^*)} \label{p_in_L2}
\end{align}
where we used the Poincar\'{e} inequality (recall $\int p\psi =0$) in the first inequality and Lemma~\ref{lem_poincare_for_p} in the last line. Integration in time over the interval $(-(10/9)^{2s},0)$ gives the estimate for the pressure, 
\begin{equation}\label{p_bound}\begin{split}
\| p \|_{L^1_tL^2_x (Q_\frac{10}9)} &\lec   
 \norm{\La^{2s-1}\na p}_{L^1(Q_5)} 
+ \norm{\mathscr{M}_4(\La^{2s-1}\na p )}_{L^1(Q_3)} +\norm{y^{\frac a2} \bna u^*}_{L^2(Q_2^*)}^2+\norm{y^{\frac a2} \bna u^*}_{L^2(Q_2^*)} \\
&\lec \ep + \sqrt{\ep}.
\end{split}\end{equation}
Thus, having shown smallness of $\norm{u}_{L^3(Q_\frac{10}9)} $ and $\| p \|_{L^1_tL^2_x (Q_\frac{10}9)}$, the claim \eqref{smallness.pre} follows if we show smallness of $\norm{u}_{L^{\infty}_t L^2_x (Q_\frac{10}9)}$.\\

\noindent\texttt{Step 2.} We show smallness of $\norm{u}_{L^\infty_tL^{2}_x(Q_\frac{10}9)}$.\\

For the convenience, we set
\begin{align*}
F(t)\coloneqq \norm{y^{\frac a2} \bna u^*(t)}_{L^2(B_2^*)}  +\norm{y^{\frac a2} \bna u^*(t)}_{L^2(B_2^*)}^2 + \norm{\La^{2s-1}\na p(t)}_{L^1(B_5)} 
+ \norm{\mathscr{M}_4(\La^{2s-1}\na p) (t)}_{L^1(B_3)}.
\end{align*}
Let $\xi=\xi(x,t)$ be a smooth cut-off in space and time satisfying $\xi(x,t) = 1$ on $Q_{\frac{10}9}$, $\supp(\xi) \subset Q_{\frac54}$, and $\xi^*$ be an extension of $\xi$ satisfying $\xi^* = 1$ on $Q_{\frac{10}9}^*$, $\supp(\xi^*) \subset Q_\frac54^*$, and $\xi^*(x,0,t) = \xi(x,t)$. The local energy inequality \eqref{LEI} applied with test function $(\xi^*)^2$ gives
\begin{equation}
\label{est.lei_xi}
\begin{split}
\frac12 \int |u(\tau )|^2 \xi(\tau )^2\d x
&+  \overline{C}_s \int_{-(5/4)^{2s}}^\tau \int_{\RR^4_+} y^a |\overline{\na } u^* |^2 (\xi^*)^2
\leq  \frac{\overline{C}_s}2  \int_{-(5/4)^{2s}}^\tau \int_{\RR^4_+} |u^*|^2 \overline{\mathrm{div}} (y^a \overline{\na } (\xi^*)^2) 
\\
&+ \int_{-(5/4)^{2s}}^\tau \int \left( \frac12|u|^2+p\right)u\cdot \na \xi^2  + \frac12 \int_{-(5/4)^{2s}}^\tau \int |u|^2 \left( \pa_t \xi^2 +\overline{C}_s \lim_{y\to 0^+}y^a\pa_y (\xi^*)^2 \right) 
\end{split}
\end{equation}
for almost every $\tau \in (-(5/4)^{2s},0)$. (Here, we used the fact that $\xi(\cdot, -(5/4)^{2s})=0$.) By Lemma~\ref{lem_Poin}
\begin{align*}
\int |u|^2 \pa_t \xi^2 \d x 
\lec \| u\|^2_{L^2 (B_\frac54)}
\lec \norm{y^\frac{a}2 \bna u^*}_{L^2(B_2^*)}^2\leq F
\end{align*}
for every $t \in (-(5/4)^{2s},0)$. Moreover, using H\"older's inequality, Lemma~\ref{lem_Poin} (note that $4<\frac6{(3-2s)}$ for $s>3/4$) and \eqref{p_in_L2}, we get 
\begin{equation}\label{computation_local_reg}\begin{split}
\int \left( \frac12|u|^2+p\right)u\cdot \na \xi^2 
&\lec 
\| u\xi \|_2  \left( \| u \|_{L^4 (B_\frac54)}^2  + \| p \|_{L^2 (B_{\frac54})} \right) \\
&\hspace{-2cm}\lec 
\| u\xi \|_2 \left( \norm{y^{\frac a2} \bna u^*}_{L^2(B_2^*)}  +\norm{y^{\frac a2} \bna u^*}_{L^2(B_2^*)}^2 +  \norm{\La^{2s-1}\na p}_{L^1(B_5)} 
+ \norm{\mathscr{M}_4(\La^{2s-1}\na p )}_{L^1(B_3)}  \right)\\
&\hspace{-2cm}\lec \left(1+ \int |u|^2 \xi^2  \right)  F
\end{split}\end{equation}
for every $t \in (-(5/4)^{2s},0)$, and so
\eqnb\label{rhs_1stpart}\begin{split}
\int_{-(5/4)^{2s}}^\tau \int \left( |u|^2 \pa_t \xi^2+ \left( \frac12|u|^2+p\right)u\cdot \na \xi^2 \right)
&\lec \int_{-(5/4)^{2s}}^\tau \left(1+ \int |u|^2 \xi^2  \right) F\\
&\lec  \norm{F}_{L^1(-(5/4)^{2s},0)}+\int_{-(5/4)^{2s}}^\tau \int |u|^2 \xi^2 F 
\end{split}\eqne
for every $\tau \in (-(5/4)^{2s},0)$. As for the remaining terms in \eqref{est.lei_xi}, we integrate by parts to get
\begin{align*}
 \int_{\RR^4_+} |u^*|^2 \overline{\mathrm{div}} (y^a \overline{\na } (\xi^*)^2) \d X
 + \int |u|^2 \lim_{y\to0^+} y^a\pa_y (\xi^*)^2 \d x
= -4 \int y^a (\xi^*\bna_j u_i^* )(u_i^* \bna_j \xi^*) \d X
\end{align*}
(at each $t\in (-(5/4)^{2s},\tau)$). Applying Young's inequality we obtain 
\begin{align*}
\abs{\int_{\R^4_+} y^a (\xi^*\bna_j u^*_i)(u_i^*  \bna_j \xi^*)\d X}
\leq \norm{y^\frac{a}2 \bna u^*}_{L^2(B_2^*)}^2
+\norm{ y^\frac{a}2 u^*}_{L^2(B_\frac54^*)}^2
\end{align*}
We now estimate the last term on the right-hand side. We first write the extension $u^*$ as
\[u^*(x,y,t) = u(x,t) + \int_0^y \pa_z u^*(x,z,t) \d z,\]
which gives
\begin{equation}\label{rep}\begin{split}
w(B_\frac74^*)|(u^*)_{B_\frac74^*}| 
&\leq w(B_\frac74^*)|(u)_{B_\frac74}| 
+ \int_{0}^2\int_{B_\frac74} \int_0^y y^az^{-\frac a2}\cdot z^{\frac a2} |\pa_z u^*(x,z,t)| \,\d z \,\d x\, \d y\\
&\lec \norm{y^\frac{a}2 \bna u^*}_{L^2(B_2^*)},
\end{split}\end{equation}
where $w(B_\frac74^*)\coloneqq \int_{B_\frac74^*} y^a \d X$ and $(f)_{B_\frac74^*} \coloneqq  \frac1{w(B_\frac74^*)}  \int_{B_\frac74^*} y^af(x,y) \d X$.  Here the last inequality follows from a modification of \eqref{est.ave} and the Cauchy-Scharz inequality. This together with the weighted Poincare inequality \eqref{weighted_poin} give
\begin{align*}
\int_{B_\frac74^*} y^a |u^*|^2   \d X
&\lec \int_{B_\frac74^*} y^a |u^*-(u^*)_{B_\frac74^*}|^2   \d X 
+ w(B_\frac74^*)|(u^*)_{B_\frac74^*}|^2\lec \int_{B_2^*} y^a |\bna u^*|^2 \d X
\end{align*}
for almost every $t\in (-(5/4)^{2s},0)$. Integrating in time we obtain $\norm{ y^\frac{a}2 u^*}_{L^2(Q_\frac54^*)} \lec \norm{y^\frac{a}2\bna u^*}_{L^2(Q_2^*)}\leq \| F \|_{L^1(-(5/4)^{2s},0)}$  
and therefore
\begin{align}\label{est.dissi}
\left|\int_{-(5/4)^{2s}}^\tau \left( \int_{\RR^4_+} |u^*|^2 \overline{\mathrm{div}} (y^a \overline{\na } (\xi^*)^2) \d X
+  \int |u|^2 \lim_{y\to0^+} y^a\pa_y (\xi^*)^2 \d x \right)  \right|
\lec \norm{F}_{L^1(-(5/4)^{2s},0)}
\end{align}
for every $\tau \in (-(5/4)^{2s},0)$. Applying the estimates \eqref{rhs_1stpart} and \eqref{est.dissi} in \eqref{est.lei_xi} gives
\begin{equation*}
\begin{split}
\int |u(\tau )|^2 \xi(\tau )^2  \d x
\lec \int_{-(5/4)^{2s}}^\tau \left(\int |u|^2 \xi^2 \d x\right)  F
+ \norm{F}_{L^1(-(5/4)^{2s},0)}.
\end{split}
\end{equation*}
for almost every $\tau \in (-(5/4)^{2s},0)$. Using the integral Gr\"onwall's inequality (see, for example, Theorem~7.3 in \cite{ODE_book}), it follows that
\begin{align*}
\int |u(\tau )|^2 \xi(\tau )^2 \d x
\lec  \norm{F}_{L^1(-(5/4)^{2s},0)}(1+   \norm{F}_{L^1(-(5/4)^{2s},0)}\exp( \norm{F}_{L^1(-(5/4)^{2s},0)}))
\end{align*}
for almost every $\tau \in (-(5/4)^{2s},0)$. 
Since $\norm{F}_{L^1(-(5/4)^{2s},0)} \lec \sqrt{\ep} + \ep$ by assumption we obtain (assuming $\ep\in (0,1)$)
\begin{align*}
\| u \|_{L^{\infty}_tL^2_x  (Q_\frac{10}9)}
\lec \sqrt{\ep},
\end{align*}
as required.
\end{proof}
\subsection{Consequences of the local regularity Theorem~\ref{thm_local_reg}}\label{sec_consequences}
We now show that the claim of Theorem~\ref{thm_local_reg} remains valid if one replaces the zero mean property by a smallness assumption on $|u|^3+|p|^{3/2}$. In other words we obtain Corollary~\ref{cor_single_scale_local_reg_into}, which we now restate for the reader's convenience.
\begin{cor}\label{cor_single_scale_local_reg}
There exists $\ep>0$ such that if a suitable weak solution $(u,p )$ of \eqref{fNS} satisfies
\begin{equation}\begin{split}
\label{smallness_cor}
&\int_{Q_5^*} y^a |\bna u^*|^2 \d X \,\d t + \int_{-5^{2s}}^0 \int_{B_5}\int_{B_5} \frac{|u(x,t)-u(y,t)|^2}{|x-y|^{3+2s}}\d x \,\d y\, \d t\\
&\hspace{1cm}+ \int_{Q_5}  \left( 
(\mathcal{M} |\La^s u |^{\frac 2{1+\de}})^{1+\de} +  |\La^{2s-1}\na p|  +|\mathscr{M}_4 (\La^{2s-1}\na p)|+|u|^3 + |p|^{\frac32} \right) \leq \ep,
\end{split}\end{equation}
where $\delta \coloneqq \frac{2s}{(6-s)}$. Then
\[
\sup_{Q_\frac12}\left(|u|+
|\na u|+|\nabla^2  u| \right) \leq C_1
\] 
for some constant $C_1>0$. 
\end{cor}
\begin{proof}
We deduce \eqref{small} by the choice of sufficiently small $\ep$ (this choice is independent of the one in Proposition~\ref{prop:red}) and the same argument as in Proposition~\ref{prop:red} except for using Lemma~\ref{lem_Poin} only to $u-(u)_\psi$, where $(u)_\psi \coloneqq (\int \psi\d y)^{-1} \int u (t) \psi \d y$. We outline the main updates to the proof of Proposition~\ref{prop:red} below.

First, we replace \eqref{est.u} by the assumption, and \eqref{p_in_L2} by writing
\eqnb\label{tempxxx}
\begin{split}
\| p \|_{L^2(B_\frac54)} &\leq \| \nabla p \|_{L^{\frac65} (B_\frac54)} 
\lec \left\| \nabla p - \int \psi \nabla p \right\|_{L^{\frac65} (B_\frac54)} +\left| \int \psi \na p  \d x\right|  \\
&\lec \norm{\La^{2s-1}\na p}_{L^1(B_5)} 
+ \norm{\mathscr{M}_4 ( \La^{2s-1} \na p )}_{L^1(B_3)}+\| p \|_{L^{\frac32}(B_1)},
\end{split}
\eqne
where we applied Lemma~\ref{lem_poincare_for_p} and integrated the last term by parts in the second inequality.
This implies that
\eqnb\label{p_bound_for_db}
\| p \|_{L^1_tL^2_x (Q_\frac54)} \lec   
 \norm{\La^{2s-1}\na p}_{L^1(Q_5)} 
+ \norm{\mathscr{M}_4 ( \La^{2s-1} \na p )}_{L^1(Q_3)} + \| p \|_{L^{\frac32} (Q_1)} ,
\eqne
which is our substitute for \eqref{p_in_L2}. We are left to estimate $\norm{u}_{L^\infty_tL^2_x(Q_\frac{10}9)}$, for which we again use the local energy inequality \eqref{est.lei_xi} with the same test function $(\xi^*)^2$, but with some estimates for the terms on the right-hand side replaced as follows:
\begin{align*}
\int_{-(5/4)^{2s}}^\tau \int |u|^2 \pa_t \xi^2
&\lec \| u(t)\|^2_{L^2 (Q_\frac54)}
\lec \| u(t) - (u)_{B_\frac54} \|^2_{L^2 (B_\frac54)} + |(u)_{B_\frac54}|^2\\
&\lec \norm{y^\frac{a}2 \bna u^*(t)}_{L^2(Q_2^*)}^2
+ \norm{u}_{L^3(Q_2)}^2,
\end{align*}
and
\begin{align*}
\int_{-(5/4)^{2s}}^\tau \int \left( \frac12|u|^2+p\right)u\cdot \na \xi^2 
\lec \norm{u}_{L^3(Q_2)}^3 + \norm{p}_{L^\frac32(Q_2)}^\frac32.
\end{align*}
The estimate for the remaining term is similar to \eqref{est.dissi} except for \eqref{rep}, where we use the estimate $|(u)_{B_\frac74}| 
\lec \norm{u}_{L^3(B_2)} $ (instead of $|(u)_{B_{\frac74}}| \lec \| y^{\frac{a}2} \bna u^*\|_{L^2(B_2^*)}$). 
\end{proof}
We now show that Corollary~\ref{cor_single_scale_local_reg} gives an estimate on the box-counting dimension of the singular set, that is we prove Corollary~\ref{cor_box_counting}.

We first note that $u\in L^r(0,{T};L^q(\R^3))$ for any $r,q\in [2,\infty ]\times [2,6/(3-2s)] $ such that
\begin{align*}
\frac {2s}{r} + \frac 3q = \frac 32,
\end{align*}
recall Section~\ref{sec_suitable_ws}. This and the fact that all the other quantites appearing in \eqref{smallness_cor} are globally integrable (see \eqref{ext.norm}, \eqref{Sob.Slo.L2norm}, \eqref{hardy_characterisation}) allow us to use a standard covering argument. Indeed, recall the definition of the box-counting dimension,
\[
d_B (K) \coloneqq \limsup_{r\to 0} \frac{\log M (K,r)}{-\log r},
\]
for a set $K\subset \RR^3 \times (0,T)$, where $M(K,r)$ denotes the maximal number of pairwise disjoint $r$-balls {(in $\RR^4$)} with centers in $K$. One can see that the box-counting dimension $d_B (K)$ can be bounded by
\eqnb\label{bound_for_db}
\limsup_{r\to 0} \frac{\log M' (K,r)}{-\log r},
\eqne
where $M'(K,r)$ denotes the maximal number of pairwise disjoint cylinders $Q_r (x,t)$ with $(x,t)\in K$ (as $r^{2s}<r$ for $r<1$).
 
 We set
\[
W(x,y,t)\coloneqq \frac{|u(x,t)-u(y,t)|^2}{|x-y|^{3+2s}}
\]
and (as in \eqref{def_of_F_G})
\[
\begin{split}
F(x,t) &\coloneqq ( \mathcal{M} |\La^s u(x,t) |^{\frac 2{1+\de}})^{1+\de}
+|\La^{2s-1}\na p(x,t)|  +|\mathscr{M}_4 (\La^{2s-1}\na p )(x,t)|\\
G(X,t) &\coloneqq y^a |\bna u^* (X,t)|^2 
\end{split}
\]
As in \eqref{global} we see that all above quantities are globally integrable,
\eqnb\label{global_repeat}
\int_0^T \iint W(x,y,t) \d x\, \d y\, \d t + \int_0^T \int F(x,t) \d x  \,\d t + \int_0^T \int_{\RR^4_+} G(X,t) \d X \, \d t \lec \int_0^T \int |\La^s u |^2  \lec \| u_0 \|_2
\eqne
Corollary~\ref{cor_single_scale_local_reg} gives that if 
\[\begin{split}
&\frac{1}r\left( \int_{-r^{2s}}^0\int_{B_r} \int_{B_r} W(x,y,t)\d x\, \d y\, \d t + \int_{Q_r} F(x,t)\d x \, \d t +\int_{Q_r^* } G (X,t) \d X \,\d t \right) \\
&\hspace{7cm}+ {r^{4s-6}} \int_{Q_r}\left( |u(x,t)|^3 + |p(x,t)|^{3/2}\d x\, \d t \right) \leq \epsilon
\end{split}
\]
then $u$ and its spatial derivatives are bounded on $Q_{r/10}$. Given $t_0>0$ we consider $r<\min ( t_0^{1/2s},1)$ and let $\{ Q_r (x_k ,t_k )\}_{k=1}^{M'(S\cap \{ t\geq t_0 \} ,r)}$ be any collection of pairwise disjoint $r$-cylinders with $(x_k,t_k)\in S \cap \{ t\geq t_0 \}$ for every $k=1,\ldots ,M'(S\cap \{ t\geq t_0 \} ,r)$. Then
\[\begin{split}
\| u_0 \|_2 &\gtrsim \int_0^T \iint W + \int_0^T \int {F} + \int_0^T \int_{\RR^4_+} G + \int_0^T \int \left( |u|^\frac{2(2s+3)}{3} + |p|^\frac{2s+3}{3}\right) \\
&\gtrsim \sum_{k=1}^{M'(S\cap \{ t\geq t_0 \} ,r)} \left( \int_{t_k-r^{2s}}^{t_k}\int_{B_r(x_k)^2} W + \int_{Q_r(x_k,t_k)}F +\int_{Q_r^*(x_k,t_k) } G  + \int_{Q_r (x_k,t_k)}     \left( |u|^\frac{2(2s+3)}{3} + |p|^\frac{2s+3}{3}\right) \right) \\
&\gtrsim \sum_{k=1}^{M'(S\cap \{ t\geq t_0 \} ,r)} \left( \epsilon r +        \left( r^{\frac{3-4s}{2}} \int_{Q_r(x_k,t_k)} \left( |u|^3+ |p|^{3/2}\right)  \right)^{\frac{2(2s+3)}{9}} \right) \\
&\geq  \sum_{k=1}^{M'(S\cap \{ t\geq t_0 \} ,r)} \left( \epsilon r +         r^{\left( \frac{3-4s}{2}+6-4s \right)\frac{2(2s+3)}{9}} \epsilon^{\frac{2(2s+3)}{9}} \right) \\
&\gtrsim_\epsilon  M'(S\cap \{ t\geq t_0 \} ,r) r^{(-8s^2-2s+15)/3}
\end{split}\]
where we used \eqref{global_repeat} and the global integrability of $|u|^\frac{2(2s+3)}{3} + |p|^\frac{2s+3}{3}$ (recall Section~\ref{sec_suitable_ws}) in the first line, and H\"older's inequality in the fourth inequality. Applying this estimate in the bound \eqref{bound_for_db} gives $d_B(S\cap \{ t\geq t_0 \} ) \leq (-8s^2-2s+15)/3$, as required.

 \section{Higher derivatives of weak solutions}\label{sec_bddness_of_higher_der}
 In this section we prove Theorem~\ref{thm_bootstrapping_intro}, which we restate for reader's convenience. 

\begin{thm}\label{thm_bootstrapping}
Suppose that a Leray-Hopf weak solution $(u,p)$ to \eqref{fNS} for $\frac  34<s<1$ satisfies  
\[\begin{split}
\| u \|_{L^\infty_{t,x}(Q_1)}&+ \| u\|_{ L^2_tW^{s,2}_x(Q_2)}+\|p\|_{ L^1_{t,x}(Q_1)}+ \| \na p \|_{L^1_{t,x}(Q_1)} \\
&+ \| \mathcal{M} (\La^s u )\|_{L^2 (Q_2)}+\|\mathcal{M} |\La^s u |^{\frac 2{1+\de}} \|_{ L^{1+\de} (Q_2)}+ \| \mathscr{M_4} |\La^{2s-1}\na p | \|_{ L^1 (Q_2)}\leq c<\infty
\end{split}
\]
for $\de= \frac {2s}{6-s}$. Then the velocity $u$ satisfies
\[
\sup_{Q_1} |u(x,t)| + |\na u(x,t) | + |\nabla^2 u(x,t) | \leq C_0
\]
for some constant $C_0=C_0(c,s)$. 
\end{thm}
\

We first introduce a lemma for pressure decomposition. 
\begin{lemma}[Pressure decomposition]\label{lem_pressure_est}
Let $p$ be a solution to $-\Delta p = \p_i\p_j(u_iu_j)$ and let $\phi$, $\bar \phi $ be smooth cut-offs in space satisfying $\supp(\phi) \Subset \{ \bar\phi =1 \}\subset \supp(\bar\phi) \subset B_1$. Then we have the decomposition
\[
\phi \nabla p = \phi \nabla  \mathcal{R}_{ij} (u_iu_j \bar\phi)  + \Ga
\]
for some $\Ga$ satisfying
\begin{align}\label{est.remainder}
\supp (\Ga)\subset \supp\, \phi, \quad
\norm{\Ga}_{W^{k, \infty }_x} \lec_{k,\phi,\bar\phi } \norm{ u}_{L^{\infty}(B_1)}^2+ \norm{ p}_{L^{1}(B_1)}, \quad \forall k\geq 0.
\end{align}
\end{lemma}

\begin{proof}
Let $E\coloneqq \supp\, \phi$. Manipulating the equation $-\De p = \pa_{ij}(u_iu_j)$, we can write $p\bar\phi$ as
\begin{equation}\begin{split}
\label{eq.loc.pre}
 p\bar\phi  &
 =\mathcal{R}_{ij} (u_iu_j \bar\phi) 
 - \left[ 2(-\De)^{-1}\pa_i(u_i u_j \pa_j\bar\phi) +(-\De)^{-1}(u_i u_j \pa_{ij}\bar\phi) \right]\\
 & \quad- \left[2 (-\De)^{-1}\na \cdot (p \na \bar\phi ) -(-\De)^{-1}( p\De\bar\phi)\right]\\
&=: \mathcal{R}_{ij} (u_iu_j \bar\phi)  +p_1+p_2.
\end{split}\end{equation}
(Note that here we have used uniqueness of solutions to the Poisson equation: if $\Delta f =0$ in $\RR^3$ and $f\in L^p$ for any $p\in [1,\infty)$ then $f=0$, which can be proved using mollification and Liouville's theorem.) Thus 
\[\bs
\phi \nabla p &=  \phi \nabla (p\bar\phi )
= \phi \nabla \mathcal{R}_{ij} (u_iu_j \bar\phi) + \phi \nabla (p_1+p_2).
\end{split}
\]
We now show that
\begin{align}
\label{pressure1}
\| \na^k p_1 \|_{L^{\infty } (E)} 
&\lec_k \| u\|^2_{L^{\infty } (B_1)},\\
\label{pressure2}
\| \na^k p_2 \|_{L^1_tL^{\infty }_x (E)} 
&\lec_k \| p \|_{L^{1}(B_1)} 
\end{align}
for every integer $k\geq 0$. Then, setting $\Ga\coloneqq \phi \nabla (p_1+p_2)$, the claim follows from the product rule.

Since every term included in $p_1$ or $p_2$ involves at least one derivative falling on $\bar\phi$, and we have $\text{dist}(E, \supp(\na \bar\phi)) \geq \text{dist}(\supp(\phi), \{\bar\phi=1\}^c)>c_{\phi,\bar\phi}>0$ by the assumption, $p_1$  and $p_2$ satisfy
 \[
|\na^k p_1 (x,t) | \lec\int \left( \frac{|u(y)|^2 |\nabla \bar\phi (y) |}{|x-y|^{2+k}} +\frac{|u(y)|^2 |\na^2 \bar\phi (y)|}{|x-y|^{1+k}}    \right) \d y \lec_{k,\phi,\bar\phi} \| u \|_{L^{\infty}(B_1)}^2
\]
and
\eqnb\label{est_on_nak_p2}
|\na^k p_2 (x,t) | \lec\int \left( \frac{|p|\, |\nabla \bar\phi |}{|x-y|^{2+k}} +\frac{|p|\, |\na^2 \phi |}{|x-y|^{1+k}}    \right) \d y \lec_{k,\phi,\bar\phi} \| p \|_{L^{1 }(B_1)},
\eqne
for any $x\in E$, from which \eqref{pressure1} and \eqref{pressure2} follow, respectively.
\end{proof}

\subsection{Proof of Theorem~\ref{thm_bootstrapping}}\label{sec_bootstrapping}
We now discuss our new bootstrapping scheme, which proves Theorem~\ref{thm_bootstrapping}. We will use use a number of commutator estimates, which we discuss in Lemmas~\ref{lem_tricks}--\ref{lem_comm.p} in Section~\ref{sec_commutator_est} below.
\begin{proof}[Proof of Theorem~\ref{thm_bootstrapping}.]
In the proof we abuse the notation by letting $\{ Q_\ell \}_{\ell \in  \frac{\mathbb{N}}2} $ be a sequence of parabolic cylinders $Q_\ell  \coloneqq (-r(\ell )^2, 0)\times B_{r(\ell )}$, where $r(\ell )$ is a radius that is strictly decreasing in $\ell $, that satisfies $Q_1:= (-1,0)\times B_1$ and $\bigcap_{\ell \geq 1 } Q_\ell  = (-\left(\frac12\right)^{2s},0)\times B_\frac12$. In particular, $Q_{\ell +1}\Subset Q_{\ell +1/2} \Subset Q_\ell $. We set $\phi_\ell  \in C_c^\infty (Q_{\ell -1/2};[0,1])$ be such that $\phi_\ell =1 $ on $Q_\ell $.

In what follows, we consider several equations of the form $(\p_t+\La^{2s})v = \sum_i f_i$, where the right hand side is given as a finite summation of forcing terms $f_i$. To deduce a certain regularity of $v$, we decompose the equation into $(\p_t+\La^{2s})v_i =  f_i$ and apply the parabolic regularity estimates in Lemma~\ref{lemma:para.est} to each equation $(\p_t+\La^{2s})v_i= f_i$. For the simplicity, if $v_i$ satisfies the regularity required on $v$, we say ``some regularity of $f_i$ gives the required regularity of $v$" below.  \\

\noindent\texttt{Step 1.}  We show that $\La^{2s-1+\varepsilon } (u\phi_2)\in L^{2}_tL^{\frac6{1+4\varepsilon }}_x$ for $\varepsilon \in (0,1/2]$.\\

Multiplying \eqref{fNS} by $\phi_2$ and using Lemma~\ref{lem_pressure_est} to write the pressure term in terms of the expressions involving the Riesz transform and the remainder we obtain 
\begin{equation}\begin{split}
\label{eqn.uphi}
(\pa_t + \La^{2s}) (u\phi_{2}) 
&= \nabla\cdot F + [\La^{2s},\phi_{2}] u  
+\left( u\otimes u :\na \phi_{2} - \nabla \phi_{2} \mathcal{R}_{ij} (u_iu_j \phi_{\frac32})   + u\pa_t \phi_{2}\right)+ \Ga \\
&=: \na \cdot F + G_1 + G_2 + \Ga,
\end{split}\end{equation}
where
\[
F\coloneqq -(u\otimes u\phi_{2}) +  \phi_{2 } \mathcal{R}_{ij} (u_iu_j \phi_{\frac32})\I.
\]

First of all, $G_2\in L^{\infty}_t L^p_x$ for any $p\in(1,\infty)$ which implies $\Lambda^{\ga } (u\phi_{2} )\in L^{\infty}_{t}L^p_{x}$ for any $\ga \in [0 , 2s )$ and $p\in (1,\infty)$ by \eqref{para.est}. Therefore, $G_2$ gives the required regularity of $u\phi_2$.  
  
Next, $\Ga\in  {L^1_tW^{k,p}_x}$ for any $k\in \mathbb{N}\cup\{0\}$ and $p\in [1,\infty]$ due to \eqref{est.remainder}, so that we have $u\phi_{2}\in L^\infty_{t}W^{k,p}_{x} $ for any $k\in \mathbb{N}\cup \{0\}$ and $p\in [1, \infty]$ by \eqref{para.est.0}, and then $\La^\be(u\phi_{2})  \in L^\infty_tL^p_x$ for any real number $\be \geq 0$ and $p\in [1, \infty]$ by interpolation. In particular, it gives the required regularity of $u\phi_2$.

As for the commutator term $G_1$, we have $G_1 \in L^{2}_{t,x}$ by Lemma~\ref{lem_1.10}, which gives the required regularity of $u\phi_{2}$ by \eqref{para.est}. (The fact that $\varepsilon\in (0,\frac 12]$ is used here.)

The term $\na \cdot F$ is the most subtle to handle. Indeed, noting that the assumption gives $F\in L^\infty_tL^p_x$ for every $p\in (1,\infty )$ the parabolic estimate \eqref{para.est} (applied with $\alpha \coloneqq -1$) only allows to estimate $\La^{\overline{\alpha } } (u\phi_2)$ for $\overline{\alpha } < 2s-1$ (as $\sigma +(1/r+1/\overline{r})\geq 0$). In order to reach above the $2s-1$ threshold we show below that 
\eqnb\label{La_2eps_to_show}
\| \La^{\frac32\varepsilon } F  \|_{L^2_tL^{q}_{x}}\leq C (\| u\|_{L^{\infty}_{t,x} (Q_1)}, \| u \|_{L^2_t W^{s,2}_x(Q_1)} )
\eqne
for $q := \frac 6{1+4\varepsilon}$. This gives the required regularity of $u\phi_2$ by using $L^q$-boundedness of the Riesz transform and applying \eqref{para.est} with $\alpha \coloneqq -1+3\varepsilon /2$, $\overline{\alpha }\coloneqq 2s-1+\varepsilon$, $\sigma \coloneqq 0$ and $r,\overline{r}\coloneqq 2$.

Therefore, we are left to prove \eqref{La_2eps_to_show}. We first use the fractional Leibniz rule \eqref{fractional_leibniz}  to get
\begin{align*}
\| \La^{\frac32\varepsilon } (u_i u_j \phi_{\frac32} ) \|_{L^{q}_x} 
&=\| \La^{\frac32\varepsilon } (u_i u_j \phi_2\phi_{\frac32} ) \|_{L^{q}_{x}} \\
&\lec \norm{\La^{\frac32\varepsilon}(u\phi_{\frac32})}_{L^{q}_x}\norm{u\phi_2}_{L^\infty_x}
+ \norm{\La^{\frac32\varepsilon}(u\phi_2)}_{L^{q}_x}\norm{u\phi_{\frac32}}_{L^\infty_x}\\
&\lec \left( \|\La^{\frac32\varepsilon}(u\phi_{\frac32})\|_{L^{q}_x} +\|\La^{\frac32\varepsilon}(u\phi_2)\|_{L^{q}_x}\right) \norm{u}_{L^\infty(B_1)}.
\end{align*}
Since we have $\supp_x(\phi_{\frac32}) 
\coloneqq \{ x: (x,t)\in \supp(\phi_{\frac32}) \text{ for some } t\}
\Subset B_1$ from its construction, we can choose sufficiently small $R>0$ such that $B_{2R}(\supp_x(\phi_{\frac32}))   \coloneqq \supp_x(\phi_{\frac32}) +B_{2R}  \subset B_1 $. For such $R$ and fixed $t\in (-1,0)$, we obtain the following interpolation inequality for $\varepsilon \in (0,2s)$
\begin{align*}
\norm{\La^{\frac32\varepsilon}(u\phi_{\frac32})}_{q}
&\lec \left(\int_{|y|<R} + \int_{|y|\geq R}\right) \frac{\norm{u\phi_{\frac32}(\cdot +y ) - u\phi_{\frac32}}_{q}}{|y|^{3+\frac32\varepsilon}} \d y\\
&\lec \int_{|y|< R}\frac{\norm{u\phi_{\frac32}(\cdot +y ) - u\phi_{\frac32}}_{2}^\frac2q }{|y|^{(3+2s)\frac1q}}\cdot \frac {\norm{u\phi_{\frac32}}_{\infty}^{1-\frac2q}}{|y|^{\frac 3{q'} -\frac32\varepsilon- \frac{2s}q} } \d y
+\int_{|y|\geq R} \frac {\norm{u\phi_{\frac32}}_{q} }{|y|^{3+\frac32\varepsilon} }\d y \\
&\lec \left[\int_{B_R(\supp_x(\phi_{\frac32}))} \int_{|y|< R} \frac{|u\phi_{\frac32}(x+y)-u\phi_{\frac32}(x)|^2}{|y|^{3+2s} }\d y\,\d x \right]^{\frac1q}\norm{u\phi_{\frac32}}_{\infty}^{1-\frac2q} + \norm{u\phi_{\frac32}}_{q}\\
&\lec \left[\int_{B_1} \int_{B_1} \frac{|u\phi_{\frac32}(y)-u\phi_{\frac32}(x)|^2}{|y-x|^{3+2s} }\d y\,\d x \right]^{\frac12} + \norm{u}_{L^\infty(B_1)}\\
&\lec \norm{u}_{\dot{W}^{s,2}(B_1)} + \norm{u}_{L^\infty(B_1)},
\end{align*}
where we used the $L^p$-interpolation inequality in the second line, H\"older's inequality and the fact that $u\phi_{\frac32}(x+y) = u\phi_{\frac32}(x)=0$ for $|y|<R$ and $x\notin B_R(\supp_x(\phi_{\frac32}))$ in the third line, Young's inequality and change of variable $y\mapsto y-x$ in the fourth line, as well as the Leibniz rule 
\[
\norm{u\phi_{\frac32}}_{\dot{W}^{s,2}(B_1)} \lec \norm{u}_{\dot{W}^{s,2}(B_1)}\norm{\phi_{\frac32}}_{L^\infty(B_1)} + \norm{u}_{L^\infty(B_1)}\norm{\phi_{\frac32}}_{\dot{W}^{s,2}(B_1)}
\lec \norm{u}_{\dot{W}^{s,2}(B_1)} +\norm{u}_{L^\infty(B_1)}
\] 
in the last line. Similarly one can show the same upper bound for $\norm{\La^{\frac32 \varepsilon} (u\phi_2)}_{L^p_x}$, up to a constant multiple. Thus \eqref{La_2eps_to_show} follows by using the $L^q$-boundedness of the Riesz transforms. \\ 

\noindent\texttt{Step 2.} We show that $\La^{\be} (u\phi_7)\in L^{\infty}_tL^p_x$ for every $p\in [6,\infty)$, $\be \in [0,2s)$.\\

We will use \eqref{eqn.uphi} with $\phi_{2}$, $\phi_{\frac32}$ replaced by the cutoffs $\phi_i$, $\phi_{i-\frac12}$ for $i=3,\cdots,7$ in each of the finite number of the iterations below. We note that $G_2$ and $\Ga$ (after replacing the cutoffs) gives the required regularity by the same argument in \texttt{Step 1}. Therefore, we only need to focus on $\na \cdot F$ and $G_1$ (with the replacement of the cutoffs).\\

\noindent\texttt{Step 2a.}  $\La^{2s-1+\varepsilon } (u\phi_5)\in L^{\infty}_tL^p_x$ for every $p\in [6,\infty) $ and $\varepsilon \in (0,s-\frac34)$. \\

To deal with $\nabla \cdot F$, we can control $\La^{2s-1+\varepsilon }F$ using the improved regularity of $u$ in \texttt{Step 1};
\eqnb\label{first2terms}\begin{split}
\| \La^{2s-1+\varepsilon } F \|_{L^{2}_tL^{\frac6{1+4\varepsilon }}_x}&\leq 
\| \La^{2s-1+\varepsilon } (u\otimes u \phi_{3})\|_{L^{2}_tL^{\frac6{1+4\varepsilon }}_x}
+\| \La^{2s-1+\varepsilon } (\phi_{3} \mathcal{R}_{ij} (u_iu_j \phi_{5/2}) \|_{L^{2}_tL^{\frac6{1+4\varepsilon }}_x} \\
&\leq C\left( \| u \|_{L^{\infty}_{t,x} (Q_1)} , \| \La^{2s-1+\varepsilon }(u \phi_2) \|_{L^{2}_tL^{\frac6{1+4\varepsilon }}_x} \right)  
\end{split}
\eqne
The last two lines follow from the fractional Leibniz rule \eqref{fractional_leibniz} and the facts that $\phi_3= \phi_3\phi_2^2$ and $\phi_{\frac52}= \phi_{\frac52}\phi_2^2$. Then, using \eqref{para.est} and $L^p$-boundedness of Riesz transform, it gives $\La^{2s-1+\varepsilon } (u\phi_3) \in L^{2}_tL^p_x$ for any $p\in [6,\infty )$ and $\varepsilon  \in (0, \frac12)$.

As for $G_1$, since \texttt{Step 1} and \eqref{norm_comparisons} imply that $u\in L^2_tW^{2s-1+\varepsilon, \frac6{1+4\varepsilon} }_x(Q_2)$ (because of $\frac6{1+4\varepsilon}\geq 2$), we have $ [\La^{2s}, \phi_3 ] u \in L^{2}_tL^{\frac6{1+4\varepsilon}}_x$ by Lemma~\ref{lem_1.10}. Thus, it follows from \eqref{para.est} that $\La^{2s-1+\varepsilon } (u\phi_3) \in L^{2}_tL^{p}_x$ for any $p\in [6,\infty )$ and $\varepsilon  \in (0, s-\frac34)$.\\

Repeating the same argument twice, we can improve to the required regularity of $u\phi_5$;
\begin{align*}
\La^{2s-1+\varepsilon } (u\phi_3)\in L^{2}_tL^p_x, \quad \forall p\in [6,\infty)
&\implies \La^{2s-1+\varepsilon } (u\phi_4)\in L^{\frac s{1-s}}_tL^p_x, \ \qquad \forall p\in [6,\infty)\\
&\implies \La^{2s-1+\varepsilon } (u\phi_5)\in L^{\infty}_tL^p_x, \qquad\quad \forall p\in [6,\infty).
\end{align*}
Indeed, the first implication follow from $ [\La^{2s}, \phi_4 ] u, \La^{2s-1+\varepsilon}F \in L^{2}_tL^{p}_x$ for $p\in [6,\infty)$, (for $F$ with $\phi_4, \phi_{7/2}$). Then the second one follows from $ [\La^{2s}, \phi_5 ] u \in L^{2}_tL^{p}_x$ for $p\in [6,\infty)$ and $\La^{2s-1+\varepsilon}F\in L^{\frac s{1-s}}_tL^p_x$ for $p\in [6,\infty)$ (for $F$ with $\phi_5, \phi_{9/2}$).  \\

\noindent\texttt{Step 2b.} $\nabla (u\phi_6)\in L^{\infty}_tL^p_x$ for every $p\in [6, \infty]$. \\

As for $\nabla\cdot F$, since if $(\pa_t + \La^{2s}) w = \na \cdot F$ then $(\pa_t + \La^{2s})\na w = \na (\na \cdot F)$, using \eqref{para.est} we get
\begin{align}\label{na_w_calc}
\norm{\na w}_{L^\infty_{t,x}} 
\lec \norm{\La^{2s-3+\varepsilon} \na(\na \cdot F) }_{L^\infty_t L^p_x} 
\lec \norm{\La^{2s-1+\varepsilon}F}_{L^\infty_t L^p_x}.
\end{align}
Therefore, the inclusion $\La^{2s-1+\varepsilon } F\in L^{\infty}_tL^p_x$ for every $p\in [6,\infty)$ and $\varepsilon  \in (0, 1)$  (obtained as in \eqref{first2terms}) gives the required regularity of $u\phi_6$ (as $\na (u\phi_6)\in L^\infty_{t,x}$ implies $\na (u\phi_6)\in L^\infty_{t}L^p_x$ for every $p$).

As for the commutator term $G_1$, we use the decomposition $L^{\infty}_tL^{ p}_x+L^{2}_tW^{k,\infty }_x$ (for $p\in [6,\infty)$) suggested in Lemma~\ref{lem_1.10}. Then the latter part (in $L^{2}_tW^{k,\infty }_x$) gives the required regularity of $u\phi_6$ by \eqref{para.est.0}, while the former part (in $L^{\infty}_tL^{ p}_x$) also does by \eqref{para.est}. (In fact, it even gives the regularity $\La^{2s-\varepsilon_1} (u\phi_6)\in L^\infty_t L^p_x$ for any $p\in (6,\infty)$ and $\varepsilon_1 \in (0, 2s]$.\\

\noindent\texttt{Step 2c.} $\La^{2s-\varepsilon_1} (u\phi_7)\in L^{\infty}_tL^p_x$ for every $p\in [6,\infty)$, $\varepsilon_1  \in (0, 2s]$. \\

As in \texttt{Step 2b}, $G_1$ gives the required regularity on $u\phi_7$. As a consequence of \texttt{Step 2b}, we have $\nabla\cdot F \in L^{\infty}_tL^p_x$ for every $p\in [6,\infty)$. Therefore, it gives the regularity on $u\phi_7$ by \eqref{para.est}.\\ 

\noindent\texttt{Step 3.} We show that $\na (\phi_8 \La^{\ga} u)\in L^{\infty}_tL^p_x$  
for every $p\in [1,\infty]$, $\ga\in (s,1)$.\\

Since $\phi_8 \La^{\ga} u$ is compactly supported, it is enough to obtain the regularity for large $p$. As a byproduct of the proof, we also get $\La^{1-\ga}(\phi_8\La^{\ga}u) \in L^\infty_t L^{\infty-}_x$. 
 
We consider the equation for $\phi_8 \La^{\gamma}u$,
\begin{align*}
(\pa_t + \La^{2s} )( \phi_8 \La^\ga u) 
&= \underbrace{- \La^{\ga}(\phi_8 (u\cdot \na)u + \phi_8 \na p )}_{=:  \La^\ga F_1 + \La^\ga \Ga}
+\underbrace{[\La^{2s},\phi_8 ]\La^{\ga} u + [\La^{\ga},\phi_8 ] ((u\cdot\na) u + \na p)}_{=: G_3} 
+ \underbrace{\La^\ga u\p_t \phi_8}_{=:G_4} 
\end{align*}
where $F_1$ and $\Ga$ are determined by Lemma~\ref{lem_pressure_est}; in particular, 
\[ 
F_1\coloneqq \phi_8 ((u\cdot \na )u + \na \mathcal{R}_{ij} (u_i u_j \phi_{15/2})), \quad \Ga\in L^2_tW^{k,\infty}_x \quad \forall k\in \mathbb{N}\cup\{0\}.
\]

To deal with the last term $G_4$, we remark that $\La^{\ga} u \in L^\infty_tL^p_x + L^2W^{k,\infty}_x(Q_{\frac{15}2})$ for any $k\in \mathbb{N}\cup\{ 0\}$ and $p\in [6,\infty)$. Indeed, $\La^{\ga}(u\phi_7)\in L^\infty_tL^p_x$ follows from \texttt{Step 2} and $\La^{\ga}(u(1-\phi_7)) \in L^2_tW^{k,\infty }(Q_{\frac{15}2})$ (for every $k\geq 0$), which follows from \eqref{trick2a_inlem}. Therefore, $G_4\in L^{\infty}_tL^p_x+L^2_tW^{k,\infty }_x$ for all $p\in [1,\infty)$, $k\geq 0$, which gives the required regularity on $\na (\phi_8 \La^{\ga} u)$.

Moreover, $\La^\ga \Ga\in L^2_tW^{k,\infty }_x$ for any integer $k \geq 0$  because of $\Ga \in L^2_tW^{k,\infty }_x$. (Indeed, $\norm{\La^\ga h}_{L^\infty_x} \lec \norm{h}_{W^{1,\infty}_x}$.) Thus, $\La^\ga \Ga$ gives the required regularity of $\na (\phi_8 \La^{\ga} u)$, as in \eqref{na_w_calc} above.

As for the commutator terms in $G_3$, we obtain  $[(-\De)^s,\phi_\frac32 ]\La^{\ga} u \in L^{\infty}_tL^p_x + L^2_tW^{k,\infty}_x$ for any integer $k \geq 0$ and $p\in [6, \infty)$ by Lemma~\ref{lem_comm.Deu.var1} (applied with any $\kappa \in (2s+\ga -1,2s)$) and \texttt{Step 2}. Similarly, 
by Lemma~\ref{lem_comm.u} together with \eqref{p_bound}, we have $[\La^{\ga},\phi_8 ] ((u\cdot\na) u + \na p)\in L^{\infty}_tL^p_x + L^1_tW^{k,\infty}_x$ for any integer $k \geq 0$ and $p\in [6, \infty)$
Therefore, these terms give the required regularity on $\na (\phi_8 \La^{\ga} u)$ via \eqref{para.est} and \eqref{para.est.0}.

Lastly, we consider $\La^\ga F_1$. We note that $\La^{2s-1-\varepsilon_2 } F_1\in L^{\infty}_tL^p_x$ for every $p\in [6, \infty) $ and $\varepsilon \in (0, 2s]$, as a consequence of \texttt{Step 2} (which gives in particular that $\La^{2s-1-\varepsilon_2}(u\phi_7), \La^{2s-1-\varepsilon_2}\na (u\phi_7)\in L^\infty_t L^p_x$) and the fractional Leibniz rules as in \eqref{first2terms}. Therefore, by \eqref{para.est} (applied with $\alpha \coloneqq 2s-1-\varepsilon_2 - \gamma $, $\overline{\alpha } \coloneqq 1$, $\sigma\coloneqq 0$, $r=\overline{r} \coloneqq \infty$), it gives the required regularity on $\na (\phi_8 \La^{\ga} u)$ when $\varepsilon_2 \in (0, 4s-3)$.\\

\noindent\texttt{Step 4.} We show that $\phi_9\na^2 u \in L^\infty_{t,x}$.\\

It is sufficient to obtain $\na (\phi_9 \na u)\in L^\infty_{t,x}$ because $\na \phi_9\na u = \na \phi_9 \na (u\phi_6) \in  L^\infty_{t,x}$ by \texttt{Step 2b}. 

We consider the equation for $\phi_9\na u$,
\begin{equation}\begin{split}
\label{eqn.Dk_uphi}
(\pa_t + \La^{2s}) (\phi_9 \na u ) 
=&\   \na(\na \cdot  F_2) + \underbrace{[\La^{2s},\phi_9] \na u}_{=: G_5}  \\
&\quad+\underbrace{\na(u\otimes u ) :\na \phi_9 - \na( \nabla \phi_9 \mathcal{R}_{ij} (u_iu_j \phi_{17/2}) )  + \na u\pa_t \phi_{9}+ \na \Ga}_{=:G_6} ,
\end{split}\end{equation}
where $F_2\coloneqq -(u\otimes u)\phi_{9} +  \phi_{9 } \mathcal{R}_{ij} (u_iu_j \phi_{17/2})\I.$ 

Since \texttt{Step 2b} and \eqref{est.remainder} give that $G_6\in L^\infty_tL^p_x+ L^1_t W^{k,\infty}$ for any integer $k\geq 0$ and $p\in [1,\infty)$, it gives
the required regularity on $\na (\phi_9 \na u)$ via \eqref{para.est} and \eqref{para.est.0}. 

By Lemma~\ref{lem_comm.nau} with the results of \texttt{Step 2b} and \texttt{Step 3}, we have $[\La^{2s}, \phi_9]\na u\in L^\infty_tL^{\infty-}_x+ L^2_tW^{k,\infty}$ for integer $k\geq 0$. Therefore, this commutator term also gives the desired regularity via \eqref{para.est} and \eqref{para.est.0}. 

Lastly, we note that $\La^{\ga} \na F\in L^\infty_tL^p_x $ for any $\ga\in (s,1)$ and $p<\infty$, which follows from the fractional Leibniz rule \eqref{fractional_leibniz1}, \texttt{Step 2b, Step 2c, Step 3} and by noting that 
\begin{align*}
 \phi_9\La^{\ga} \na u  = \phi_9  \na (\phi_8\La^{\ga} u  ) \in L^\infty_tL^p_x, \quad \forall  p\in [1,\infty]. 
\end{align*} 
Therefore, by \eqref{para.est} (applied with $\alpha \coloneqq \gamma -1$, $\overline{\alpha } \coloneqq 1$, $\sigma =0$, $r=\overline{r}\coloneqq \infty$), it gives the required regularity of $\na (\phi_9 \na u)$. 
\end{proof}

\subsection{Commutator estimates}\label{sec_commutator_est}
In this section, we prove several commutator estimates of the form $[\La^\be, \phi] G$, used in the bootstrapping argument above. The main difficulty of these estimates is to control the commutators by local information in space, while fractional laplacians involve global information. This results in a number of tail estimates, for which we develop a technique that allows us to estimate the tails of $u$ and $\La^\ga u$, where $\ga \in [s,s+1)$, using only $\mathcal{M} (\La^s u )$ and a local mass of $u$, which we state in the lemma below. These tail estimates are the heart of this section, and will be used repeatedly in the commutator estimates that follow in Lemmas~\ref{lem_1.10}--\ref{lem_comm.u}.  
\begin{lemma}[The main tail estimates]
\label{lem_tricks}
Let $s\in (0,1)$, $R>0$ and $\rho\in (0,R/2)$. Choose $\chi_\rho \in C_c^\infty (B_{2\rho };[0,1])$ satisfying $\chi_\rho =1$ on $B_\rho$. Then, for every integer $k\geq 0$,
\eqnb\label{trick1}
\left\| \int \frac{u(y)}{|x-y|^{3+\be}}  (1-\chi_\rho )(x-y) \d y \right\|_{W^{k,\infty } (B_R )} \lec_{k,R,\rho, \be}  
\| \cM (\La^s u) \|_{L^2(B_R)}
+ \| u\|_{L^1 (B_R)} 
\eqne
and
\eqnb\label{trick2}
\left\| \int \frac{\La^\ga u(y)}{|x-y|^{3+\be}}  (1-\chi_\rho )(x-y) \d y \right\|_{W^{k,\infty } (B_R )} \lec_{k, R,\rho,\be, \ga}  
\| \cM (\La^s u) \|_{L^2(B_R)}
\eqne
for $\be >s$ and $0\le \ga-s<1$. 
\end{lemma}
\begin{rem}\label{rem.trick}
\eqref{trick2} is also valid with $\La^\ga$ replaced by the classical derivative $\na$ (when $\ga=1$). 
\end{rem}

\begin{proof}[Proof of Lemma~\ref{lem_tricks}.]
We consider \eqref{trick1} first. We let  $\varphi \in C_c^\infty (B_R ; [0,\infty ))$ be such that $\int \varphi \, \d x =1$, and we denote $\ph$-mean of $u$ by $(u)_\varphi \coloneqq \int u \varphi \d x$. We will show that 
\begin{align}\label{est.t1.suff}
\int_{|y-x|\leq 2^j \rho} |u(y) -(u)_\varphi| \d y \lec 2^{j(3+s)} \sup_{R' \geq 4R} \dint_{B_{R'}} |\La^s u|
\end{align}
for $x\in  B_R$ and $j\in \mathbb{N}$. Then for each such $x$ and integer $k\geq 0$ 
\begin{equation}\label{est.It11}\begin{split}
\left| \na^k \int \frac{u(y)-(u)_\varphi }{|x-y|^{3+\be}}  (1-\chi_\rho )(x-y) \d y \right| &
\lec_k \sum_{m=0}^k \int_{|y-x|\geq \rho}  \frac{|(u-(u)_\varphi)(y)|}{|x-y|^{3+\be+m}}\d y\\
&= \sum_{m=0}^k\sum_{j\geq 1}
\int_{2^{j-1}\rho\leq |y-x|<  2^j\rho}  \frac{|(u-(u)_\varphi)(y)|}{|x-y|^{3+\be+m}}\d y\\
&\lec_{k,\rho} \sum_{j\geq 1} 2^{-j(3+\be)} 
\int_{ |y-x|\leq  2^j\rho}  |(u-(u)_\varphi)(y)|\d y \\
&\lec  \sum_{j\geq 1} 2^{-j(\be-s)}  \sup_{R' \geq 4R} \dint_{B_{R'}} |\La^s u|\\
&\lec_{\be,s} \sup_{R' \geq 4R} \dint_{B_{R'}} |\La^s u| ,
\end{split}\end{equation}
 where the third line follows from $|x-y|^{-m}\lec_{\rho,m} 1$ (since $|x-y|\geq \rho$). On the other hand, the remaining part with the $\varphi$-mean can be easily estimated by 
 \[
\int \frac{(u)_\varphi }{|x-y|^{3+\be}}  (1-\chi_\rho )(x-y) \d y
\sim_{s, \rho, \be} (u)_\varphi 
 \lec \| u \|_{L^1 (B_R)}
\]
and all its derivatives vanish. Since for every $R'\geq 4R$ 
\begin{equation}\begin{split}
\label{est.max}
\dint_{B_{R'}} |\La^s u(z)| \d z 
&\lec \dint_{B_R}\dint_{B_{R'+R}(w)} |\La^s u(z)| \d z \, \d w \\
&\lec \int_{B_R}\cM(\La^s u)(w)  \d w 
\lec \| \cM (\La^s u) \|_{L^2(B_R)},
\end{split}\end{equation}
the claim \eqref{trick1} follows.

In order to get \eqref{est.t1.suff}, we fix $j \geq 1$ and write $u= \La^{-s}(\La^s u)$ to get 
\begin{equation}\begin{split}
\label{repr_u_for_tail}
(u-(u)_\varphi)&(y) = \int (u(y) - u(w)) \varphi(w) \d w
=\int (\La^{-s} (\La^s u)(y) - \La^{-s} (\La^s u)(w)) \varphi(w) \d w \\
&\sim_s\iint \left( \frac {1}{|y-z|^{(3-s)}} - \frac {1}{|w-z |^{(3-s)}} \right)  
\La^s u(z) \varphi (w)
\d z \,\d w\\
&=\int_{|z-x|\leq 2^{j+n_0}\rho}  \frac {\La^s u(z)}{|y-z|^{(3-s)}}  
\d z
- \int_{|z-x| \leq 2^{j+n_0}\rho}  \int \frac {\varphi(w)}{|w-z |^{(3-s)}} \d w\,\La^s u(z)  \d z
\\
&\quad +  \iint_{|z-x|> 2^{j+n_0}\rho } \left(\frac {1}{|y-z|^{(3-s)}} - \frac {1}{|w-z |^{(3-s)}} \right)
\La^s u(z)\varphi(w) 
\d z \,\d w \\
&= u_{\loc,1}(y) + u_{\loc,2}  + u_\tail (y)
\end{split}\end{equation}
for $y\in B_{2^j\rho }(x)$, where $n_0=n_0(\rho,R)$ is the smallest integer satisfying $2^{n_0}\rho \geq 4R$. In particular, $n_0\geq 4$ (since $\rho<2R$) and $2^{n_0}\rho < 8R$. As for $u_{\loc,1}$, we have
\begin{equation}\begin{split}
\label{est.uloc1}
\int_{B_{2^j \rho}(x)} |u_{\loc,1}(y)|\d y
&\lec
\int_{B_{2^{j+n_0}\rho}(x)} \int_{B_{2^j \rho}(x)}
\frac {1}{|y-z|^{(3-s)}} \d y \,|\La^s u(z)|\d z \\
&\lec_{s,n_0,\rho} 2^{j(3+s)}\dint_{B_{2^{j+n_0}\rho}(x)} |\La^s u| \d z\\
&\lec 2^{j(3+s)}\dint_{B_{2^{j+n_0+1}\rho}} |\La^s u| \d z\\
&\lec 2^{j(3+s)}\sup_{R'\geq 4R} \dint_{B_{R'}} |\La^s u| \d z ,
\end{split}\end{equation}
where the second inequality follows from the inequality $|y-z|\lec |y-x|+ |z-x|\leq 2^{j}\rho +2^{j+n_0}\rho \leq 2^jc\rho$ (so that $B_{2^j \rho }(x) \subset B_{2^jc\rho}(z)$ when $z\in B_{2^{j+n_0}\rho}(x)$) and the third inequality follows from the fact that $|x|\leq R\leq 2^{n_0-2}\rho$. As for $u_{\loc,2}$, recalling that $x\in B_R$, we obtain
\begin{equation}\begin{split}\label{est.uloc2}
|u_{\loc, 2}|
&\lec \left( \int_{B_{2^{n_0}\rho}(x)}+ \sum_{l=n_0+1}^{j+n_0} \int_{B_{2^{l}\rho}(x)\setminus B_{2^{l-1}\rho}(x)}\right)  
\int \frac {\varphi(w)}{|w-z |^{(3-s)}} \d w \,|\La^s u(z)| \d z\\
&\lec \int_{B_{8R}(x)} \int_{|w-z|\leq 10R} \frac {1}{|w-z |^{(3-s)}} \d w\, |\La^s u(z)| \d z \\
&\hspace{5cm}+ \sum_{l=n_0+1}^{j+n_0} \int_{B_{2^{l}\rho}(x)\setminus B_{2^{l-1}\rho}(x)}\frac {|\La^s u(z)|}{|x-z |^{(3-s)}}   \d z\\
&\lec_{s,R,\rho}  \int_{B_{8R}(x)} |\La^s u|  
+ \sum_{l=n_0+1}^{j+n_0} 2^{l(-3+s)} \int_{B_{2^l\rho }(x)}  |\La^s u(z)| \d z\\
&\lec  \int_{B_{9R}} |\La^s u|  
+ \sum_{l=n_0+1}^{j+n_0} 2^{l(-3+s)} \int_{B_{2^{l+1} \rho}}  |\La^s u(z)| \d z\\
&\lec_\rho 2^{js}\sup_{R'\ge 4R} \dint_{B_{R'}}  |\La^s u|  \d z.
\end{split}\end{equation}
Here, the second inequality follows from $|w-z|\leq |w|+|z-x|+|x|\leq R+8R+R=10R$ (in the first term) and $|w-z| \geq |z-x|-|x|-|w|\geq |z-x|-2R\geq \frac {|x-z|}2$ (in the second term). In the fourth inequality, we used the fact that $|z-x|\geq 2^{n_0}\rho \geq 4R$ which implies that $|z| \leq |z-x| + |x| \leq 2^l \rho + R\leq 2^l \rho + \frac 14\cdot 2^{n_0} \rho \leq 2^{l+1}\rho$ ,

As for $u_\tail$, we decompose the integral region $(B_{2^{j+n_0}\rho}(x))^c$ into sets $\{ z \colon 2^{l-1}\rho<|z-x|\leq 2^l\rho \}$ for $l\geq j+n_0+1$, and we note that on each such set we have
\[
\left| \frac 1{|y-z|^{(3-s)}} -\frac 1{ |w-z|^{(3-s)}} \right|
\lec_s \frac { |w-y|}{|\th y + (1-\th) w - z|^{(4-s)}} 
\lec 2^j2^{-l(4-s)}
\]
for some $\th\in[0,1]$. Indeed, here we used the fact that $|w-y| \leq |x-y|+ |x|+|w|\leq 2^j \rho + 2R \lec_{\rho,R} 2^j $ for $j\geq 1$ and the fact that
\[\begin{split}
| \th y + (1-\th) w - z| &\geq |z-x| - \theta |y-x| - (1-\theta )|w-x|  \geq 2^{l-1}\rho - 2^{j}\rho-2R \\
&\geq 2^{l}\rho \left(\frac 12 - 2^{j-l}-2^{n_0-l-1} \right)\geq 2^{l}\rho \left(\frac 12 - 2^{-5}-2^{-3} \right) \\
&\geq \frac14  \cdot 2^{l}\rho.
\end{split}
\]
where the second line follows from the choice of $n_0$ ($n_0\geq 4$ and $R\leq 2^{n_0-2}$). Therefore, recalling that $x\in B_R$, we have
\begin{equation}\begin{split}\label{est.utail}
|u_\tail (y) | 
&\lec 2^j \sum_{l\geq j+n_0+1} 2^{-l(4-s)} \int_{B_{2^{l}\rho}(x)} |\La^s u | \\
&\lec 2^j \sum_{l\geq j+n_0+1} 2^{-l(4-s)} \int_{B_{2^{l+1}\rho}} |\La^s u |  \\
&\lec 2^j \sum_{l\geq j+3} 2^{-l(1-s)}  
\sup_{R'\geq 4R} \dint_{B_{R'}} |\La^s u|\\
&\lec 2^{js} \sup_{R'\geq 4R} \dint_{B_{R'}} |\La^s u|
\end{split}\end{equation}
for every $y \in B_{2^j \rho }(x)$. Combining \eqref{est.uloc1}--\eqref{est.utail}, we obtain \eqref{est.t1.suff}, as required.

\

Now, we consider \eqref{trick2}. The case $\ga=s$ can be obtained easily as follows;
\[\begin{split}
\left| \na^k \int \frac{\La^s u(y)}{|x-y|^{3+\be}}  (1-\chi_\rho )(x-y) \d y\right| & = \left| \int \La^s u(y) \na^k \left( \frac{(1-\chi_\rho )(x-y)}{|x-y|^{3+\be}} \right)  \d y \right| \\
&\hspace{-3cm}\lec \sum_{m=0}^k \left(
\int_{\rho < |y-x|\leq 2^{n_0} \rho }
+\sum_{j\geq n_0 }  \int_{2^j \rho < |y-x|\leq 2^{j+1}\rho }\right)
 |\La^s u(y) |  |x-y|^{-(3+\be+m)}  \d y \\
&\hspace{-3cm}\lec_{s,\rho, \be}   \int_{|y|\leq 9R  } |\La^s u(y) |   \d y +  \sum_{j\geq n_0 } 2^{-j(3+\be)} \int_{ |y|\leq 2^{j+2}\rho } |\La^s u(y) |   \d y  \\
&\hspace{-3cm}\lec_{\rho,R}  \left( 1+ \sum_{j\geq n_0} 2^{-\be j} \right) \sup_{R'\geq 4R } \dint_{B_{R'}} |\La^s u | \d y  \\
&\hspace{-3cm}\lec_{n_0}  \sup_{R'\geq 4R } \dint_{B_{R'}} |\La^s u |  \d y
\end{split}\]
for every integer $k\geq 0$, $x\in B_R$, where, the second inequality follows from $|y|\leq |y-x|+|x|\leq 2^{n_0 }\rho + R \leq 9R$ (in the first term) and the fact that $|y|\leq |y-x| + |x| \leq 2^{j+1} \rho + R \leq 2^{j+1}\rho + 2^{n_0-2}\rho \leq 2^{j+2} \rho$ (in the second term).\\

As for the case $\ga >s $, we first set $\overline{\eta } (y) \coloneqq \chi_\rho (y ) - \chi_{\rho } (2y )$ and \[
\eta_j (y)\coloneqq \frac{2^{\be j }}{|y|^{3+\be}} \overline{\eta }(2^{-j }y)\qquad \text{ for }j\in \ZZ.\]
Then $\eta_j(y)= 2^{-3j} \eta_0(2^{-j} y)$, so that $\norm{ \eta_j}_1 = \norm{\eta_0}_1$, $\supp \, \eta_j \subset B_{2^{j+1}\rho }\setminus B_{2^{j-1}\rho }$ and $|\na^k \eta_j |\lec_k 2^{j(-3-k)}$ for every $j\in \ZZ$, integer $k\geq 0$, and 
\eqnb\label{eta_j_for_large_y_for_tricks}
|\La^\al \na^k \eta_j (y) |\lec_{\al} |y|^{-(3+k+\al)} \qquad \text{ for } |y| \geq 2^{j+2}\rho, \ \al\in (-1,1).
\eqne
The case $a=0$ trivially holds while $a\in (0,1)$ can be 
verified by writing 
\[
\begin{split}
|\Lambda^\al \na^k \eta_j (y) | &\sim_\al \left| \mathrm{p.v.} \int \frac{\na^k \eta_j (y) -  \na^k \eta_j (z)}{|z-y|^{3+\al }}  \d z\right|\\
&=  \left| \int \eta_j (z) \na^k \left( \frac{1}{|z-y|^{3+\al }} \right) \d z \right| \\
&\lec  \int   \frac{\eta_j (z)}{|z-y|^{3+\al+k }}  \d z  \\
&\lec \frac{1}{|y|^{3+\al +k }} \int |\eta_j| \d z\sim  \frac{1}{|y|^{3+\al +k}},
\end{split}
\]
as required, where, in the second line, we used that $\eta_j(y) =0$ for $|y|\geq 2^{j+2}\rho$ as well as integration by parts (which is allowed since $\eta_j(z)=0$ when $|z-y|\leq 2^{j+1}\rho $ (as then $|z|\geq |y|-|z-y|\geq 2^{j+1}\rho $)), and, in the fourth line, we used the inequality $|z-y|\geq |y|-|z| \geq |y|-2^{j+1}\rho \geq \frac{|y|}2$. The case $\al\in (-1,0)$ follows by skipping the first line. 

Using the auxiliary functions $\eta_j$ we can write 
\[
\frac{1-\chi_{\rho } (x-y)}{|x-y|^{3+\be}} = \sum_{j\geq 1} 2^{-\be j } \eta_j(x-y),
\]
and obtain that, for every integer $k\geq 0$  and every $x\in B_R$,
\eqnb\label{trick2_calculation}\begin{split} 
&\left| \na^k \int \frac{\La^\ga u(y)}{|x-y|^{3+\be}}  (1-\chi_\rho )(x-y) \d y \right| = \left| \sum_{j\geq 1 } 2^{-\be j } \na^k (\eta_j \ast \La^\ga u) (x)\right| \\
&\quad\leq   \sum_{j\geq 1 } 2^{-\be j } \left| (\na^k \La^{\ga-s } \eta_j \ast \La^s u) (x)\right| \\
&\quad\leq   \sum_{m=0}^k \sum_{j\geq 1} 2^{-\be j} \left( 
\left| \int_{|x-y|\leq 2^{j-n_0+5}R }  \na^{m} \La^{\ga-s} \eta_j(x-y) \La^s u(y) \d y\right|  \right.\\
&\hspace{3cm}+ \sum_{l\geq j-n_0+5}\left.
\left| \int_{2^lR < |x-y|\leq 2^{l+1}R}\na^{m} \La^{\ga-s}\eta_j(x-y) \La^s u(y) \d y\right| \right) \\
&\quad\lec_{R,k} \sum_{m=0}^k \sum_{j\geq 1} 2^{-\be j } \left( 2^{j(-3-m-\ga+s)} 
 \int_{|y|\leq 2^{j+2}R} | \La^s u(y) |\d y   \right.\\
&\hspace{3cm}+ \sum_{l\geq j-n_0+5}2^{l(-3-m-\ga+s)} \left.
 \int_{|y|\leq 2^{l+n_0+2}R}| \La^s u(y) |\d y  \right) \\
&\quad\lec_R  \sum_{j\geq 1} 2^{-\be j} \left( 2^{j(-\ga+s)} 
\dint_{B_{2^{j+2}R}} | \La^s u | \d y+ \sum_{l\geq j-n_0+5}2^{l(-\ga+s)} \dint_{2^{l+n_0+2}R}| \La^s u | \d y\right) \\
&\quad\lec  \sup_{R'\geq 4R }\dint_{B_{R'}}| \La^s u | \d y
\end{split}
\eqne
(recall $n_0\geq 4$ is such that $4R\leq 2^{n_0} \rho < 8 R$), as required, where, in the third inequality, we used the bound $|\na^m \La^{\ga-s} \eta_j |\lec_{m,\ga,s} 2^{j(-3-m-\ga+s)}$ and $|y|\leq |y-x|+|x| \leq 2^{j+1}R +R \leq 2^{j+2}R$ (for the first term) as well as \eqref{eta_j_for_large_y_for_tricks} (which is allowed since $|x-y|>2^l R \geq 2^{j-n_0+5} R \geq 2^{j+2}\rho$ by definition of $n_0$) together with the lower bound $|x-y|\gtrsim_R  2^l$ and the inequality $|y| \leq |y-x|+|x| \leq 2^{l+1}R +R \leq 2^{l+2}R$ (for the second term).
\end{proof}

We now move on to the commutator estimates. In  Lemmas~\ref{lem_1.10}--\ref{lem_comm.p}, we consider the commutators of the form $[\La^\be, \phi] G$ where $\be\in (1/2,2)$, $\phi=\phi(x,t)$ is a smooth function supported on $B_{R}\times (-T,0]$ for some $R>0$ and $T>0$, and $G$ is chosen differently in each lemma. Also, we introduce $R_0>R$. Since we use the finite number of candidates for $R,R_0,\phi, \be$ in the bootstrapping argument, we ignore their dependence in the implicit constants of the commutator estimates. We also ignore the dependence on $s$. 
 
\begin{lemma}\label{lem_1.10} Let $s\in (\frac12,1)$. Let $\phi=\phi(x,t)$ be a smooth function compactly supported on $B_{R}\times (-T,0]$ for some $R>0$ and $T>0$. Let $R_0>R$. Then, for any $\bar\varepsilon\in (0,2-2s)$, $r\in [1,\infty ]$ and $p\in(1,\infty )$, we have a decomposition
\[
[(-\De)^s, \phi] u = f + g
\]
where $f$ and $g$ satisfy
\begin{align*}
\norm{f}_{L^r(-T,0;L^p)} &\lec_{\bar\varepsilon} 
\norm{u}_{L^r(-T,0;W^{2s-1+\bar\varepsilon,p}(B_{R_0}))}\\
\norm{g}_{L^2(-T,0; W^{k,\infty})} &\lec_{k} \norm{\cM (\La^s u)}_{L^2([-T,0]\times B_{R})}+\norm{u}_{L^2(-T,0;L^1 (B_R))}, \quad \forall k\in \mathbb{N}\cup \{0\}.
\end{align*}
Furthermore, $g$ is compactly supported in $B_{R}\times (-T,0)$, which gives that
\begin{align*}
\norm{g}_{L^2(-T,0; W^{k,q})} &\lec_{k} \norm{\cM (\La^s u)}_{L^2([-T,0]\times B_{R})}+\norm{u}_{L^2(-T,0;L^1 (B_R))}, \quad \forall k\in \mathbb{N}\cup \{0\}, \ q\in [1,\infty].
\end{align*} 
\end{lemma}

We note that the lemma is true for any $\bar\varepsilon >0$, but (for brevity) we restrict ourselves to $\bar\varepsilon < 2-2s$ since then the Sobolev-Slobodeckij space $W^{2s-1+\bar\varepsilon,p}$ is of order less than $1$. We set
\eqnb\label{def_of_rho}
\rho \coloneqq \frac 15(R_0-R )
\eqne
We may assume $R_0<\frac 72 R$ to achieve $2\rho<R$; otherwise we choose $R_0'$ satisfying $R_0'<\frac 72 R$ instead of $R_0$ and then the desired estimates follow by expanding the domain.
\begin{proof} For the convenience, we omit the variable $t$ of $u$ and $\phi$ unless it is needed. Using the definition \eqref{def.frac}, the commutator $[(-\De)^s, \phi] u$ can be written as
\begin{align*}
[(-\De)^s, \phi] u(x) 
&\sim  \left( \text{p.v.}\int \frac{\phi(x) u(x) - \phi(y) u(y)}{|x-y|^{3+2s}} \d y - \text{p.v.} \int \frac{\phi(x) ( u(x) - u(y) ) }{|x-y|^{3+2s}} \d y\right)\\
&=  \text{ p.v.}\int \frac{u(y)( \phi(x) - \phi (y) ) }{|x-y|^{3+2s}} \d y.
\end{align*} 
We first decompose the integral in the last line into local and tail parts,
\begin{align*}
[(-\De)^s, \phi] u (x)
&\sim \text{ p.v.} \int \frac{u(y)( \phi(x) - \phi (y) ) }{|x-y|^{3+2s}} \chi_\rho(x-y) \d y +  \int\frac{u(y)( \phi(x) - \phi (y) ) }{|x-y|^{3+2s}} (1-\chi_\rho)(x-y) \d y\\
&=: I_\ell + I_t,
\end{align*}
where $\chi_\rho=\chi_\rho(x)$ is a radial function in space satisfying $\chi_\rho =1$ on $B_\rho$ and supported on $B_{2\rho}$. 

\noindent\texttt{Step 1.} We estimate the local part $I_\ell$. \\

We decompose $I_\ell$ by writing
\begin{equation}\begin{split}
\label{supp_comment_below}
I_\ell (x) 
&= \text{ p.v.}\int \frac{u(x)(\phi (x) - \phi(y))}{|x-y|^{3+2s}} \chi_\rho(x-y) \d y \\
&\quad + \text{ p.v.}\int \frac{(u(y)-u(x))(\phi (x) - \phi(y))}{|x-y|^{3+2s}}  \chi_\rho(x-y)\d y \\
&=: I_{\ell 1} (x)  + I_{\ell 2} (x)
\end{split}\end{equation}
We note that both $I_{\ell 1}$ and $I_{\ell 2}$ are supported on $(-T_0,0)\times B_{R+2\rho}$. Indeed, if $|x|\geq R+2\rho$, we have $|x|\geq R$ and $|y|\geq |x|-|y-x|\geq (R+2\rho ) - 2\rho =R$, which make $\phi(x)-\phi(y)$ vanish. Moreover using $\text{p.v.}\int \frac{(x-y)\chi_\rho (x-y)}{|x-y|^{3+2s}} \d y =0$, we see that $\phi$ satisfies
\begin{align*}
&\left| \text{p.v.}\int \frac{\phi (x,t) - \phi(y,t)}{|x-y|^{3+2s}}  \cdot \chi_\rho(x-y) \d y \right|\\
&=\left|\text{p.v.}\int \frac{\phi (x,t) - \phi(y,t)-(x-y)\cdot \na \phi(x,t)}{|x-y|^{3+2s}} \cdot \chi_\rho(x-y)\d y \right|\\
&\lec \norm{\phi}_{C(-T,0;C^2(\R^3))} , 
\qquad \forall (x,t)\in \R^3 \times [-T,0].
\end{align*}
Thus we can estimate $I_{\ell 1}$,
\begin{align}\label{est.Iell1}
\norm{I_{\ell 1}}_{L^r(-T,0;L^p)}
=\norm{I_{\ell 1}}_{L^r(-T,0;L^p(B_{R+2\rho}))}
\lec \norm{u}_{L^r(-T,0;L^p(B_{R+2\rho}))}.
\end{align} 
As for $I_{\ell 2}$, we use the following estimate for $\phi$: for any $\bar\varepsilon>0$ and $1\le q<\infty$, 
\begin{align*}
\left|
\int_{|y-x|\le 2\rho} \frac{|\phi(x,t)-\phi(y,t)|^{q}}{|x-y|^{3+(1-\bar\varepsilon)q}} \d y \right|
\lec_{\bar\varepsilon} \norm{\phi}_{C(-T,0;C^1(\R^3))}^q, \quad \forall (x,t)\in \R^3\times (-T,0),
\end{align*}
while for $q=\infty$ we use
\[
\sup_{|y-x|\le 2\rho} \frac{|\phi(x,t)-\phi(y,t)|}{|x-y|^{1-\bar\varepsilon}}\lec _{\bar\varepsilon} \norm{\phi}_{C(-T,0;C^1(\R^3))}. 
\]
Using H\"{o}lder's inequality, this gives for $p\in [1,\infty)$ 
\begin{align*}
|I_{\ell 2}(x)|^p
&\lec  \int_{|y-x|\le 2\rho} \frac{|u(y)-u(x)|^p}{|x-y|^{3+(2s-1+\bar\varepsilon)p}} \d y
\left( \int_{|y-x|\le 2\rho} \frac{|\phi (y) - \phi (x)|^{p'}}{|x-y|^{3+(1-\bar\varepsilon){p'}}} \d y
\right)^{\frac p{p'}}\\
&\lec_{\bar\varepsilon}  \int_{|y|\le R+4\rho} \frac{|u(y)-u(x)|^p}{|x-y|^{3+(2s-1+\bar\varepsilon)p}} \d y,
\end{align*}
where $p'$ is the H\"{o}lder conjugate of $p$, and hence
\begin{align}\label{est.Iell2}
\norm{I_{\ell 2}}_{L^r(-T,0;L^p)} 
&= \norm{I_{\ell 2}}_{L^r(-T,0;L^p(B_{R+2\rho}))} 
\lec \norm{u}_{L^r(-T,0;W^{2s-1+\bar\varepsilon,p}(B_{R+4\rho}))}.
\end{align}
Therefore, combining \eqref{est.Iell1} and \eqref{est.Iell2}, we have
\[
\norm{I_{\ell}}_{L^r(-T,0;L^p)} 
\lec_{\bar\varepsilon} \norm{u}_{L^r(-T,0;W^{2s-1+\bar\varepsilon,p}(B_{R_0}))}.
\]

\noindent\texttt{Step 2.} We estimate the tail part $I_t$.\\

We decompose $I_t$,
\begin{align*}
I_t &= C_s \int \frac{u(y)( \phi(x) - \phi (y) ) }{|x-y|^{3+2s}} (1-\chi_\rho)(x-y) \d y\\
&=C_s \ \phi(x)\int  \frac{u(y)}{|x-y|^{3+2s}}  (1-\chi_\rho)(x-y)\d y
-C_s \int  \frac{u(y)\phi (y)  }{|x-y|^{3+2s}} (1-\chi_\rho)(x-y) \d y\\
&=: I_{t1} + I_{t2},
\end{align*}
and we show below that  
\begin{align*}
\norm{I_{t1}}_{W^{k,\infty }}
&\lec \norm{\cM (\La^s u)}_{L^2( B_{R})}+\norm{u}_{L^1(B_R)}, \quad \forall t\\
\norm{I_{t2}}_{L^r(-T,0;L^p)} 
&\lec
  \norm{u}_{L^r(-T,0;L^p(B_{R_0}))}.
\end{align*}
This concludes the lemma by letting 
 \[
 f\coloneqq I_{\ell 1} + I_{\ell 2} +  I_{t2} \quad
 g\coloneqq I_{t1}.
 \]

The first term $I_{t1}$ is supported in $B_R\times (-T,0]$ and so can be estimated as in \eqref{trick1} with $\be:= 2s$,
\[
\| I_{t1} \|_{L^2(-T,0;W^{k,\infty })} \lec \norm{\cM (\La^s u)}_{L^2([-T,0]\times B_{R})}+\norm{u}_{L^2(-T,0;L^1 (B_R))}.
\]
As for $I_{t2}$ we see that, since $|x-y|>\rho$ and $\supp_x(\phi)\subset B_R$, it can be bounded for any $x\in \R^3$, 
\[
|I_{t2}(x)| \lec  \int |u(y)||\phi (y)| \d y \lec
\norm{ u }_{L^1 (B_R)}.
\]
Moreover, it also satisfies a decay estimate
\[
|I_{t2} (x) | 
\lec \int \frac{|u(y) \phi (y) | }{|x-y|^{3+2s}}\d y 
\lec \frac1{|x|^{3+2s}} \norm{u}_{L^1 (B_R)}, \quad \forall |x| \geq 2R
\]
because of $|x-y|\geq |x|-|y| \geq |x|-R \geq |x|/2$. 
Combining the two inequalities we get
\begin{align}\label{est.It2}
\norm{I_{t2}}_{L^r(-T,0;L^p)} 
\lec 
\norm{u}_{L^r(-T,0;L^p(B_R))},
\end{align}
as required.
\end{proof}

In the following lemmas, we keep using a decomposition of a commutator suggested in the proof above. Namely, given a function $G$ we use the decomposition
\begin{align*}
[\La^\be, \phi] G = I_{\ell 1}  + I_{\ell 2} + I_{t1} + I_{t2},
\end{align*}
where
\eqnb\label{decomposition_of_I}
\begin{split}
I_{\ell 1} (x)&\coloneqq  C_\be \,\text{ p.v.}\int \frac{G(x)(\phi (x) - \phi(y))}{|x-y|^{3+\be}} \chi_\rho(x-y) \d y \\
I_{\ell 2} (x)&\coloneqq C_\be\, \text{ p.v.}\int \frac{(G(y)-G(x))(\phi (x) - \phi(y))}{|x-y|^{3+\be}}  \chi_\rho(x-y)\d y \\
I_{t1} (x)&\coloneqq C_\be\, \phi(x)\int  \frac{G(y)}{|x-y|^{3+\be}}  (1-\chi_\rho)(x-y)\d y \\
I_{t2} (x)&\coloneqq -C_\be \int  \frac{G(y)\phi (y)  }{|x-y|^{3+\be}} (1-\chi_\rho)(x-y) \d y,
\end{split}
\eqne
where $\phi \in C_c^\infty (B_R \times (-T,0))$, $\rho=\frac15(R_0-R)$, $\chi_\rho \in C_c^\infty (B_{2\rho} )$ with $\chi_\rho =1$ on $B_\rho$. We recall that the local terms, $I_{\ell 1}$, $I_{\ell 2}$ are supported in $B_{R+2\rho }$ (regardless of $G$).
As in the proof of Lemma~\ref{lem_1.10}, we may assume $2\rho<R$. 

\begin{lem}\label{lem_comm.Deu.var1} 
Let $s\in (\frac12, 1)$ and $R_0>R$. Let $\phi=\phi(x,t)$ be a smooth function compactly supported on $B_{R}\times (-T,0]$ for some $R>0$ and $T>0$. 
Then, given $\ga\in (s,1)$,  $\ka-\ga\in (2s-1, 1)$, and $\bar\phi=\bar\phi(x,t)\in C_c^\infty(\R^4)$ with $\bar\phi=1$ on $[-T,0]\times B_{R_0}$, for any $r\in [1,  \infty]$, and $p\in  [2,\infty)$, we have  a decomposition 
\begin{align*}
[(-\De)^s, \phi] \La^\ga u  = f_0  + g_0,
\end{align*}
where $f_0$ and $g_0$ satisfy
\begin{align*}
\norm{f_0}_{L^r(-T,0;L^p)}
&\lec \norm{\La^\ka (u\bar\phi)}_{L^r(-T,0; L^p)}
+\norm{\La^\ga (u\bar\phi ) }_{L^r(-T,0; L^p)}
\\
\norm{g_0}_{L^2(-T,0;W^{k,\infty})}
&\lec \norm{\cM (\La^s u)}_{L^2([-T,0]\times B_{R_0})} + \norm{ u}_{L^2(-T,0;L^1(\supp_x \bar\phi ))}
\quad \text{ for } k\in\mathbb{N}\cup \{0\} .
\end{align*} 
Furthermore for every integer $k\geq 0$
\eqnb\label{trick2a_inlem}
\| \La^\ga (u(1-\bar\phi))  \|_{L^2(-T,0;W^{k,\infty } (B_{\frac25 R + \frac35 R_0}))} \lec  \norm{\cM (\La^s u)}_{L^2([-T,0]\times B_{R_0})} + \norm{ u}_{L^2(-T,0;L^1(\supp_x \bar\phi ))}.
\eqne
\end{lem}

\begin{rem}
The motivation of the terms $\La^\ka (u\bar\phi)$ and $\La^\ga (u\bar\phi )$ (rather than $\La^\ka u$ and $\La^\ga u$) comes from the bootstrapping argument (see \texttt{Step 3} in the proof of Theorem~\ref{thm_bootstrapping}.) These terms are the reason why the above lemma cannot be proved in the same way as Lemma~\ref{lem_1.10} by replacing \eqref{trick1} by \eqref{trick2} (in estimating $I_{t1}$). Instead we need to estimate additional error terms of the form of \eqref{trick2a_inlem}, which we include as part of $g_0$.
\end{rem}

\begin{proof}[Proof of Lemma~\ref{lem_comm.Deu.var1}.] For the convenience, we omit the variable $t$ of $u$ and $\phi$ unless it is needed.
We use the decomposition \eqref{decomposition_of_I} with $G\coloneqq \La^\ga u$ and $\be:=2s$ to obtain 
\begin{align*}
[(-\De)^s, \phi] \La^\ga u 
=  I_{\ell 1} + I_{\ell 2} + I_{t1} + I_{t2}.
\end{align*}
First, $I_{t1}$ can be estimated using \eqref{trick2} with $\be:=2s$ as
\[
\norm{I_{t1}}_{L^2(-T,0;W^{k,\infty})}\lec \norm{\cM (\La^s u)}_{L^2([-T,0]\times B_{R_0})}
\]
for every integer $k\geq 0$. As for the other terms, we further decompose, writing $\La^\ga u = \La^\ga (u\bar\phi ) + \La^\ga (u(1-\bar\phi ) ) $. We denote the corresponding decomposition by
\[
I_{\ell 1} + I_{\ell 2}  + I_{t2} = (I_{\ell 11} +I_{\ell 12} )  +( I_{\ell 21}  +I_{\ell 22}  )+ (I_{t21} + I_{t22}). 
\] 
We will estimate the local parts (i.e. the ones with subscript ending with ``$1$'') in \texttt{Step 1}, and the tail parts (i.e. the ones with subscript ending with ``$2$'') in \texttt{Step 2}. \\

\noindent\texttt{Step 1.} We estimate for the local parts, that is we show that 
\[
\| I_{\ell 11} + I_{\ell 21}+ I_{t21} \|_{L^r (-T,0;L^p )} \lec \norm{\La^\ga (u\bar\phi ) }_{L^r(-T,0; L^p(B_{R_0}))} + \norm{\La^\ka (u\bar\phi ) }_{L^r(-T,0; L^p)}.
\]

Indeed, the same calculations as in \eqref{est.Iell1}, \eqref{est.Iell2}, \eqref{est.It2} give  
\[
\| I_{\ell 11} + I_{\ell 21}+ I_{t21} \|_{L^r (-T,0;L^p )} \lec \norm{\La^\ga (u\bar\phi ) }_{L^r(-T,0; L^p(B_{R_0}))} + \norm{\La^\ga (u\bar\phi)}_{L^r(-T,0; W^{2s-1+\bar\varepsilon,p}(B_{R+4\rho}))}
\]
for $\bar\varepsilon \coloneqq \ka - \ga -2s+1\in (0,2-2s)$. Using \eqref{norm_comparisons}, we then bound the last term by $\norm{\La^\ka (u\bar\phi ) }_{L^r(-T,0; L^p)}$ and $\norm{\La^\ga (u\bar\phi)}_{L^r(-T,0; L^p)}$, which concludes this step.\\

\noindent\texttt{Step 2.} Estimate for the tail parts, that is we show that 
\[
\|I_{\ell 12} +I_{\ell 22}+ I_{t22} \|_{L^2(-T,0;W^{k,\infty })}  \lec \norm{\cM (\La^s u)}_{L^2([-T,0]\times B_{R_0})} + \| u \|_{L^2(-T,0;L^1 (\supp \bar\phi ))}.
\] 
(This completes the proof of the lemma by setting 
\[
f_0\coloneqq I_{\ell 11} + I_{\ell 21}+ I_{t21},\quad g_0\coloneqq I_{t1} + I_{\ell 12} +I_{\ell 22}+ I_{t22}.)
\]

Recall that
\[
\begin{split}
I_{\ell 12}  (x)  &\sim_s  \La^\ga (u(1-\bar\phi))(x)\,\mathrm{p.v.}\int \frac{\phi (x-y) - \phi(x)}{|y|^{3+2s}} \chi_\rho(y) \d y \\
&= \La^\ga (u(1-\bar\phi))(x) \td \chi(x) \left(  \La^{2s} \phi(x) - \int \frac{\phi (x-y) - \phi(x)}{|y|^{3+2s}} (1- \chi_\rho(y)) \d y \right) \\
 I_{\ell 22} (x) &\sim_s \, \text{ p.v.}\int \frac{(\La^\ga(u(1-\bar\phi ))(x-y)-\La^\ga(u(1-\bar\phi ))(x))(\phi (x) - \phi(x-y))}{|y|^{3+2s}}  \chi_\rho(y)\d y ,\\
 I_{t22} (x) &\sim_s - \int  \frac{\La^\ga(u(1-\bar\phi ))(y)\phi (y)  }{|x-y|^{3+2s}} (1-\chi_\rho)(x-y) \d y.
\end{split}
\]
where $\td\chi \in C_c^\infty (B_{R_0};[0,1])$ is such that $\chi =1$ on $B_{R+2\rho}$ (recall that $I_{\ell 12}$ vanishes for $|x|>R+2\rho $).

We first reduce our aim to showing that
\eqnb\label{trick2a}
\| \La^\ga (u(1-\bar\phi))  \|_{W^{k,\infty } (B_{R+3\rho })} \lec  \sup_{R'\geq 4R }\dint_{B_{R'}}| \La^s u | +  \| u \|_{L^1 (\supp \bar\phi )},
\eqne
which will be obtained in \texttt{Step 2a} below.

Assuming \eqref{trick2a} we have
\[
\|  I_{\ell 12}   \|_{W^{k,\infty }} \lec \| \La^\ga (u(1-\bar\phi))  \|_{W^{k,\infty } (B_{R+3\rho })} \left( \| \La^{2s} \phi   \|_{W^{k,\infty } (B_{R+3\rho })} + \|  \phi   \|_{W^{k,\infty }} \right)
\]
for every integer $k\geq 0$, and so the required estimate on $I_{\ell 12}$ follows from \eqref{est.max}. As for $I_{t22} $,
\[
|\na^k I_{t22} (x) | \leq \| \La^\ga (u(1-\bar\phi))  \|_{L^{\infty } (B_{R})} \int \left|  \na^k \left( \frac{1-\chi_\rho (x-y) }{|x-y|^{3+2s}} \right)  \right| \d y \lec_{k,\rho} \| \La^\ga (u(1-\bar\phi))  \|_{L^{\infty } (B_{R})}
\]
for every integer $k\geq 0$, $x\in \R^3$, and so the required estimate on $I_{t22}$ follows from \eqref{trick2a} and \eqref{est.max} as well.

Lastly,  we recall that $ I_{\ell 22} $ is supported on $\{|x|\leq R+2\rho\}$, and for such $x$
\[
\La^\ga(u(1-\bar\phi ))(x) \sim_\ga\, \mathrm{ p.v.} \int \frac{u(1-\bar\phi)(x)-u(1-\bar\phi)(z)}{|x-z|^{3+\ga}} \d z
= - \int \frac{u(1-\bar\phi)(z)}{|x-z|^{3+\ga}} \d z
\]
because $\bar\phi =1$ on $B_{R+5\rho }$ and similarly with $x $ replaced by $x-y$ (note that $|x-y|\leq R+2\rho$ since $|y|\leq 2\rho$ and either $|x|\leq R$ or $|x-y|\leq R$ (as otherwise $\phi(x)-\phi(x-y)$ vanishes)). This implies that
\eqnb\label{temp99}\begin{split}
\La^\ga(u(1-\bar\phi ))(x-y)&-\La^\ga(u(1-\bar\phi ))(x) \\
&\sim_\ga -\int \left[ \frac 1{|x-y-z|^{3+\ga}}-\frac 1{|x-z|^{3+\ga}} \right] u(1-\bar\phi)(z) \d z\\
&\sim_\ga -\int \int_0^1 \frac{(x-\th y-z)\cdot y }{|x-\th y-z|^{5+\ga}}   u(1-\bar\phi)(z) \d \th \,\d z ,
\end{split}\eqne
which gives  
\begin{align*}
I_{\ell 22}(x) \sim_\ga
\int_0^1\iint\frac{\chi_\rho(y)(\phi(x-y)-\phi(x))}{|y|^{3+2s}} \cdot \frac{(x-\th y-z)\cdot y}{|x-\th y-z|^{5+\ga}}  \ u(1-\bar\phi)(z) 
\d y \,\d z\, \d\th.
\end{align*}
Observe that for every integer $k\geq 0$
\begin{align*}
|\na^k \phi(x-y) - \na^k \phi(x)|
\leq c_k |y| \norm{\na^{k+1}\phi}_{L^\infty}
\end{align*}
and 
\begin{align*}
 \left| \na^k \left(\frac{x-\th y-z}{|x-\th y-z|^{5+\ga}}\right)\right| \lec \frac 1{|x-\th y-z|^{4+k+\ga}} 
 \lec \frac {1-\chi_\rho (x-\theta y -z)}{|x-\th y-z|^{4+\ga}},
\end{align*}
where we used the fact that $|x-\th y -z| \geq |z| -|x-\th y|\geq R_0 - (R+2\rho)\geq 3\rho$ in the last inequality. (Recall that $|z|\geq R_0$, as otherwise $(1-\bar\phi )(z)$ vanishes, and that $|x|,|x-y|\leq R+2\rho $ which gives the same bound on $|x-\th y|$.) This implies that
\begin{align*}
\norm{I_{\ell22}}_{L^2(-T,0;W^{k,\infty})}
&\lec
 \Norm{\int_{|y|\leq 2\rho} \frac 1{|y|^{1+2s}} 
\iint_0^1 \frac{|u(z)|(1-\chi_\rho (x-\theta y -z))}{|x-\th y-z|^{4+\ga}} \d \th\,\d z\,\d y }_{L^2(-T,0;L^\infty(B_{R+2\rho}))} \\
&\lec
\left\| \sup_{|x'|\leq R+4\rho} \int \frac {|u(z)|}{|x'-z|^{4+\ga}} (1-\chi_\rho (x'-z)) \d z\right\|_{L^2(-T,0)}\\
&\lec \norm{\cM(\La^s u)}_{L^2([-T,0]\times B_{R_0})} + \norm{u}_{L^2(-T,0; L^1 (B_{R_0}))} 
\end{align*}
for all integer $k\geq 0$, where we used \eqref{trick1} (with $\be:=1+\ga$) in the last inequality.\\

\noindent\texttt{Step 2a.} We prove \eqref{trick2a}. (Note that this also proves \eqref{trick2a_inlem} by applying \eqref{est.max}.)\\

For any $x\in B_{R+3\rho }$ and integer $k\geq 0$
\eqnb\label{temp1}
\begin{split}
 \La^\ga (u(1-\bar\phi)) (x)  
&\sim_\ga \mathrm{p.v.} \int \frac{ (u(1-\bar\phi ))(x) - (u(1-\bar\phi ))(y) }{|x-y|^{3+\ga }} \d y  
= \int \frac{ (u(1-\bar\phi ))(y) }{|x-y|^{3+\ga }} \d y\\
&= \int \frac{ (u(1-\bar\phi ))(y) }{|x-y|^{3+\ga }} (1-\chi_\rho(x-y))\d y,  \\
&= \int \frac{ u(y) }{|x-y|^{3+\ga }} (1-\chi_\rho(x-y))\d y
-\int \frac{ u\bar\phi (y) }{|x-y|^{3+\ga }} (1-\chi_\rho(x-y))\d y\\
&=: J_1 + J_2,
\end{split}
\eqne
where the second line follows from $|x-y|\geq |y|-|x| \geq R_0 -(R+3\rho) \geq 2\rho$, so that $\chi_\rho(x-y)=0$ for such $x,y$. As for $J_1$, we use \eqref{trick1} with $\be\coloneqq \ga>s$ to obtain
\begin{align}\label{est.J1}
\norm{J_1}_{W^{k,\infty}(B_{R+3\rho})}
\lec \norm{\mathcal{M}(\La^s u)}_{L^2(B_R)} + \norm{u}_{L^1(B_R)}.  
\end{align}
As for $J_2$ we have for any integer $k\geq 0$ and $x\in B_{R+3\rho}$
\begin{align*}
|\na^k J_2(x)|
\lec \int u\bar\phi(y) \na^k \left(\frac{1-\chi_\rho(x-y)}{|x-y|^{3+\ga}}\right) \d y
\lec_{k,\rho } \norm{u\bar\phi}_{L^1(\R^3)}  
\leq \norm{u}_{L^1(\supp_x(\bar\phi))},
\end{align*}
which gives $\norm{J_2}_{W^{k,\infty}(B_{R+3\rho})}
\lec \norm{u}_{L^1(\supp_x(\bar\phi))}$. Combining this with \eqref{est.J1} gives \eqref{trick2a}. 
\end{proof}

\begin{lem}\label{lem_comm.nau}
Let $\frac12<s<1$  and $R_0>R$. Let $\phi=\phi(x,t)$ be a smooth function compactly supported on $B_{R}\times (-T,0]$ for some $R>0$ and $T>0$. Then, given $2s-1<\ga<1$ and $\bar\phi=\bar\phi(x,t)\in C_c^\infty(\R^4)$ with $\bar\phi=1$ on $[-T,0]\times B_{R_0}$, for any $1 \le r \le \infty$ and $2\le p <\infty $
 we have a decomposition,
\begin{align*}
[(-\De)^s, \phi] \na u  = f_1  + g_1,
\end{align*}
where $f_1$ and $g_1$ satisfy
\begin{align*}
\norm{f_1}_{L^r(-T,0;L^p)}
&\lec \norm{\na (\bar\phi \La^\ga u )}_{L^r(-T,0; L^p)}
+\norm{\La^{1-\ga}(\bar\phi \La^\ga u )}_{L^r(-T,0; L^p)}
+\norm{\na u}_{L^r(-T,0;L^p(B_{R_0}))} \\
\norm{g_1}_{L^2(-T,0;W^{k,\infty})}
&\lec \norm{\cM (\La^s u)}_{L^2([-T,0]\times B_{R_0})},
\qquad\qquad \forall k\in \mathbb{N}.
\end{align*} 
\end{lem}
\begin{rem} The motivation of $\na (\bar\phi \La^\ga u )$ and $\La^{1-\ga}(\bar\phi \La^\ga u )$ comes from the bootstrapping argument. (See \texttt{Step 4} in the proof of Theorem~\ref{thm_bootstrapping}. Indeed, it is easy to see that one could control the commutator $[(-\De)^s, \phi] \na u $ using Lemma~\ref{lem_1.10} (with $u$ replaced by $\nabla u$). However, then we would not be able to use the resulting estimate in (\texttt{Step 4} of) the bootstrapping argument, as we would not have sufficient control over the term $\norm{\na u}_{L^r(-T,0;W^{2s-1+\bar\varepsilon,p}(B_{R_0}))}$. Indeed, using \eqref{norm_comparisons} this would require a bound on $\La^{2s-1+\bar\varepsilon } \na u$, which we have no control using \texttt{Step 2} of the bootstrapping argument (since the order of the operator is $2s+\bar\varepsilon > 2s-\bar\varepsilon$). Using Step 3 of the bootstrapping we have some control over derivatives of order $2s+\bar\varepsilon$, but it comes only via the derivatives of $\bar\phi \La^\ga u $. This the reason why terms $\na (\bar\phi \La^\ga u )$ and $\La^{1-\ga}(\bar\phi \La^\ga u )$ appears in the estimate on $f_1$ in the above lemma. We are able prove such estimate due to a nontrivial decomposition of $\nabla u$ (see \eqref{na_u_decomposition} below). The decomposition results in additional tail terms that need to be estimated (as part of $g_1$).
\end{rem}

\begin{proof}[Proof of Lemma~\ref{lem_comm.nau}.] For brevity we omit the variable $t$ of $u$ and $\phi$ unless it is needed.
We use the decomposition \eqref{decomposition_of_I} with $G\coloneqq \na u$. Similarly to \eqref{est.Iell1} and \eqref{est.It2}, we have
\begin{align*}
\norm{I_{\ell 1}}_{L^r(-T,0;L^p)} 
+\norm{I_{t2}}_{L^r(-T,0;L^p)} 
\lec \norm{\na u}_{L^r(-T,0;L^p(B_{R+2\rho}))} .
\end{align*}
Moreover, using \eqref{trick2} with $\be:= 2s$ and $\La^\ga u$ replaced by $\na u$ (see Remark~\ref{rem.trick}), we obtain
\[
\norm{I_{t1}}_{L^2(-T,0;W^{k,\infty})} \leq \norm{\cM|\La^s u|}_{L^2(-T,0;L^2(B_{R_0}))}
\]
As for $I_{\ell 2}$ we will use the decomposition
\eqnb\label{na_u_decomposition}
\begin{split}
\na u (x)  &= \na \La^{-\ga} (\bar\phi \La^{\ga} u + (\La^\ga u)(1-\bar\phi)) (x)\\
& =  \na \La^{-\ga}( \bar\phi \La^{\ga} u) (x) +C_{-\ga} \int \frac{x-z}{|x-z|^{5-\ga}} ((\La^{\ga} u)(1-\bar\phi))(z)  \d z \\
 &= (\na u)_{\ell} + (\na u)_t.
\end{split}\eqne
This gives further decomposition of $I_{\ell2}$;
\begin{align*}
I_{\ell2} (x) 
&= C_s \text{ p.v.} \int  \frac{\chi_{\rho}(x-y)(\na u(x)-\na u(y))(\phi(x)-\phi(y))}{|x-y|^{3+2s}} \d y \\
&= C_s \text{ p.v.} \int \frac{\chi_{\rho}(x-y)((\na u)_\ell(x)-(\na u)_\ell(y))(\phi(x)-\phi(y))}{|x-y|^{3+2s}} \d y \\
&\quad+ C_s \text{ p.v.} \int \frac{\chi_{\rho}(y)((\na u)_t(x-y)-(\na u)_t(x))(\phi(x-y)-\phi(x))}{|y|^{3+2s}} \d y \\
&=: I_{\ell 2\ell}(x) + I_{\ell 2t}(x).
\end{align*}
Then the first term can be estimated in the same way as \eqref{est.Iell2} with the choice of $\bar\varepsilon = \ga-(2s-1)$,
\begin{align*}
\norm{I_{\ell2\ell}}_{L^r(-T,0;L^p)} 
&\lec \norm{(\na u)_\ell}_{L^r(-T,0;\dot W^{\ga,p})}\\
&\lec \norm{\La^{\ga} \na \La^{-\ga} ((\La^{\ga} u)\bar\phi)}_{L^r(-T,0;L^p)} 
+\norm{\na \La^{-\ga} ((\La^{\ga} u)\bar\phi)}_{L^r(-T,0;L^p)} \\
&\lec  \norm{ \na ((\La^{\ga} u)\bar\phi)}_{L^r(-T,0;L^p)} 
+\norm{\La^{1-\ga} ((\La^{\ga} u)\bar\phi)}_{L^r(-T,0;L^p)}.
\end{align*}
As for $I_{\ell 2 t}$ we have
\begin{align*}
&I_{\ell 2t} 
=C_s\text{ p.v.}
\iint \frac{\chi_\rho(y)(\phi(x-y)-\phi(x))}{|y|^{3+2s}} \left(\frac {x-y-z}{|x-y-z|^{5-\ga}} - \frac {x-z}{|x-z|^{5-\ga}} \right) \La^\ga u(1-\bar\phi)(z) \d y\,\d z
\end{align*} 
(This is similar to \eqref{temp99}, except that now we have $\La^\ga u$ instead of $u$, which is an additional difficulty.)\\
Note that $|x-y-z| \geq |z| - |x|-|y| \geq R_0 - (R+2\rho ) - 2\rho \geq \rho$ and similarly $|x-z|\geq \rho$. We set $\overline{\eta } (y) \coloneqq \chi_\rho (y ) - \chi_{\rho } (2y )$ and \[
\eta_j (y)\coloneqq 2^{j(1-\ga )}\frac{y}{|y|^{5-\ga }} \overline{\eta }(2^{-j }y).\]
Then $\norm{\eta_j}_1 =C$ for some constant $C>0$ independent of $j$, $\supp \, \eta_j \subset B_{\rho 2^{j+1}}\setminus B_{\rho 2^{j-1}}$ and $|\na^k \eta_j |\lec_k 2^{j(-3-k)}$ for every $j\in \ZZ$, integer $k\geq 0$, and 
\eqnb\label{eta_j_for_large_y}
|\La^\al \na^k \eta_j (y) |\lec |y|^{-(3+k+\al)} \qquad \text{ for } |y| \geq 2^{j+2}\rho,\, \al\in (0,1),
\eqne
which can be verified as \eqref{eta_j_for_large_y_for_tricks}. Moreover for any $y$ with $|y|\geq \rho$
\[
\frac{y}{|y|^{5-\ga }} = \sum_{j\geq 0}2^{-j(1-\ga )} \eta_j (y)
\]
Thus, we can write $I_{\ell 2t} $ as
\begin{align*}
&I_{\ell 2t} 
\sim_s\sum_{j\geq 0} 2^{-j(1-\ga )}
\int \frac{\chi_\rho(y)(\phi(x-y)-\phi(x))}{|y|^{3+2s}}  \left[ \left(  \eta_j \ast \La^\ga u (1-\bar\phi ) \right) (x-y) -  \left(  \eta_j \ast \La^\ga u (1-\bar\phi ) \right) (x) \right] \d y \\
& =\sum_{j\geq 0} 2^{-j(1-\ga )}
\int \frac{\chi_\rho(y)(\phi(x-y)-\phi(x))}{|y|^{3+2s}}  \left[ \left( \La^{\ga-s} \eta_j \ast \La^s u  \right) (x-y) -  \left( \La^{\ga-s} \eta_j \ast \La^s u \right) (x) \right] \d y \\
& - \sum_{j\geq 0}2^{-j(1-\ga )} 
\int \frac{\chi_\rho(y)(\phi(x-y)-\phi(x))}{|y|^{3+2s}}  \left[ \left(  \eta_j \ast \La^\ga u \bar\phi  \right) (x-y) -  \left(  \eta_j \ast \La^\ga u \bar\phi  \right) (x) \right] \d y \\
& =: I_{\ell 2t1} + I_{\ell 2t2} 
\end{align*} 
As for $I_{\ell 2t2} $, for every $x\in B_{R+2\rho }$, $y\in B_{2\rho }$ we have
\begin{align*}
\left| \left(  \eta_j \ast \La^\ga u \bar\phi  \right) (x-y) -  \left(  \eta_j \ast \La^\ga u \bar\phi  \right) (x) \right| 
&\leq  \left| \eta_j \ast \na (\La^\ga u \bar\phi ) (x-\theta y)\right| \, | y | \leq \| \eta_j \|_\infty \| \na (\La^\ga u \bar\phi ) \|_{L^1(B_{R_0})} |y| \\
&\lec 2^{-3j }\| \na (\bar\phi\La^\ga u  ) \|_{L^1(B_{R_0})} |y|
\end{align*}
for some $\theta= \th_{x,y} \in [0,1]$. Thus,
\[
|I_{\ell 2t2}  (x) | \lec   \| \na \phi \|_\infty \| \na (\La^\ga u \bar\phi ) \|_{L^1(B_{R_0})} \int_{B_{2\rho }} |y|^{-(1+2s)} \d y \sum_{j\geq 0} 2^{-j(4-\ga )} \lec \| \na (\La^\ga u \bar\phi ) \|_{L^1(B_{R_0})}
\]
for every $x \in B_{R+2\rho }$, which gives that 
\[ \norm{I_{\ell2t2}}_{L^r(-T,0;L^p)} \lec  \norm{\na (\La^\ga u \bar\phi )}_{L^r(-T,0;L^p(B_{2 R_0}))} .\]
As for $I_{\ell 2t1 }$, we have for every integer $k\geq 0$, $x\in B_{R+2\rho }$
\[\begin{split}
|\na^k I_{\ell 2t1}& (x)| 
\lec \| \phi \|_{W^{k,\infty }} \sum_{m=0}^k
\sum_{j\geq 0} 2^{-j(1-\ga )} \int_{B_{2\rho} } |y|^{-(1+2s)} 
\norm{\na^{m+1}\La^{\ga-s}  \eta_j \ast \La^s u }_{ L^\infty (B_{R+4\rho })}
\d y \\
&\lec   \sum_{m=0}^k \sum_{j\geq 0} 2^{-j(1-\ga )} \left(  \left\| \int_{|z-y|\leq 2^{j}R_0} \La^{\ga-s} \na^{m+1} \eta_j (y-z )  \La^s u (z)  \d z \right\|_{L_y^\infty (B_{R+4\rho })} \right.\\
&\hspace{2cm}+ \left.\sum_{l\geq j} \left\| \int_{2^lR_0 < |z-y|\leq 2^{l+1}R_0} \La^{\ga-s} \na^{m+1} \eta_j (y-z )  \La^s u (z)  \d z \right\|_{L_y^\infty (B_{R+4\rho })}  \right) \\
&\lec_k  \sum_{m=0}^k \sum_{j\geq 0} 2^{-j(1-\ga )} \left( 2^{j(-4-m-\ga +s)} \left\| \int_{|z-y|\leq 2^{j+2}R_0} | \La^s u (z) | \d z \right\|_{L_y^\infty (B_{R+4\rho })} \right.\\
&\hspace{2cm}+ \left.\sum_{l\geq j+2} 2^{l(-4-m-\ga +s)} \left\| \int_{ |z-y|\leq 2^{l+1}R_0} |  \La^s u (z)|  \d z \right\|_{L_y^\infty (B_{R+4\rho })}  \right) \\
&\lec   \sum_{j\geq 0} 2^{-j(1-\ga )} \left( 2^{j(-1-\ga +s)}  \dint_{|z|\leq 2^{j+3}R_0} | \La^s u (z) | \d z +\sum_{l\geq j+2} 2^{l(-1-\ga +s)} \dint_{ |z|\leq 2^{l+2}R_0} |  \La^s u (z)|  \d z   \right) \\
&\lec \sum_{j\geq 0} 2^{-j(2-s )} \sup_{R'\geq 4R_0} \dint_{B_{R'}} | \La^s u  |  
\lec  \sup_{R'\geq 4R_0} \dint_{B_{R'}} | \La^s u  | ,
\end{split}
\]
where, in the third inequality, we used the bound $\|\La^{\ga-s} \na^n \eta_j \|_\infty \lec 2^{j(-3-n-\ga+s)}$ (in the first term) and \eqref{eta_j_for_large_y} (in the second term; recall \eqref{def_of_rho} that $R_0>R >2\rho$). Thus using \eqref{est.max}
\[
\norm{ I_{\ell 2t1} }_{L^2(-T,0;W^{k,\infty })}  \lec \norm{\cM (\La^s u)}_{L^2([-T,0]\times B_{R_0})}.
\]
Hence lemma follows by setting
\begin{align*}
f_1 = I_{\ell1} + I_{\ell 2\ell} + I_{\ell 2 t2} + I_{t2}, \quad
g_1 = I_{\ell2t1} + I_{t1}.
\end{align*}
\end{proof}
We now move on to the commutators involving the nonlinear term and the pressure term.
\begin{lem}\label{lem_comm.u} 
Let $\frac 34 <s<1$ and $R_0>R$. Let $\phi=\phi(x,t)$ be a smooth function compactly supported on $B_{R}\times (-T,0]$ for some $R>0$ and $T>0$. For any $\ga \in(2s-1,1)$, $r\in[1,\infty]$, and $p\in [1,\infty)$, we have decompositions
\begin{align*}
[\La^\ga, \phi] (u\cdot\na) u = f_2 + g_2\\
\end{align*}
such that for any $k\in \mathbb{N}\cup\{0\}$,
\begin{align*}
\norm{f_2}_{L^r(-T,0;L^p)}
&\lec \norm{(u\cdot\na) u}_{L^r(-T,0;L^p(B_R))} + \norm{u}_{L^\infty([-T,0]\times B_{R_0})}^2\\
\norm{g_2}_{L^1_t(-T,0; W^{k,\infty})}
&\lec_k \norm{\cM |\La^s u|^{\frac2{1+\delta }}}_{L^{1+\delta }([-T,0]\times B_{R_0})} 
+\norm{u}_{L^\infty([-T,0]\times B_{R_0})}^2
\end{align*}
where $\delta \coloneqq \frac{2s}{6-s}$.
\end{lem}

\begin{rem} The lemma is also valid with $\delta =0$, but such version would be of no use to us (in \texttt{Step 3} of the bootstrapping argument) since the global boundedness of $\norm{\cM |\La^s u|^2}_{L^1_{t,x}}$ is not guaranteed. In fact, making sure that $|\La^s u|$ appears in the estimate above with a power lower than $2$ under the maximal operator $\mathcal{M}$ as well as estimating the quadratic nonlinearity are the two main difficulties of this lemma.
\end{rem}
\begin{proof} We apply decomposition \eqref{decomposition_of_I} with $G\coloneqq (u\cdot \na ) u$ and $\be \coloneqq \ga$,
\begin{align*}
[\La^\ga, \phi] (u\cdot\na) u
= I_\ell + I_{t1} +I_{t2},
\end{align*}
where 
\[I_\ell = I_{\ell 1}+I_{\ell 2} \sim_\ga\,\text{ p.v.}\int \frac {\chi_\rho(x-y)(\phi(x)-\phi(y))(u\cdot \na) u(y)}{|x-y|^{3+\ga}} \d y .\]
We recall that (due to the support of $\phi$) $I_\ell$ is supported in $B_{R+2\rho}\times (-T,0]$, and that $R_0 = R+5\rho$ (see \eqref{def_of_rho}). Since $\ga <1$, we can easily estimate it by writing
\begin{equation}\begin{split}
\label{est.Iell}
&\norm{I_\ell}_{L^r(-T,0;L^p)}\\
&\quad\lec \Norm{\int \frac{|\chi_\rho(y)|}{|y|^{3+\ga}}
\norm{(\phi(x)-\phi(x-y))(u\cdot \na) u(x-y)}_{L^p(B_{R+2\rho})} \d y }_{L^r(-T,0)}\\
&\quad\lec\int
\frac {|\chi_\rho(y)|}{|y|^{2+\ga}} 
\d y\norm{(u\cdot \na) u}_{L^r(-T,0;L^p(B_{R+4\rho}))}
\lec \norm{(u\cdot \na) u}_{L^r(-T,0;L^p(B_{R_0}))}.
\end{split}\end{equation}
As for $I_{t2}$, 
\begin{align*}
\norm{I_{t2}}_{L^r(-T,0;L^p(\R^3))} 
\lec \norm{(u\cdot \na) u}_{L^r(-T,0;L^p(B_{R_0}))}.
\end{align*}
Lastly, consider the remaining piece 
\begin{align*}
I_{t1} (x)
&\sim_\ga   \phi(x) \int \frac{(1-\chi_\rho(x-y)) \na\cdot (u\otimes u)(y)}{|x-y|^{3+\ga} } \d y\\
&=  -(3+\ga)\,\phi(x) \int \frac{(1-\chi_\rho(x-y)) (x-y)_i }{|x-y|^{5+\ga}} (u-(u)_\ph)_i(u-(u)_\varphi)(y) \d y\\
&\quad 
-(3+\ga)\,\phi(x) \int \frac{(1-\chi_\rho(x-y)) (x-y)_i }{|x-y|^{5+\ga}} [u_iu(y)- (u-(u)_\varphi)_i(u-(u)_\varphi)(y)] \d y\\
&\quad- \phi(x) \int \frac{(u(y)\cdot\na) \chi_\rho(x-y)}{|x-y|^{3+\ga}}  u (y) \d y\\
&=: I_{t11} + I_{t12} + I_{t13}
\end{align*}
for $x\in \R^3$, where $\varphi \in C_c^\infty (B_R;[0,\infty ))$ is such that $\int \varphi \,\d x=1$. Then it is easy to see that
\begin{align*}
\norm{I_{t13}}_{L^r(-T,0;L^p)} 
\lec \left\|\int_{|x-y|\leq 2\rho} |u(y)|^2 \d y \right\|_{L^r(-T,0;L^p(B_R))} 
\lec \norm{u}_{L^{\infty}([-T,0]\times B_{R_0})}^2,
\end{align*}
where we used the fact that $\supp(\na\chi_\rho)\subset B_{2\rho}\setminus B_{\rho}$ in the first inequality. Also, since $I_{t12}$ satisfies 
\begin{align*}
|\na^k I_{t12}(x)|
\lec_k \sum_{m=0}^k 
|(u)_\ph| \int_{|y-x|\geq \rho}  \frac{|(u-(u)_\varphi)(y)|}{|x-y|^{4+\ga + m}}  \d y + |(u)_\ph|^2,
\end{align*}
for every $x\in B_{R }$, we obtain (as in \eqref{est.It11})
\begin{align*}
\norm{I_{t12}}_{L^2(-T,0;W^{k,\infty})}
&\lec \norm{u}_{L^\infty([-T,0]\times B_R)}^2 
+\norm{\cM|\La^s u|}_{L^2([-T,0]\times B_{R_0})}^2.
\end{align*} 
As for $I_{t11}$, we will show that for any $x\in B_R$,
\begin{align}\label{est.It11.suff2}
\int_{|y-x|\leq 2^j \rho} |u-(u)_\varphi|^2(y) \d y
\lec 2^{j(3+2s)} \left( \sup_{R'\geq 4R} \dint_{B_{R'}} |\La^s u|^{2/(1+\delta )} \right)^{1+\delta }.
\end{align}
(This is a quadratic version of \eqref{est.t1.suff}; recall that $\delta$ satisfies $1+\delta = \frac{6+s}{6-s}$.)\\
Given \eqref{est.It11.suff2}, similarly as in \eqref{est.It11}, we obtain
\begin{align*}
\norm{I_{t11}}_{L^1(-T,0;W^{k,\infty})}
&\lec\sum_{m=0}^k \sum_{j=1}^\infty 
\left\| \int_{2^{j-1}\rho\leq |y-x|\leq 2^j \rho}
\frac{|u-(u)_\varphi|^2(y,t)}{|x-y|^{4+\ga+m}} \d y \right\|_{L^1_t(-T,0;L^{\infty}_x(B_R))}\\
&\lec \sum_{j=1}^\infty 2^{-j(4+\ga)} 
\left\| \int_{|y-x|\leq 2^j \rho}
|u-(u)_\varphi|^2(y,t) \d y\right\|_{L^1_t(-T,0;L^{\infty}_x(B_R))}\\
&\lec \int_{-T}^0 \left( \sup_{R'\geq 4R} \dint_{B_{R'}} |\La^s u|^{\frac 2{1+\delta }}  \d y \right)^{1+\delta } \\
&\lec \norm{\cM |\La^s u|^{\frac 2{1+\delta }}}_{L^{1+\delta }((-T,0)\times B_{R})},
\end{align*}
where we used \eqref{est.It11.suff2} and the fact that $\ga >2s-1$ in the third inequality, and \eqref{est.max} (without its last step) in the last inequality. Thus (given \eqref{est.It11.suff2}) the lemma follows by letting 
\begin{align*}
f_2 \coloneqq  I_{\ell} + I_{t13}
,\quad
g_2 \coloneqq  I_{t11}+ I_{t12}. 
\end{align*}

In order to obtain \eqref{est.It11.suff2}, we write
\[
u(y)-(u)_\ph = u_{\loc,1}(y) +u_{\loc,2} +u_{\tail }(y) ,
\]
as in \eqref{repr_u_for_tail}. Using using \eqref{est.uloc2} and \eqref{est.utail}, we easily obtain 
\begin{align*}
\int_{|y-x|\leq 2^j \rho} \left( |u_{\loc,2}|^2+|u_{\tail}(y)|^2 \right)\d y
&\lec 2^{j(3+2s)} \left( \sup_{R'\geq 4R} \dint_{B_{R'}} |\La^s u| \right)^2\\
&\lec 2^{j(3+2s)} \left( \sup_{R'\geq 4R} \dint_{B_{R'}} |\La^s u|^{2/(1+\delta )} \right)^{1+\delta } .
\end{align*} 
As for $u_{\loc,1}$, we obtain the following estimate by a modification of \eqref{est.uloc1}, 
\begin{align*}
\int_{|y-x|\le 2^j\rho} |u_{\loc,1}(y)|^2\d y
&\lec 
\int_{|y-x|\le 2^j\rho}
\left[\int_{|z-x|\leq 2^{j+n_0}\rho}
\frac{|\La^s u(z)|}{|y-z|^{3-s}}\d z\right]^2\d y\\
&\le 
\int
\left[\int
\frac{|\La^s u(z)|1_{B_{2^{j+n_0}\rho}(x)}(z)1_{B_{2^{j+n_0+1}\rho}}(y-z)}{|y-z|^{3-s}}\d z\right]^2\d y\\
&\lec \norm{\La^s u}_{L^{\frac 2{1+\delta}}(B_{2^{j+n_0}\rho}(x))}^2
\norm{|\cdot|^{-(3-s)}}_{L^{\frac{6-s}{6-2s}}(B_{2^{j+n_0+1}\rho})}^2\\
&\lec 2^{js\frac{6-2s}{6-s}}\left( \int_{B_{2^{j+n_0}\rho}(x)} |\La^s u|^{\frac2{1+\delta }} \right)^{1+\delta } \\
&\lec 2^{j(3+2s)} \left( \sup_{R'\geq 4R} \dint_{B_{R'}} |\La^s u|^{\frac2{1+\delta }}  \right)^{1+\delta },
\end{align*}
where we used Young's inequality for convolutions in the third inequality. 
\end{proof}
Finally, we turn to the commutator concerning the pressure function.

\begin{lem}\label{lem_comm.p} 
Let $p$ satisfy $-\De p = \pa_{ij}(u_iu_j)$ on $(-T,0)\times \RR^3$, and let $\frac 34 <s<1$ and $R_0>R$. Let $\phi=\phi(x,t)$ be a smooth function compactly supported on $B_{R}\times (-T,0]$ for some $R>0$ and $T>0$. For any $\ga \in(2s-1,1)$, $r\in[1,\infty]$, and $q\in [1,\infty)$, we have decompositions
\begin{align*}
[\La^\ga, \phi]\na p = f_3 +g_3
\end{align*}
such that for any $k\in \mathbb{N}\cup\{0\}$,
\begin{align*}
\norm{f_3}_{L^r(-T,0;L^q)}
 &\lec \norm{(u\cdot\na)u }_{L^r(-T,0;L^q(B_{R_0}))}\\
\norm{g_3}_{L^1_t(-T,0; W^{k,\infty})} 
&\lec_k \norm{u}^2_{L^2(-T,0;L^\infty (B_{R_0})} +\norm{p}_{L^1([-T,0]\times B_{R_0})}+
\norm{\mathscr{M}_4 (\La^{2s-1}\na p)}_{L^1([-T,0]\times B_{R})}.
\end{align*}
\end{lem}

\begin{proof}
We apply decomposition \eqref{decomposition_of_I} with $G\coloneqq \na p$ and $\be=\ga$, 
\begin{align*}
[\La^\ga, \phi]\na p = I_\ell + I_{t1}  + I_{t2},
\end{align*}
where
\[I_\ell  \sim_\ga \text{p.v.}\int \frac {\chi_\rho(x-y)(\phi(x)-\phi(y))\na p(y)}{|x-y|^{3+\ga}} \d y .\]
As in the proof of Lemma~\ref{lem_1.10}, we assume that $R_0< \frac 72 R$ (so that $2\rho <R$). \\

\noindent\texttt{Step 1.} We estimate $I_\ell $ and $I_{t2}$.\\ 

To deal with the local part, $I_\ell$, we use the pressure decomposition \eqref{eq.loc.pre},
\[\begin{split}
\na( p\td\phi)  &=   \na \mathcal{R}_{ij} (u_iu_j \td\phi) 
 - \na \left[ 2(-\De)^{-1}\pa_i(u_i u_j \pa_j\td\phi) +(-\De)^{-1}(u_i u_j \pa_{ij}\td\phi) \right] +\na p_2 \\
 &= \mathcal{R}  \mathcal{R}_{j} (u \cdot \na  u_j \td\phi) 
 - \na \left[ (-\De)^{-1}\pa_i(u_i u_j \pa_j\td\phi) +(-\De)^{-1}(u_i u_j \pa_{ij}\td\phi) \right] +\na p_2 \\
 &=:\mathcal{R}  \mathcal{R}_{j} (u \cdot \na  u_j \td\phi) 
 + \na \td p
 \end{split}
\]
where $p_2 \coloneqq - \left[2 (-\De)^{-1}\na \cdot (p \na \bar\phi ) -(-\De)^{-1}( p\De\bar\phi)\right]$ and $\td\phi=\td\phi(x,t)\in C_c^\infty( B_{R_0}\times \RR )$ satisfies $\td\phi\equiv 1$ on $B_{R+\frac92\rho}\times [-T,0]$. (We note that $\td p$ is the same as $p_1+p_2$ in \eqref{eq.loc.pre}, except for the factor of $2$ in $p_1$). As in \eqref{pressure1}, \eqref{pressure2} we obtain 
\[
\| \na^k \td p \|_{L^1 (-T,0;L^\infty (B_{R+4\rho }))} \lec_k \| u \|_{L^2 (-T,0;L^\infty (B_{R_0}))}^2 + \| p \|_{L^1 (-T,0; L^1 (B_{R_0}))} 
\]
for every integer $k\geq 0$. Thus, since $I_\ell(x)$ is supported on $B_{R+2\rho}$ (so that $x,x-y\in B_{R+4\rho}$, which implies that $\td \phi (x-y)=1$), we can write 
\begin{align*}
I_\ell(x)
&\sim_\ga \text{ p.v.}\int  \frac{\chi_\rho(y)(\phi(x)-\phi(x-y))\na (p\td\phi)(x-y)}{|y|^{3+\ga}} \d y\\
&=\text{ p.v.}\int \frac{\chi_\rho(y)(\phi(x)-\phi(x-y))\mathcal{R}  \mathcal{R}_{j} (u \cdot \na  u_j \td\phi)  (x-y)}{|y|^{3+\ga}} \d y\\
&\quad+ \text{ p.v.}\int \frac{\chi_\rho(y)(\phi(x)-\phi(x-y))
  \na \td p (x-y)}{|y|^{3+\ga}} \d y\\
&=: \td  I_{\ell 1} + \td I_{\ell 2}, 
\end{align*}
which gives that 
\[\norm{\td I_{\ell 1}}_{L^r(-T,0;L^q)}
\lec  \norm{\mathcal{R}  \mathcal{R}_{j} (u \cdot \na  u_j \td\phi) }_{L^r(-T,0;L^q)}\lec_q  \norm{(u\cdot \na) u}_{L^r(-T,0;L^q(B_{R_0}))}
\]
for every $q<\infty$, $r\in [1,\infty ]$, and 
\[ \norm{\td I_{\ell2}}_{L^1(-T,0;W^{k,\infty})} \lec_k  \| u \|_{L^2 (-T,0;L^\infty (B_{R_0}))}^2 + \| p \|_{L^1 ([-T,0]\times B_{R_0})}  \]
for every integer $k\geq 0$.

As for the tail part $I_{t2}$ we have
\begin{align*}
\norm{I_{t2}}_{L^1(-T,0;W^{k,\infty})}
&\lec_k  
\sum_{m=1}^k \left\|\int_{|x-y|\geq \rho} \frac{ |\na p(y)||\phi(y)|}{|x-y|^{3+\ga+m}} \d y\right\|_{L^1(-T,0;L_x^{\infty})}
\lec_k \norm{\na p}_{L^1([-T,0]\times B_R)}
\end{align*} 
for any integer $k\geq 0$, as required. \\

\noindent\texttt{Step 2.} We estimate $I_{t1}$. \\

We first decompose $I_{t1}$ as
\eqnb\label{def_of_It11}\begin{split}
I_{t1} (x) &\sim_\ga \, \phi(x)\int  \frac{\na p(y)}{|x-y|^{3+\ga }}  (1-\chi_\rho)(x-y)\d y \\
&=\, \phi(x)\left( \int  \frac{(\na p(y)-(\na p )_\varphi )}{|x-y|^{3+\ga }}  (1-\chi_\rho)(x-y)\d y + \int  \frac{(\na p)_\varphi }{|x-y|^{3+\ga }}  (1-\chi_\rho)(x-y)\d y \right) \\
&= : I_{t11} (x)+I_{t12} (x),
\end{split}
\eqne
where $\varphi \in C_c^\infty (B_R)$ is such that $\int \varphi \d x =1$. Then we obtain
\begin{align*}
\norm{I_{t12}}_{L^1(-T,0;W^{k,\infty})}
&\sim\, \Norm{\phi(x) \int\frac{(1-\chi_\rho(x-y))}{|x-y|^{3+\ga}} \d y (\na p)_\varphi}_{L^1(-T,0;W^{k,\infty}(B_R))}\\
&\lec \norm{(\na p )_\varphi }_{L^1(-T,0)}\\
&\lec \norm{ p}_{L^1([-T,0]\times B_R)}.
\end{align*}

As for $I_{t11}$ we will show below that  
\eqnb\label{It11_optimal}
\norm{ I_{t11} }_{L^1(-T,0;W^{k,1})} \lec_k \| \mathscr{M}_4 (\La^{2s-1}\na p ) \|_{L^1 ((-T,0]\times B_{R_0})} .
\eqne
In fact this is the most challenging estimate in this section. Actually one could instead use a similar approach as in \eqref{trick1} to prove the same estimate, but with the grand maximal function $\mathscr{M}_4$ replaced by the Hardy-Littlewood maximal function. However, as mentioned in the introduction, such estimate would be of no use to us (as $\cM f \not \in L^1$ for any $f\in L^1$, $f\ne 0$). This is the point where the use of the grand maximal function becomes necessary and, in the remaining part of this section, we show that \eqref{It11_optimal} can be obtained by combining the two ideas that we have already used (in showing \eqref{trick1} and in the proof of Lemma~\ref{lem_poincare_for_p}) and adapting them to fit the structure of $\mathscr{M}_4$. 

In fact, in order to see \eqref{It11_optimal} we will show that   
\begin{align}\label{It11_optimal_reduced}
\norm{ \na^m I_{t11} }_{L^1(-T,0;L^1)} \lec_m
\norm{\mathcal{L}_m(\La^{2s-1}\na p)}_{L^1((-T,0]\times B_{R_0})}, 
\end{align}
where the operator $\mathcal{L}_m$ is defined by
\eqnb\label{def_of_L}\begin{split}
\mathcal{L}_m(H) &\coloneqq     
\sup_{j\geq 1}\left| (\eta_{j,m} \ast \zeta_j )\ast H  \right|   + \sup_{j\geq 1}\left| {\zeta }_{j}\ast H  \right|  
+\sup_{j\geq 0} |\td \zeta_{j,m+1} \ast H|\\
&\quad +\sup_{\substack{1-n_0\leq l\leq j\\ a\in [1/2,1]}} |\eta_l^{(a)} \ast \td \zeta_{m+1,j}\ast H| + \sup_{\substack{1-n_0\leq l\leq j\\ a\in [1/2,1]}}  |\td\eta_l^{(a)} \ast \td \zeta_{m+1,j}\ast H|.
\end{split}\eqne
(Recall (above \eqref{est.uloc1}) that the integer $n_0\geq 4 $  is determined by $R$ and $R_0$.) Then \eqref{It11_optimal} follows from the inequality $\mathcal{L}_m\lec_m \mathscr{M}_4$, which we show in Lemma~\ref{lem_bound_G_by_grand_max} below. The auxiliary functions appearing under the suprema above are defined as
\eqnb\label{def_of_fcns}
\begin{split}
\eta_j (y) &\coloneqq 2^{j\ga } \frac{1}{|y|^{3+\ga }} \overline{\eta }(2^{-j}y), \quad
\eta_j^{(\th)} (y) \coloneqq 
\frac 1{\th^3} \eta_j\left(\frac y {\th}\right),\quad
\td\eta_j^{(\th)} (y) \coloneqq 
\frac {2^{-j}y}{\th^4} \eta_j\left(\frac y {\th}\right)\\
\zeta_j (y) &\coloneqq 2^{j(1-2s)} \frac{\chi_\rho (2^{-j} y)}{|y|^{4-2s}} ,\quad
\td\zeta_j (y) \coloneqq 2^{j(1-2s)} \frac{\bar\eta (2^{-j} y)}{|y|^{4-2s}}, \\
\eta_{m,j} &\coloneqq 2^{jm} \na^m \eta_j, \hspace{1.7cm} 
\td\zeta_{m,j} \coloneqq  2^{jm} \na^m \td\zeta_j,
\end{split}
\eqne
where $\th\in [1/2,1]$ and $\overline{\eta } (y) \coloneqq \chi_\rho (y) - \chi_\rho (2y)$ for $\chi_\rho \in C_c^\infty (B_{2\rho };[0,1])$ satisfying $\chi_\rho =1$ on $B_\rho$. Every function defined in \eqref{def_of_fcns} satisfies 
\begin{align*}
g_j (y) = 2^{-3j} g(2^{-j}y)
\end{align*}
where $g=\eta,\eta^{(\th )}, \td \eta^{(\th )}, \zeta, \td \zeta, \eta_m , \td\zeta_m$. We also note that 
\begin{align*}
\eta_{m,j}\ast \zeta_j (y)
= 2^{-3j} (\eta_{m,0} \ast\zeta_0)(2^{-j} y).
\end{align*}

In order to prove \eqref{It11_optimal_reduced} note that 
\[
\frac{1-\chi_{ \rho } (2^{-n}( y))}{|y|^{3+\ga }} = \sum_{j\geq n+1} 2^{-j\ga  } \eta_j(y)
\]
for every $n\in \ZZ$, and so $I_{t11}$ can be written as
\[
\begin{split}
I_{t11} (x) &= C_\ga \,\phi(x)   \int \frac{1-\chi_\rho (x-y)}{|x-y|^{3+\ga }} \left( \na p(y)-(\na p)_\varphi  \right) \d y\\
&=C_\ga\phi(x)\sum_{j\geq 1}  2^{-j \ga } \int \eta_j (x-y) \left( \na p(y)-(\na p)_\varphi  \right) \d y.
\end{split}
\]
Using 
\[
\frac{1}{|x|^{4-2s}}
= \frac{\chi_\rho (2^{-j}x)}{|x|^{4-2s}} + \frac{1-\chi_\rho (2^{-j}x)}{|x|^{4-2s}} 
= 2^{-j(1-2s)}\zeta_j(x) + \sum_{k>j}2^{-k(1-2s)}\td \zeta_k(x),
\]
we decompose $\na p(y)- (\na p)_\varphi$,
\[
\begin{split}
\na p(y)- (\na p)_\varphi &=  \int (\na p(y)-\na p(w) ) \varphi (w) \d w \\
&= \iint \left( \frac{1}{|y-z|^{4-2s}} - \frac{1}{|w-z|^{4-2s}}   \right) H(z) \varphi(w) \d z \, \d w  \\
&=2^{-j(1-2s)} \int \left(\zeta_j\ast H(y) - \zeta_j\ast H(w) \right)\ph(w) \d w\\
&\hspace{1cm}+\sum_{k>j} 2^{-k(1-2s)} \int \left(\td\zeta_k\ast H(y) - \td\zeta_k\ast H(w) \right)\ph(w) \d w\\
&=: q_{j,1} (y) + q_{j,2}(y),
\end{split}
\]
where we write $H \coloneqq  \La^{2s-1}\na p$ for brevity. Plugging this  back into the integral, we obtain the corresponding decomposition of $I_{t11}$, 
\begin{align*}
I_{t11}(x)  
&=C_\ga \, \phi(x)
\sum_{j\geq 1}  2^{-j \ga } \left( \eta_j \ast q_{j,1} (x) +\eta_j \ast q_{j,2} (x) \right) \\
&=:C_\ga \, \phi(x)\left( J_1(x)  +J_2(x)\right).
\end{align*}
Since $\phi$ is smooth and supported in $B_R$, the claim \eqref{It11_optimal_reduced} follows if we show that
\begin{align}\label{It11_optimal_reduced2}
\| \na^m J_1\|_{L^1 (B_R)} + \|\na^m J_2 \|_{L^1 (B_R)} \lec_m \| \mathcal{L}_m(H) \|_{L^1 (B_{R_0})}
\end{align}
for every $m\geq 0$ (and almost every $t\in (-T,0)$). 

The estimate of $J_1$ can be obtained easily by noting that 
\[
\begin{split}
J_1 (x) &= \sum_{j\geq 1} 2^{-j(\ga+1-2s ) }
\left(\eta_j \ast \zeta_j \ast H(x) - c\int (\zeta_j\ast H) \ph \,\d w\right)
\end{split}
\]
since $\int \eta_j \, \d x  =c$ and $\int \ph \,\d x =1$. Thus for any integer $m\geq 0$, using $\sum_{j\geq 1}  2^{-j(\ga+1-2s) }2^{-jm} \lec 1$ 
\[
\| \na^m J_1 \|_{L^1 (B_R)} 
\lec \left\| \sup_{j\geq 1}\left| (\eta_{m,j} \ast \zeta_j )\ast H  \right|  \right\|_{L^1 (B_R)} + \left\| \sup_{j\geq 1}\left| {\zeta }_{j}\ast H  \right|  \right\|_{L^1 (B_R)} 
\lec \| \mathcal{L}_m(H)\|_{L^1 (B_R)}.
\]

\

As for $J_2$, we first write the $m$th derivatives of the integral in the definition of $q_{j,2}$ for integer $m\geq 0$ as
\begin{align*}
\pa_y^m \int
&\left(\td\zeta_k \ast H(y) - \td\zeta_k \ast H(w) \right)\ph(w) \d w
= \pa_y^m\iint_0^1 \na \td\zeta_k \ast H(\th y+(1-\th)w) \cdot (y-w) \d \th\,\ph(w) \d w\\
&= 2^{-k(m+1)}\th^m\iint_0^1 2^{k(m+1)}\na^{m+1} \td\zeta_k \ast H(\th y+(1-\th)w) \cdot (y-w) \d \th\, \ph(w) \d w  \\
&=2^{-k(m+1)}\th^m\left(\iint_0^{2^{-j-n_0}} + \sum_{l=1-n_0}^{j} \iint_{2^{-j+l-1}}^{2^{-j+l}}\right) \td\zeta_{m+1,k} \ast H(\th y+(1-\th)w) \cdot (y-w) \d \th \,\ph(w) \d w.
\end{align*}
We now rewrite $\na^m J_2$ by applying the integration by parts and have a corresponding decomposition, 
\begin{align*}
\na^m J_2(x)
&=\sum_{j\geq 1} \eta_j \ast \p^m q_{j,2} (x)  \\
&=\sum_{j\geq 1} \sum_{k>j} 2^{-j\ga}2^{k(2s-1)} 
\iint \eta_j(x-y)
\pa_y^m \int
\left(\td\zeta_k \ast H(y) - \td\zeta_k \ast H(w) \right)\ph(w) \d w\d y \\
&=: \na^m J_{21}(x) + \na^m J_{22}(x),
\end{align*}
where the last decomposition is obtained by the decomposition of the $\th$-integral pointed out above.

In order to estimate $\na^m J_{21}(x)$ for $x\in B_R$ we note that we have $|x-y|\leq 2^{j+1}\rho$, $|w|\leq R$, $\th\in [0,2^{-j-n_0}]$, so that $|y| \leq |x-y|+|x|\leq 2^{j+1}\rho + \frac 14\cdot 2^{n_0} \rho \leq 2^{j+n_0}\rho$ (recall above \eqref{est.uloc1} that $4R\leq 2^{n_0}\rho$), $|\th y + (1-\th )w | \leq \th|y|+ |w|\leq \rho + R < R_0$ and $|y-w|\lec 2^j$, which gives that 
\begin{align*}
&| \na^m J_{21}(x) | \\
& \lec \sum_{j\geq 1} \sum_{k> j} 2^{-j\ga } 2^{k(2s-2)} \iiint_0^{2^{-j-n_0}} |\eta_j (x-y)| |\td \zeta_{m+1,k} \ast H(\theta y + (1-\theta )w ) |\,|y-w|\,|\varphi (w) | \, \d \theta\, \d w\, \d y \\
& \lec \sum_{j\geq 1} \sum_{k> j} 2^{j(1-\ga)} 2^{k(2s-2)}\iint_0^{2^{-j-n_0}} |\eta_j (x-y)|
\int_{B_R} |\td \zeta_{m+1,k} \ast H(\theta y + (1-\theta )w ) |\d w 
 \, \d \theta\,   \, \d y\\
& \lec \sum_{j\geq 1} \sum_{k> j} 2^{j(1-\ga)} 2^{k(2s-2)} \int \int_0^{2^{-j-n_0}} |\eta_j (x-y)|
\norm{\td \zeta_{m+1,k} \ast H}_{L^1(B_{R_0})} \frac{\d \theta}{(1-\th)^3}   \, \d y\\
& \lec \sum_{j\geq 1} \sum_{k> j} 2^{-j\ga} 2^{k(2s-2)} \norm{\td \zeta_{m+1,k} \ast H}_{L^1(B_{R_0})}  \int  |\eta_j (x-y)|\d y\\
&\lec \norm{\sup_{k\geq 0} |\td \zeta_{m+1,k} \ast H|}_{L^1(B_{R_0})}.
\end{align*}

Therefore, it easily follows that 
\[
\norm{\na ^m J_{21}}_{L^1(B_R)} 
\lec \norm{\sup_{k\geq 0} |\td \zeta_{m+1,k} \ast H|}_{L^1(B_{R_0})} \lec \norm{\mathcal{L}_m(H)}_{L^1(B_{R_0})}. 
\]

To deal with $\na^m J_{22}$, we rewrite the following integral, when $\th \in (2^{-j+l-1},2^{-j+l}]$, as
\begin{align}
& \iint \eta_j(x-y)
 \td \zeta_{m+1,k}\ast H(\th y + (1-\th) w) \cdot (y-w) \ph(w) \, \d w\, \d y  \nonumber\\
&\quad= \iint \frac 1{\th^3}\eta_j\left(\frac{\th x + (1-\th) w - y'}{\th}\right)  
 \td \zeta_{m+1,k}\ast H(y') \cdot (x-w) \ph(w) \, \d w\, \d y' \nonumber\\
&\qquad -
 \iint \frac 1{\th^3} \eta_j\left(\frac{\th x + (1-\th) w - y'}{\th}\right)  
 \td \zeta_{m+1,k}\ast H(y') \cdot \left(\frac{\th x +(1-\th) w - y'}{\th}\right) \ph(w) \, \d w\, \d y' \nonumber\\
\begin{split} 
&\quad= \int
\eta_l^{(\th')} \ast \td \zeta_{m+1,k}\ast H(\th x + (1-\th) w ) \cdot (x-w) \ph(w) \, \d w \\
&\qquad -
 2^j\int \td\eta_{l,i}^{(\th')}\ast \td \zeta_{m+1,k}\ast H_i(\th x + (1-\th) w)  \ph(w) \, \d w \label{decom.J22}
 \end{split}
\end{align}
where $H_i$ is $i$th component of $H$, and the summation convention in $i$ is used. Here the first equality follows from the change of variable $y\mapsto\th y + (1-\th)w =: y' $ and the decomposition $y-w  = \frac 1{\th} (y'-w) = (x-w) - \frac{\th x +(1-\th) w - y'}{\th}$, and the second equality follows by setting $\th' = 2^{j-l}\th \in (1/2,1]$ and noting that
\begin{align*}
\frac 1{\th^3} \eta_j \left( \frac z{\th}\right)
= \frac 1{2^{3l} (2^{j-l}\th)^3} \eta_0 \left( \frac z{2^{l} (2^{j-l}\th)}\right) = \eta_l^{(\th')}(z), \quad
\frac {2^{-j}z_i}{\th^4} \eta_j \left( \frac z{\th}\right)
=\td \eta_{l,i}^{(\th')}(z).
\end{align*}
Using \eqref{decom.J22}, we can decompose $\na^m J_{22}$ into two parts and estimate on $B_R$ as
\begin{align*}
&\int_{B_R}|\na^m J_{22}(x)| \d x\\
&\lec  \sum_{j\geq 1} \sum_{k>j} 2^{-j\ga} 2^{k(2s-2)} \sum_{l=1-n_0}^j\int_{2^{-j+l -1}}^{2^{-j+l}}
\int_{B_R}\int_{B_R}
|\eta_l^{(\th')} \ast \td \zeta_{m+1,k}\ast H(\th x + (1-\th) w )| |x-w|  
 \,\d x \, \d w\,\d \th\\
&\quad+ \sum_{j\geq 1} \sum_{k>j} 2^{j(1-\ga)} 2^{k(2s-2)} \sum_{l=1-n_0}^j\int_{2^{-j+l -1}}^{2^{-j+l}}
\int_{B_R} \int_{B_R} |\td\eta_{l,i}^{(\th')}\ast \td \zeta_{m+1,k}\ast H_i(\th x + (1-\th) w)| \,  \d x \, \d w\,
  \d \th \\
  &\lec \sum_{j\geq 1} \sum_{k>j} 2^{-j\ga} 2^{k(2s-2)} 
\sum_{l=1-n_0}^j\int_{2^{-j+l -1}}^{2^{-j+l}}\frac{\d \th}{\th^3}\\
&\qquad \cdot\left(\Norm{\sup_{\substack{1-n_0\leq l\leq k\\ k\geq 1}} \left| \eta_l^{(\th')} \ast \td \zeta_{m+1,k}\ast H\right| }_{L^1(B_R)} + 2^j\Norm{\sup_{\substack{1-n_0\leq l\leq k\\ k\geq 1}}  \left| \td\eta_l^{(\th')} \ast \td \zeta_{m+1,k}\ast H\right|}_{L^1(B_R)}
\right)\\
&  \lec\, \norm{\mathcal{L}_m(H)}_{L^1(B_R)},
\end{align*}
where we used the trivial bound $2^{-km}\leq 1$ in the first inequality, the bound $|x-w|\leq 2R$ in the second inequality, and the last line follows by noting that $\int_{2^{-j+l -1}}^{2^{-j+l}}\frac{\d \th}{\th^3} \sim 2^{2(j-l)}$, which gives convergence  of the triple sum, 
\[
\sum_{j\geq 1} \sum_{k>j} 2^{j(1-\ga)} 2^{k(2s-2)} 
\sum_{l=1-n_0}^j 2^{2(j-l)} \lec \sum_{j\geq 1} \sum_{k>j} 2^{j(1-\ga)} 2^{k(2s-2)} \lec \sum_{j\geq 1}  2^{-j(\ga+1-2s)} \lec 1.
\]
 This together with the same estimates for $J_1$ and $J_{21}$ above give \eqref{It11_optimal_reduced2}, as required.   
\end{proof}
We now conclude this section by showing the relation between $\mathcal{L}_m(H)$ and $\mathscr{M}_4(H)$. 
%
\begin{lemma}\label{lem_bound_G_by_grand_max} If $\mathcal{L}_m(H)$ is defined by \eqref{def_of_L} then for any $H$
\[
\mathcal{L}_m(H)(x)\lec_m \mathscr{M}_4(H)(x) \qquad \text{ for all } x\in \R^3. 
\]
\end{lemma}
\begin{proof}
We recall that $\mathcal{L}_m(H)$ consists of the terms which can be represented as 
\eqnb\label{terms_in_G}
\sup_{j\geq 1} |\eta_{j} \ast H|, \quad
\quad \sup_{-b\le l \leq k, a\in [1/2,1]} | \eta_l^{(a)} \ast \zeta_{k} \ast H|
\eqne
where $b$ is some fixed positive constant, $\eta_j(y) = 2^{-3j}\eta_0(2^{-j}y)$ and $\zeta_j(y) = 2^{-3j}\zeta_0(2^{-j}y)$ for some $\eta_0, \zeta_0\in C_c^\infty(B_{2\rho})$, and $\eta_j^{(a)}(y)=a^{-3}\eta_j (a^{-1}y)$. 

By the definition of $\mathscr{M}_4$, it easily follows that 
\begin{align*}
\sup_{j\geq 1} |\eta_{2^j} \ast H| \lec_{\varrho} \mathscr{M}_4 (H)
\end{align*}
for every $\varrho$. Since $\mathcal{L}_m$ contains only finitely many candidates for $\eta$, the claim follows for such terms. 

Therefore, we are left to deal with the second representative term in \eqref{terms_in_G}, 

Since $l\leq k$, we set $l \coloneqq k + n$, $n\leq 0$.
Using $\eta_l^{(a)}(y) = 2^{-3k} \eta_{n}^{(a)}(2^{-k}y)$, we have
\begin{align*}
 \eta_l^{(a)} \ast \zeta_{k} (y)
 = 2^{-3k}(\eta_{n}^{(a)}\ast \zeta_0) (2^{-k}y)
=: 2^{-3k}\Psi^{n,a} (2^{-k}y)
=: \Psi_{k}^{n,a} (y).
\end{align*} 

Then we obtain
\begin{align*}
\sup_{-b \le l\le k,a\in [1/2,1]} |  \eta_l^{(a)} \ast \zeta_{k}  \ast H(x)|
&\le \sup_{n\leq 0, a\in [1/2,1]}\sup_{k\geq -b} 
|\Psi_{k}^{n,a}  \ast H(x)| 
\le \sup_{n\leq 0, a\in [1/2,1]}
|\mathcal{M}(H;\Psi^{n,a} )(x)|\\
&\le \sup_{n\leq 0, a\in [1/2,1]}|\mathcal{M}_1^*(H;\Psi^{n,a} )(x)|. 
\end{align*}

Since for each $n\leq 0$ and $a\in [1/2,1]$, $\Psi^{n,a}$ satisfies 
\begin{align*}
\int (1+|x|)^4 &\sum_{|\al|\leq 5} |\pa^{\al} \Psi^{n,a}(x)| \,\d x
= \int (1+|x|)^4 \sum_{|\al|\leq 5} |\eta_{n}^{(a)}\ast \pa^\al\zeta_0| \,\d x\\
&\le \sum_{|\al|\leq 5}
\iint (1+|x|)^4 |\eta_{n}^{(a)}(x-y)| |\pa^\al  \zeta_0(y)| \,\d y\, \d x\\
&= \sum_{|\al|\leq 5}
\iint (1+|x|)^4 \frac{1}{(2^{n}a)^{3}} \left|\eta_0\left(\frac{x-y}{2^na}\right)\right| |\pa^\al  \zeta_0(y)| \,\d y\, \d x\\
&\leq 
(1+4\rho)^4 \norm{\eta_0}_1 \sum_{|\al|\leq 5}\norm{\pa^\al \zeta_0}_1,
\end{align*}
where the last line follows from $\supp(\eta_0), \supp ( \zeta_0)\subset B_{2\rho}$, so that 
\begin{align*}
|x-y|\leq 2^n a \cdot 2\rho \leq 2\rho, \quad |y|\leq 2\rho \implies
|x|\leq 4\rho. 
\end{align*}
Observe that the upper bound is independent of $n$ and $a$. Therefore, by rescaling, we can make it bounded by $1$. Hence by definition of $\mathscr{M}_4$ (recall Section~\ref{sec_hardy_and_grand_max})
\begin{align*}
\sup_{0\le l\le k,a\in [1/2,1]} | {\eta}_{l}^{(a)}\ast \zeta_{k}  \ast H(x)|
\le
\sup_{n\leq 0, a\in [1/2,1]}|\mathcal{M}_1^*(H;\Psi^{n,a} )(x)|
\lec \mathscr{M}_4 (H)(x),
\end{align*}
as required.
\end{proof}

\section*{Acknowledgements}
H. Kwon has been supported by the National
Science Foundation under Grant No. DMS-1638352.
W. S. O\.za\'nski has been supported by the funding from Charles Simonyi Endowment at the Institute for Advanced Study as well as the AMS Simons Travel Grant. The authors are grateful to Camillo De Lellis for suggesting this project and many helpful discussions.

\bibliographystyle{alpha}

\newcommand{\etalchar}[1]{$^{#1}$}

\end{document}